\documentclass[11pt]{article}
\usepackage[utf8]{inputenc}
\usepackage[top=3cm,left=3cm,right=3cm, bottom=3cm]{geometry}
\usepackage{amsmath,amsthm,amssymb}
\usepackage{hyperref}
\usepackage[sorting=nty]{biblatex}
\usepackage{cases}
\usepackage{url}
\makeatletter
\newcommand{\address}[1]{\gdef\@address{#1}}
\newcommand{\email}[1]{\gdef\@email{\url{#1}}}
\newcommand{\@endstuff}{\par\vspace{\baselineskip}\noindent\small
\begin{tabular}{@{}l}\scshape\@address\\\textit{E-mail address:} \@email\end{tabular}}
\AtEndDocument{\@endstuff}
\makeatother

\newcommand{\rvline}{\hspace*{-\arraycolsep}\vline\hspace*{-\arraycolsep}}

\newtheorem{theorem}{Theorem}[section]
\newtheorem{lemma}[theorem]{Lemma}
\newtheorem{proposition}[theorem]{Proposition}

\theoremstyle{definition}
\newtheorem{example}[theorem]{Example}

\theoremstyle{definition}
\newtheorem{definition}[theorem]{Definition}

\theoremstyle{definition}
\newtheorem{remark}[theorem]{Remark}

\DeclareMathOperator{\diver}{div}
\DeclareMathOperator{\ric}{Ric}
\DeclareMathOperator{\tr}{tr}
\DeclareMathOperator{\diag}{diag}

\DeclareMathOperator{\generated}{span}
\DeclareMathOperator{\spec}{Sp}
\DeclareMathOperator{\re}{Re}
\DeclareMathOperator{\img}{Im}

\title{Asymptotics of Solutions to Silent Wave Equations}
\author{Andrés Franco Grisales}
\date{}
\address{Department of Mathematics, KTH, 100 44 Stockholm, Sweden}
\email{anfg@kth.se}

\begin{document}

\maketitle

\begin{abstract}
    We study the asymptotics of solutions to a particular class of systems of linear wave equations, namely, of silent equations. Here, the asymptotics refer to the behavior of the solutions near a cosmological singularity, or near infinity in the expanding direction. Leading order asymptotics for solutions of silent equations were already obtained by Ringström in \cite{hans}. Here we improve upon Ringström's result, by obtaining asymptotic estimates of all orders for the solutions, and showing that solutions are uniquely determined by the asymptotic data contained in the estimates. 
    
    As an application, we then study solutions to the source free Maxwell's equations in Kasner spacetimes near the initial singularity. Our results allow us to obtain an asymptotic expansion for the potential of the electromagnetic field, and to show that the energy density of generic solutions blows up along generic timelike geodesics when approaching the singularity. The asymptotics we study correspond to the heuristics of the BKL conjecture, where the coefficients of the spatial derivative terms of the equations are expected to be small, and thus these terms are neglected in order to obtain the asymptotics.
\end{abstract}

\section{Introduction}

In the present work, we are concerned with the study of solutions to systems of linear wave equations in cosmological spacetimes. Specifically, the asymptotic behavior of such solutions. By cosmological we mean spacetimes with compact Cauchy surfaces. By asymptotics, we mean the behavior of the solutions as they approach a big bang/big crunch type singularity, or as they approach infinity in the expanding direction. 

In \cite{hans}, Ringström studies a specific class of wave equations with the objective of deriving optimal energy estimates and leading order asymptotics for the solutions. Here we shall focus on what Ringström calls \emph{silent equations}. Intuitively, these can be thought of as equations such that for the associated geometry, observers lose the ability to communicate at late times. In terms of the equations themselves, these are such that the coefficients of the spatial derivative terms decay exponentially, after rewriting the equations in terms of an appropriate time coordinate. This suggests that solutions can be approximated by dropping the spatial derivative terms from the equations. This results in a system of ODEs for each point in space, which is a greatly simplified situation. This kind of behavior is sometimes referred to as \emph{asymptotically velocity term dominated} (AVTD) in the literature.

In \cite{hans}, Ringström proved leading order asymptotic estimates for solutions of silent equations, which confirms that this heuristic picture is correct in this setting. We show that his results concerning the asymptotics can be substantially improved, at least upon allowing for further loss of regularity in the estimates. In what follows, we derive asymptotic estimates of all orders for solutions to silent equations; moreover, we show that solutions are uniquely determined by the asymptotic data in these estimates. As an application of our results, we then analyze the asymptotic behavior of solutions to the source free Maxwell's equations on Kasner spacetimes near the initial singularity.  

In the context of cosmology, such results are of great interest because of the BKL conjecture, due to Belinskiĭ, Khalatnikov and Lifschitz. This conjecture attempts to describe the generic behavior of solutions to Einstein's equations which present a spacelike cosmological singularity; see \cite{KL,bkl1,bkl2,BK-electromagnetic,BKeffect}, and \cite{heinzle-uggla-rohr} for a recent improvement. The conjecture states that near the initial singularity, observers are expected to lose the ability to communicate; that is, there is silence asymptotically. Furthermore, it proposes that the spatial derivative terms in Einstein's equations become negligible, thus it should be possible to approximate solutions by dropping these terms from the equations. Even though we only consider linear equations, our results present a step forward in this direction. Regarding the importance of the Kasner spacetimes, the conjecture also states that the time evolution of the spacetime, according to Einstein's equations, is chaotic, and described by a sequence of generalized Kasner spacetimes; see \cite{bkl1}. This motivates the use of Kasner spacetimes for the application of our results.

Another important problem in general relativity is the Strong Cosmic Censorship conjecture, which states that for generic initial data \footnote{The necessity of using generic initial data in the statement of the conjecture, arises because of the existence of maximal globally hyperbolic developments which admit inequivalent extensions; see \cite{cauchyproblem}.} for Einstein's equations, the maximal globally hyperbolic development is inextendible; see \cite{chrusciel}, and \cite{penrose} for the original version. By the Hawking singularity theorem (see Theorems 14.55A and 14.55B in \cite{oneill}), cosmological solutions to Einstein's equations typically present causal geodesic incompleteness. One method that has been used prove strong cosmic censorship for particular classes of cosmological solutions, is to prove that curvature invariants, such as the Kretschmann scalar $R^{\alpha \beta \gamma\delta} R_{\alpha \beta \gamma \delta}$, blow up along the incomplete causal geodesics. See for instance \cite{hans-scc, chrusciel-isenberg-moncrief, radermacher}. Note that such a statement ensures $C^2$ inextendibility of the spacetime. Motivated by this approach, in the application of our results to Maxwell's equations in Kasner spacetimes, we show that the energy density of generic solutions blows up along generic past directed timelike geodesics as they approach the initial singularity.   

There are several related results in the literature. As already mentioned, there are the results of \cite{hans}, which are our main reference. In \cite{hans2021wave}, Ringström studies asymptotics of systems of linear wave equations in silent big bang backgrounds which are not necessarily spatially homogeneous. In \cite{hans2019}, he studies solutions to the Klein-Gordon equation on Bianchi backgrounds with silent singularities. In \cite{alho}, Ahlo, Fournodavlos and Franzen study the asymptotics of the scalar wave equation near the initial singularity of vacuum FLRW spacetimes and Kasner spacetimes with $\mathbb{T}^3$ spatial topology. The scalar wave equation in Kasner spacetimes is also studied by Petersen in \cite{oliver}.

Equations where it is possible to neglect spatial derivative terms to obtain asymptotics come up, for instance, in the work of Fournodavlos, Rodnianski and Speck in \cite{fournodavlos-rodnianski-speck}, where they prove stable big bang formation around generalized Kasner solutions; in particular, they show that (nonlinear) Einstein's equations for the perturbed solutions present said behavior. In the context of Einstein's equations, this picture is also observed in the work of Andersson and Rendall in \cite{andersson-rendall}, where they study quiescent cosmological singularities in the analytic setting; and more recently in \cite{quiescent}, where Oude Groeniger, Petersen and Ringström study quiescent big bang singularities without the analyticity assumption. There is also the result by Fournodavlos and Luk, where they prove existence of solutions to the Einstein vacuum equations given initial data on the singularity, which also present said behavior; see \cite{fournodavlos-luk}. For further results where solutions are uniquely determined through asymptotic data see \cite{allen-rendall, andersson-rendall}.

\subsection{Silent equations} \label{geometric conditions}

In what follows, we work within the setting introduced in \cite{hans}. For convenience, we use the same notation and terminology as the one used therein. We consider systems of linear wave equations of the form 
\begin{equation} \label{pde}
    -g^{\mu \nu}(t) \partial_{\mu} \partial_{\nu} u - \sum_{r = 1}^R a_r^{-2}(t) \Delta_{g_r} u + \alpha(t) \partial_t u + X^j(t) \partial_j u + \zeta(t) u = f.
\end{equation}
Here $g^{\mu \nu}$, $a_r \in C^{\infty}(I, \mathbb{R})$ for $\mu, \nu = 0, \ldots, d$ and $r = 1, \ldots, R$ with $0 \leq d, R \in \mathbb{Z}$, where $I = (t_-,t_+)$ is an open interval. Also, $-g^{00}$ and $a_r$ take values in $(0, \infty)$ and for each $t \in I$, the $g^{jl}$ for $j,l = 1, \ldots, d$ are the components of a positive definite matrix. $(M_r, g_r)$ are closed Riemannian manifolds and $\Delta_{g_r}$ is the Laplace-Beltrami operator on $(M_r, g_r)$. Furthermore,  $u, f \in C^{\infty}(M, \mathbb{C}^m)$, where $1 \leq m \in \mathbb{Z}$; and $M := I \times \bar{M}$, where
\begin{equation} \label{manifold}
    \bar{M} := \mathbb{T}^d \times M_1 \times \cdots \times M_R.
\end{equation}
The $\partial_j$ are the standard coordinate vector fields in $\mathbb{T}^d$; $\partial_0 = \partial_t$ is the $t$ coordinate vector field in $M$; and $\alpha, X^j, \zeta \in C^{\infty}(I, \mathbb{M}_m(\mathbb{C}))$. Note that, if $Y$ denotes a smooth $m \times m$ matrix of vector fields on $M$ and $\xi$ a smooth $m \times m$ matrix valued function on $M$, then \eqref{pde} is of the form
\begin{equation} \label{tensor equation}
    \Box_g u + Y u + \xi u = f
\end{equation}
on $(M,g)$, where the Lorentzian metric $g$ is given by
\[
g = g_{\mu \nu} dx^{\mu} \otimes dx^{\nu} + \sum_r a_r^2 g_r.
\]
Here $x^0 = t$ and $(g_{\mu \nu}) = (g^{\mu \nu})^{-1}$. Moreover, $(M,g)$ is globally hyperbolic and the hypersurfaces
\[
\bar{M}_t := \{t\} \times \bar{M}
\]
for $t \in I$ are Cauchy hypersurfaces. See \cite[Chapter 25]{hans} for a justification of all these statements about $g$. The induced metric on $\bar{M}_t$ is denoted by $\bar{g}_t$ and the second fundamental form $\bar{k}_t$ is given by
\[
\bar{k}_t(X,Y) = g(\nabla_X U, Y),
\]
where $U$ denotes the future pointing unit normal of $\bar{M}_t$ and $\nabla$ is the Levi-Civita connection of $g$. We will often omit the $t$ subscript and just write $\bar{g}$ and $\bar{k}$ for the induced metric and second fundamental form respectively, when obvious from the context. For equations of the form \eqref{tensor equation}, smooth initial data for $u$ and its normal derivative on a Cauchy hypersurface give rise to smooth solutions in the globally hyperbolic spacetime $(M,g)$; for details, see \cite[Theorem~12.19]{cauchyproblem}.

The assumption that the coefficients of \eqref{pde} only depend on $t$ ensures that the equation is separable. Since we are only interested in the cosmological setting, assuming spatial homogeneity is standard. Therefore, this assumption on the coefficients of the equation is natural.

An equation of the form \eqref{pde} can be modified in two ways which result in an equation of the same type. One way is changing the time coordinate and the other is multiplying the equation by a positive smooth function of time. Multiplying by a function of time corresponds to a conformal rescaling of the associated metric $g$. We can use these two facts to ``normalize" the geometry associated with \eqref{pde}.

A metric like $g$ can always be written in the form 
\[
g = -N^2 dt \otimes dt + g_{ij} (\chi^i dt + dx^i) \otimes (\chi^j dt + dx^j) + \sum_r a_r^2 g_r,
\]
where $N$ is called the \emph{lapse function} and $\chi := \chi^i \partial_i$ is called the \emph{shift vector field}. In the cosmological setting, it is natural to expect the spacetime to be either expanding or contracting. Specifically, let $V(t)$ denote the volume of $\bar{M}_t$ with respect to $\bar{g}_t$. Then we expect that there is some $t_0 \in I$, such that $\partial_t V(t)$ is either positive or negative for $t \geq t_0$, so that $V(t)$ is strictly increasing or decreasing for $t \geq t_0$. If that is the case, we can introduce a new time coordinate
\[
\tau(t) := \pm \ln \frac{V(t)}{V(t_0)},
\]
where we choose the plus sign if $V(t)$ is increasing and the minus sign otherwise. Note that, if we assume $V(t) \to \infty$ or $V(t) \to 0$ as $t \to t_+$, the interval $[t_0,t_+)$ in $t$ time corresponds to $[0,\infty)$ in $\tau$ time. Furthermore,
\[
U = \frac{1}{N}(\partial_t - \chi), \quad U(\ln V) = \tr_{\bar{g}} \bar{k};
\]
cf. \cite[(25.19), p. 471]{hans}. Hence, the mean curvature $\tr_{\bar{g}} \bar{k}$ is strictly positive or negative for $t \geq t_0$ and we can consider the conformally rescaled metric $\hat{g} := (\tr_{\bar{g}} \bar{k})^2 g$ on the interval $[0,\infty)$ in $\tau$ time. $\hat{g}$ can then be written as
\[
\hat{g} = -d\tau \otimes d\tau + \hat{g}_{ij}(\hat{\chi}^i d\tau + dx^i) \otimes (\hat{\chi}^j d\tau + dx^j) + \sum_r \hat{a}_r^2 g_r;
\]
see Lemma~1.14 in \cite{hans}. Motivated by these remarks, we introduce the following definition; cf. Definition~1.18 in \cite{hans}.

\begin{definition} \label{canonical}
A \emph{canonical separable cosmological model manifold} is a Lorentzian manifold $(M,g)$ such that $M = I \times \bar{M}$, where $0 \in I = (t_{-}, \infty)$, $\bar{M}$ is given by \eqref{manifold} and
\[
g = - dt \otimes dt + g_{ij} (\chi^i dt + dx^i) \otimes (\chi^j dt + dx^j) + \sum_r a_r^2 g_r.
\]
\end{definition}

The fact that $N = 1$ implies that $g^{00} = -1$, see \cite[Lemma 25.2]{hans} for details. From now on, we shall always assume that $(M,g)$ is a canonical separable cosmological model manifold; that is, we think of the metric $g$ as resulting from a change of time coordinate and a conformal rescaling as we just described.

\begin{example}
    The Kasner spacetimes are of the form $((0, \infty) \times \mathbb{T}^n,g)$, where 
    \[
    g = -dt \otimes dt + \sum_i t^{2p_i} dx^i \otimes dx^i,
    \]
    and the $p_i$ are constants. We assume that they are vacuum, which is equivalent to 
    \[
    \sum_i p_i = \sum_i p_i^2 = 1.
    \]
    These are usually called the \emph{Kasner relations}. The Kasner spacetimes are canonical separable cosmological model manifolds. 

    Another example is de Sitter spacetime with cosmological constant $\Lambda > 0$, which is the spacetime $(\mathbb{R} \times S^3,g)$ where
    \[
    g = -dt \otimes dt + \cosh^2(Ht) g_{S^3_H},
    \]
    $H = (\Lambda/3)^{1/2}$, and $g_{S^3_H}$ is the standard metric on the 3-sphere $S^3$ with radius such that $\ric(g_{S^3_H}) = 2H^2g_{S^3_H}$.

    We remark that, even though these examples are already in the form of a canonical separable cosmological model manifold, we expect that in order for our results to be applicable it will be necessary to make a change of time coordinate and a conformal rescaling as described above. This transformation of the metric will of course result again in a canonical separable cosmological model manifold.
\end{example}

We are interested in equations for which the coefficients of the spatial derivative terms decay exponentially, after making a change of time coordinate as described above. Even though this could be directly imposed on the relevant coefficients, it is preferable to formulate geometric conditions, and then to deduce the exponential decay from the geometric conditions. For that purpose, we introduce the following definition.

\begin{definition} \label{c1 silent}
    Consider \eqref{pde} and assume that $(M,g)$ is a canonical separable cosmological model manifold. Suppose there is a $0 < \eta_{\text{sh}, 0} \leq 1$ such that 
    \begin{equation} \label{uniformly timelike}
        g(\partial_t, \partial_t) \leq -\eta_{\text{sh}, 0}^2
    \end{equation}
    for $t \geq 0$, and a $0 < \mu \in \mathbb{R}$ and a continuous non-negative function $\mathfrak{e} \in L^1([0,\infty))$ such that
    \begin{equation} 
        \bar{k} \geq (\mu - \mathfrak{e}) \bar{g}, \quad |\chi|_{\bar{g}} |\dot{\chi}|_{\bar{g}} \leq \mathfrak{e}
    \end{equation}
    for $t \geq 0$, where $\dot{\chi} = \mathcal{L}_U \chi$ is the Lie derivative of $\chi$ in the direction of $U$ \footnote{Since $N = 1$, then $\dot{\chi}$ is also tangent to $\bar{M}_t$, so that we can apply $\bar{g}_t$ to it.}. Then \eqref{pde} is said to be \emph{$C^1$-silent}.
\end{definition}

The motivation for the use of the word silent is that, the lower bound on $\bar{k}$ implies that observers lose the ability to communicate at late times. For a justification, see \cite[Lemma~25.23, Remark~25.24]{hans}. If \eqref{pde} is $C^1$-silent, then the coefficients of the spatial derivative terms, except the $X^j$, decay exponentially (see Lemma~\ref{silent lemma} below). Now we move on to the conditions on the remaining terms. Let us introduce some notation. Let
\[
\bar{h} := g_{ij} dx^i \otimes dx^j.
\]
Then the $\bar{h}(t)$ comprise a family of metrics on $\mathbb{T}^d$. Also let
\[
|\mathcal{X}|_{\bar{h}} := \left( \sum_{\varsigma \in \mathfrak{S}} \sum_{i,j} \bar{h}_{ij} \varsigma_i \|X^i\| \varsigma_j \|X^j\| \right)^{1/2},
\]
where $\mathfrak{S}$ is the set of $\varsigma \in \mathbb{R}^d$ with components $1$ or $-1$ and $\mathcal{X} := X^j \partial_j$.

\begin{definition} \label{c0 future bounded}
    If there is a constant $0 < C_0 \in \mathbb{R}$ such that
    \[
    |\mathcal{X}|_{\bar{h}} \leq C_0
    \]
    for $t \geq 0$, then $\mathcal{X}$ is said to be \emph{$C^0$-future bounded}.
\end{definition}

This condition is required to prevent solutions to \eqref{pde} from growing super exponentially; for details, see \cite[Chapter~2]{hans}. In particular, if \eqref{pde} is also $C^1$-silent, then the $X^j$ decay exponentially. Finally, regarding $\alpha$ and $\zeta$, we require that they converge to some constant matrices.

\begin{definition} \label{weak convergence}
    If there exist $\alpha_{\infty}, \zeta_{\infty} \in \mathbb{M}_m(\mathbb{C})$ and constants $\eta_{\text{mn}} > 0$ and $C_{\text{mn}}$ such that
    \[
    \|\alpha(t) - \alpha_{\infty}\| + \|\zeta(t) - \zeta_{\infty}\| \leq C_{\text{mn}} e^{-\eta_{\text{mn}} t}
    \]
    for $t \geq 0$, then Equation \eqref{pde} is said to be \emph{weakly convergent} \footnote{The reason for using the word weakly in this definition, is that in \cite{hans} similar conditions are also introduced, which in addition present bounds on the first derivatives of the objects of interest.}.
\end{definition}

Let \eqref{pde} be $C^1$-silent and weakly convergent, and assume $\mathcal{X}$ is $C^0$-future bounded. Then we expect to be able to approximate solutions to \eqref{pde} with solutions to
\[
u_{tt} + \alpha_{\infty} u_t + \zeta_{\infty} u = f,
\]
or equivalently
\begin{equation} \label{model equation 2}
\partial_t
\begin{pmatrix}
    u\\
    u_t
\end{pmatrix} = A_{\infty}
\begin{pmatrix}
    u\\
    u_t
\end{pmatrix} +
\begin{pmatrix}
    0\\
    f
\end{pmatrix}
\end{equation}
where
\[
A_{\infty} := 
\begin{pmatrix}
0 & I_m\\
-\zeta_{\infty} & -\alpha_{\infty}
\end{pmatrix}.
\]
The decay rate of the coefficients of \eqref{pde} that are dropped to get to \eqref{model equation 2} is given by 
\[
\beta_{\text{rem}} := \min\{ 
\mu, \eta_{\text{mn}} \};
\]
see Lemma~\ref{silent lemma} below and the comments that follow. If we assume for a moment that $f = 0$, then smooth solutions to \eqref{model equation 2} are of the form $e^{A_{\infty}t}V$ where $V \in C^{\infty}(\bar{M}, \mathbb{C}^{2m})$. In order to describe the asymptotics, we need a way to relate $\beta_{\text{rem}}$ with the growth or decay rate of $e^{A_{\infty}t}V$. This is because we need a way to determine when the latter is asymptotically larger that the error terms. The way to accomplish this is by referring to the generalized eigenspaces of $A_{\infty}$. This motivates the following definition.  

\begin{definition} \label{eigenspaces}
    If $\lambda$ is an eigenvalue of $A \in \mathbb{M}_k(\mathbb{C})$, the generalized eigenspace of $A$ associated to $\lambda$ is 
    \[
    E_{\lambda} = \ker(A - \lambda I)^{k_{\lambda}}
    \]
    where $k_{\lambda}$ is the algebraic multiplicity of $\lambda$. Let $\spec A$ be the set of eigenvalues of $A$ and define
    \[
    \kappa_{\text{max}}(A) := \sup\{ \re \lambda \, | \, \lambda \in \spec A \}, \quad \kappa_{\text{min}}(A) := \inf\{ \re \lambda \, | \, \lambda \in \spec A \}.
    \]
    Given $0 < \beta \in \mathbb{R}$, let $\mathcal{N}$ be the smallest non-negative integer such that
    \[
    \kappa_{\text{max}}(A) - \kappa_{\text{min}}(A) < \mathcal{N}\beta.
    \]
    For every positive integer $n \leq \mathcal{N}$, define $E^n_{A,\beta}$ as the direct sum of the generalized eigenspaces $E_{\lambda}$ such that 
    \[
    \text{Re} \lambda \in (\kappa_1 - n\beta, \kappa_1 - (n-1)\beta], 
    \]
    where $\kappa_1 := \kappa_{\text{max}}(A)$. Note that $\mathbb{C}^k = E_{A,\beta}^1 \oplus \cdots \oplus E_{A,\beta}^{\mathcal{N}}$.
\end{definition}

Finally, regarding the function $f$, we assume that for every $s \in \mathbb{R}$, 
    \begin{equation} \label{f condition 2}
        \|f\|_{A, s} := \int_0^{\infty} e^{-\kappa_1 \tau} \|f(\tau)\|_{(s)} d\tau < \infty.
    \end{equation}
Here $\|\cdot\|_{(s)}$ denotes the Sobolev norms for functions on $\bar{M}$; see \eqref{sobolev norm} below for details. This assumption is made to prevent $f$ from dominating the asymptotics.

\subsection{Results} \label{results}

Before stating the results, define $\mathcal{N}$ and $\kappa_1$ as in Definition~\ref{eigenspaces} with respect to $A = A_{\infty}$ and $\beta = \beta_{\text{rem}}$. For $1 \leq n \leq \mathcal{N}$, let $E^n = E_{A_{\infty,\beta_{\text{rem}}}}^n$. Also define
\[
D := 
\begin{pmatrix}
0 & \, 0\\
g^{jl}\partial_j \partial_l + \sum_r a_r^{-2} \Delta_{g_r} - X^j\partial_j + \zeta_{\infty} - \zeta & \, 2g^{0l} \partial_l + \alpha_{\infty} - \alpha
\end{pmatrix}.
\]
Note that the differential operator $D$ corresponds to the terms in \eqref{pde} that are dropped to arrive at \eqref{model equation 2}. For $\xi \in \mathbb{C}^k$, we use the notation $\langle \xi \rangle := (1 + |\xi|^2)^{1/2}$. Now we can state our main results regarding the asymptotics of solutions to silent equations.

\begin{theorem} \label{pde asymptotic theorem}
    Let \eqref{pde} be $C^1$-silent and weakly convergent, let $\mathcal{X}$ be $C^0$-future bounded and assume that \eqref{f condition 2} holds. Then for every integer $n \geq 1$, there are non-negative constants $C$, $N$, $s_{\emph{hom}, n}$ and $s_{\emph{ih}, n}$, depending only on the coefficients of the equation, such that the following holds. If $u$ is a smooth solution to \eqref{pde}, there is a unique $V_{\infty,n} \in C^{\infty}(\bar{M},E^1 \oplus \cdots \oplus E^n)$ (if $n > \mathcal{N}$, then $V_{\infty,n} \in C^{\infty}(\bar{M},\mathbb{C}^{2m})$) such that
    \begin{equation} \label{pde estimate 2}
    \begin{split}
        \bigg\| 
        &\begin{pmatrix}
            u(t)\\
            u_t(t)
        \end{pmatrix} - F_{\infty, n}(t) \bigg\|_{(s)}\\
        &\leq C \langle t \rangle^N e^{(\kappa_1 - n\beta_{\emph{rem}})t} (\|u_t(0)\|_{(s + s_{\emph{hom}, n})} + \|u(0)\|_{(s + s_{\emph{hom}, n} + 1)} + \|f\|_{A, s + s_{\emph{ih}, n}})
    \end{split}
    \end{equation}
    for $t \geq 0$, where $F_{\infty,n}$ is given by the recursive formula
    \[
    F_{\infty,n}(t) = e^{A_{\infty}t} V_{\infty,n} + \int_0^t e^{A_{\infty} (t-\tau)} D F_{\infty,n-1}(\tau) d\tau + \int_0^t e^{A_{\infty}(t-\tau)}
        \begin{pmatrix}
            0\\
            f(\tau)
        \end{pmatrix}d\tau  
    \]
    and $F_{\infty,0} = 0$. Furthermore, 
    \begin{equation} \label{inverse continuity}
        \|V_{\infty,n}\|_{(s)} \leq C(\|u_t(0)\|_{(s + s_{\emph{hom}, n})} + \|u(0)\|_{(s + s_{\emph{hom}, n} + 1)} + \|f\|_{A, s + s_{\emph{ih}, n}}).
    \end{equation}
\end{theorem}

\begin{remark} \label{2nd component derivative of first}
    Note that
    \[
    e^{A_{\infty}t} V_{\infty,n} + \int_0^t e^{A_{\infty}(t-\tau)}
        \begin{pmatrix}
            0\\
            f(\tau)
        \end{pmatrix}d\tau 
    \]
    is a solution to \eqref{model equation 2}. Moreover, note that $F_{\infty,n}$ is a solution to
    \[
    \partial_t
    \begin{pmatrix}
        u\\
        v
    \end{pmatrix} =
    A_{\infty}
    \begin{pmatrix}
        u\\
        v
    \end{pmatrix} + DF_{\infty, n-1} + 
    \begin{pmatrix}
        0\\
        f
    \end{pmatrix},
    \]
    where $u$ and $v$ take values in $\mathbb{C}^m$. In particular, because of the form of $A_{\infty}$ and $D$, the first $m$ equations of the system above reduce to $\partial_t u = v$. Thus the last $m$ components of $F_{\infty,n}$ are the $\partial_t$ derivatives of the first $m$ components.
\end{remark}

\begin{remark}
    More information on the dependence of the constants $s_{\text{hom},n}$, $s_{\text{ih},n}$ and $N$ appearing in the estimates can be found in Remark~\ref{dependence of the constants} below.
\end{remark}

This result is a direct consequence of Theorem~\ref{pde asymptotic lemma} and Lemma~\ref{silent lemma} below. The results of \cite{hans} regarding the asymptotics, correspond to the case $n = 1$ in Theorem~\ref{pde asymptotic theorem}. The improvement provided by Theorem~\ref{pde asymptotic theorem} can be viewed as allowing one to construct a full asymptotic expansion for the solution. The advantage of this improvement becomes most apparent when considering actual systems of equations, as opposed to individual equations. In that case, the results of \cite{hans} might provide explicit asymptotic information about only one of the components of the solution, while only yielding estimates for the rest. By using Theorem~\ref{pde asymptotic theorem} with $n$ large enough, it should be possible to obtain detailed asymptotics for all the components of the solution. This aspect will become clear in our study of Maxwell's equations in Section~\ref{maxwells equations}.

We are also interested in knowing whether it is possible to specify the asymptotic data contained in the above estimates. This data is contained in the functions $V_{\infty,n}$ for $1 \leq n \leq \mathcal{N}$. For that purpose, we introduce the following definition.

\begin{definition}
    Let \eqref{pde} be $C^1$-silent and weakly convergent, let $\mathcal{X}$ be $C^0$-future bounded and assume that \eqref{f condition 2} holds. Let $u$ be a solution to \eqref{pde}. Define $\mathcal{V}_{\infty} \in C^{\infty}(\bar{M}, \mathbb{C}^{2m})$ by
    \[
    \mathcal{V}_{\infty} := V_{\infty,1} + \pi_2(V_{\infty,2}) + \cdots + \pi_{\mathcal{N}}(V_{\infty,\,\mathcal{N}})
    \]
    where $\pi_n$ is the projection onto $E^n$. Then $\mathcal{V}_{\infty}$ is called \emph{asymptotic data} for $u$.
\end{definition}

The following result ensures that solutions to \eqref{pde} are uniquely determined by the asymptotic data. It turns out to be enough to prove the statement in the case $f = 0$. To obtain it, it is necessary to also impose an upper bound on $\bar{k}$. This is because we need to control the growth of the energy when going backwards in time. 

\begin{theorem} \label{homeomorphism theorem}
    Let \eqref{pde} be $C^1$-silent and weakly convergent, and assume $\mathcal{X}$ to be $C^0$-future bounded. Suppose that $f = 0$ and that there is a constant $0 < \mu_+ \in \mathbb{R}$ and a non-negative continuous function $\mathfrak{e}_+ \in L^1([0,\infty))$ such that 
    \[
    \bar{k} \leq (\mu_+ + \mathfrak{e}_+) \bar{g}
    \]
    for $t \geq 0$. Then there is a linear isomorphism
    \[
    \Phi_{\infty}:C^{\infty}(\bar{M},\mathbb{C}^{2m}) \to C^{\infty}(\bar{M},\mathbb{C}^{2m}) 
    \]
    such that, for $\mathcal{V}_{\infty} \in C^{\infty}(\bar{M},\mathbb{C}^{2m})$, 
    \[
    \|\Phi_{\infty}(\mathcal{V}_{\infty})\|_{(s)} \leq C\|\mathcal{V}_{\infty}\|_{(s + \xi_{\infty})},
    \]
    where $C$ and $\xi_{\infty}$ only depend on the coefficients of the equation. Furthermore, if $u$ is the solution of \eqref{pde} such that
    \[
    \begin{pmatrix}
        u(0)\\
        u_t(0)
    \end{pmatrix} = \Phi_{\infty}(\mathcal{V}_{\infty}),
    \]
    then $u$ has asymptotic data given by $\mathcal{V}_{\infty}$.    
\end{theorem}

This result is a direct consequence of Theorem~\ref{homeomorphism}, Lemma~\ref{silent lemma} and Remark~\ref{silent remark} below. Note that the estimates \eqref{inverse continuity} ensure that the inverse of $\Phi_{\infty}$ is continuous, so that $\Phi_{\infty}$ is a homeomorphism between asymptotic data and initial data. When comparing to the results of \cite{hans}, we remark that in \cite{hans} a homeomorphism between initial and asymptotic data is obtained only in some favourable cases. In general, the results there only ensure that given asymptotic data, there exists a corresponding solution. By contrast, we can always parametrize the solutions by using asymptotic data.

Now we move on to our application to Maxwell's equations. Here we only consider dimension 4. Recall that the Kasner spacetimes are of the form $(M,g)$, where $M = (0, \infty) \times \mathbb{T}^3$, 
\[
g = -dt \otimes dt + \sum_i t^{2p_i} dx^i \otimes dx^i,
\]
and the $p_i$ are constants. If we take $p_1 \leq p_2 \leq p_3$, the Kasner relations
\[
p_1 + p_2 + p_3 = p_1^2 + p_2^2 + p_3^2 = 1
\]
lead to two possibilities. Either $p_3 < 1$, or $p_3 = 1$. The latter forces $p_1 = p_2 = 0$ and the spacetime is actually flat. Here we focus on the non-flat Kasner spacetimes, so we assume $p_3 < 1$. In accordance with the comments made in connection to Definition~\ref{canonical}, in order to study the behavior as $t \to 0$, we are led to consider the time coordinate $\tau = -\ln t$, and the conformally rescaled metric
\[
\hat{g} = e^{2\tau} g = -d\tau \otimes d\tau + \sum_i e^{2(1-p_i)\tau} dx^i \otimes dx^i.
\]
If $\bar{g}$ and $\bar{k}$ are the metric and second fundamental form of the $\mathbb{T}^3_{\tau}$ induced by the metric $\hat{g}$, then
\[
\bar{k} = \sum_i (1-p_i) e^{2(1-p_i)\tau} dx^i \otimes dx^i \geq (1-p_3) \bar{g}.
\]
Since $\chi = 0$ in this case, by recalling the definition of $C^1$-silent, we see that $(M,\hat{g})$ is an appropriate background for the application of our results. 

The source free Maxwell's equations are
\[
\diver F = 0, \quad dF = 0,
\]
where $F \in \Omega^2(M)$ is the Faraday tensor. Since $dF = 0$, locally we can find a one form $\omega$ such that $F = d\omega$. Furthermore, we can choose $\omega$ in such a way that Maxwell's equations become 
\[
\Box_g \omega = 0, \quad \diver \omega = 0. 
\]
We can then use our results to analyze the asymptotics of $\omega$ as $t \to 0$. We obtain the following result regarding the energy of $F$. Here we only give a rough statement of the result. For the precise statement and proof see Section~\ref{proof of blow up} below; in particular, Theorem~\ref{generic blow up} and Remark~\ref{localisation argument} below.

\begin{theorem} \label{generic blow up of energy}
    There is an open and dense subset $\mathcal{A}$ of the space of initial data at $t = 1$ for the source free Maxwell's equations on a Kasner spacetime, in the $C^{\infty}$ topology, such that if $F$ is a solution with initial data in $\mathcal{A}$, the following holds. Define the stress-energy tensor $T$ associated with $F$ by
    \[
    T_{\alpha \beta} = \frac{1}{4\pi} \bigg( F_{\alpha \gamma} F_{\beta}^{\phantom{a} \gamma} - \frac{1}{4} g_{\alpha \beta} F_{\mu \nu} F^{\mu \nu} \bigg).
    \]
    Then for almost every timelike geodesic $\gamma$ (see Remark~\ref{almost every} below) the energy density $T(\dot{\gamma}, \dot{\gamma})$ along $\gamma$ blows up like $t^{-(2p_2 + 4p_3)}$ as $t \to 0$.
\end{theorem}

\begin{remark} \label{almost every}
    Since for a fixed $t_0$, $\{ t_0 \} \times \mathbb{T}^3$ is a Cauchy surface, we can identify the set of inextendible timelike geodesics in $M$ with the set $P$ of past directed unit timelike vectors with base point in $\{ t_0 \} \times \mathbb{T}^3$. Since $P$ is a manifold, there is a notion of sets of measure zero in $P$. Thus when we say ``almost every timelike geodesic" in the statement of Theorem~\ref{generic blow up of energy}, we mean that the set of inextendible timelike geodesics for which the statement does not hold, corresponds to a set of measure zero in $P$.
\end{remark}

\subsection{Outline}

We proceed as follows. In Section \ref{silent equations}, we introduce the analytic conditions on \eqref{pde} which will be used to prove the main results. The main technical observation about the equations, is that it is possible to perform a Fourier decomposition to obtain a sequence of ODEs for the modes of the solution. These ODEs can then be analyzed to obtain conclusions about the solution of the original equation. In Section~\ref{ode analysis}, we analyze the class of ODEs which arises from the Fourier decomposition, obtaining asymptotic estimates of all orders and a complete characterization of solutions in terms of the asymptotic data. In Section~\ref{pde analysis}, we use the results obtained in Section~\ref{ode analysis} to prove analogous results for silent equations. The main results are Theorem~\ref{pde asymptotic lemma}, which provides the asymptotic estimates, and Theorem~\ref{homeomorphism}, which proves that there is a homeomorphism between initial data and asymptotic data in the $C^{\infty}$ topology. Finally, in Section~\ref{maxwells equations}, we apply our results to Maxwell's equations in Kasner spacetimes. The main results of the analysis are Theorem~\ref{omega expansion theorem}, which gives an asymptotic expansion for the potential of the electromagnetic field; and Theorem~\ref{generic blow up}, which gives generic blow up of the energy density.

\section{Weakly silent, balanced and convergent equations} \label{silent equations}

In Section~\ref{geometric conditions}, we defined the class of equations of interest through the formulation of geometric conditions. In this section, we introduce the analytic conditions that are used in the proofs of the main results. Since the coefficients of \eqref{pde} only depend on time, the equation can be decomposed into a sequence of ODEs for the Fourier modes of the solution. We formulate the analytic conditions in terms of this decomposition. 

\subsection{Fourier decomposition of the equation}

Next, for the sake of the reader, we recall Sections 24.1.2 and 24.1.3 in \cite{hans}. Consider \eqref{pde} and let $\varphi_{r,i}$ be orthonormal eigenfunctions of $\Delta_{g_r}$ with corresponding eigenvalues $-\nu_{r,i}^2$. Define
\[
\iota := (n, i_1, \ldots, i_R),
\]
where $n \in \mathbb{Z}^d$, and denote the set of such $\iota$ by $\mathcal{J}_B$. Also define 
\[
\nu(\iota) := (n, \nu_{1,i_1}, \ldots, \nu_{R, i_R}).
\]
We think of $\mathbb{T}^d$ as $[0,2\pi]^d$ with the ends identified. Let $x \in \mathbb{T}^d$, $p_r \in M_r$ and $p = (x, p_1, \ldots, p_R)$. Define $\varphi_{\iota}$ by
\[
\varphi_{\iota}(p) := (2\pi)^{-d/2} e^{in \cdot x} \varphi_{1,i_1}(p_1) \cdots \varphi_{R,i_R}(p_R).
\]
Note that $\varphi_{\iota}$ is an eigenfunction of $\Delta_T + \Delta_1 + \cdots + \Delta_R$ with eigenvalue $-|\nu(\iota)|^2$, where $\Delta_T$ is the standard Laplacian on $\mathbb{T}^d$.

The volume form that we will use is given by
\begin{equation*} 
    \mu_B := dx \wedge \mu_{g_1} \wedge \cdots \wedge \mu_{g_R},
\end{equation*}
where $dx$ is the standard volume form in $\mathbb{T}^d$ and $\mu_r$ is the Riemannian volume form of $(M_r, g_r)$. Then we can think of $L^2(\bar{M})$ as a complex Hilbert space with inner product
\[
\langle u, v \rangle_B := \int_{\bar{M}} uv^* \mu_B,
\]
where $*$ denotes complex conjugation. Note that the $\varphi_{\iota}$ for $\iota \in \mathcal{J}_B$ form an orthonormal basis of $L^2(\bar{M})$, in particular 
\[
\int_{\bar{M}} |u|^2 \mu_B = \sum_{\iota \in \mathcal{J}_B} |\hat{u}(\iota)|^2
\]
for $u \in C^{\infty}(\bar{M}, \mathbb{C})$, where
\[
\hat{u}(\iota) := \langle u, \varphi_{\iota} \rangle_B.
\]
If $u \in C^{\infty}(\bar{M}, \mathbb{C})$ and $s \in \mathbb{R}$, define
\begin{equation} \label{sobolev norm}
    \|u\|_{(s)} := \bigg( \sum_{\iota \in \mathcal{J}_B} \langle \nu(\iota) \rangle^{2s} |\hat{u}(\iota)|^2 \bigg)^{1/2},
\end{equation}
where we use the notation $\langle \xi \rangle = (1 + |\xi|^2)^{1/2}$. For $s \in \mathbb{R}$, the Sobolev space $H^s(\bar{M})$ is defined as the set of complex valued distributions such that the value of $\| \cdot \|_{(s)}$ is finite. The space $H^s(\bar{M}, \mathbb{C}^n)$ is defined similarly; in this case, $\hat{u}(\iota)$ is defined componentwise.

Going back to Equation \eqref{pde}, assume the associated spacetime $(M,g)$ is a canonical separable cosmological model manifold. Multiplying \eqref{pde} by $\varphi_{\iota}$, with respect to $\langle \, \cdot \, , \, \cdot \, \rangle_B$, we obtain
\begin{equation} \label{mode equation}
    \Ddot{z} + \mathfrak{g}^2 z - 2in_l g^{0l} \dot{z} + \alpha \dot{z} + i n_l X^l z + \zeta z = \hat{f},
\end{equation}
where
\begin{align*}
    &\mathfrak{g}(t, \iota) := \left( g^{jl}(t) n_j n_l + \sum_r a_r^2(t) \nu_{r, i_r}^2 \right)^{1/2},\\
    &z(t, \iota) := \langle u(t), \varphi_{\iota} \rangle_B,\\
    &\hat{f}(t, \iota) := \langle f(t), \varphi_{\iota} \rangle_B.
\end{align*}
For $\iota \neq 0$, define
\begin{equation*}
    \sigma(t, \iota) := \frac{n_l g^{0l}(t)}{\mathfrak{g}(t, \iota)}, \quad X(t, \iota) := \frac{n_l X^l(t)}{\mathfrak{g}(t, \iota)};
\end{equation*}
then we can write \eqref{mode equation} as 
\begin{equation*}
    \Ddot{z} + \mathfrak{g}^2 z - 2i \sigma \mathfrak{g} \dot{z} + \alpha \dot{z} + iX \mathfrak{g} z + \zeta z = \hat{f}.
\end{equation*}

\subsection{Weak silence and weak balance}

In Section~\ref{geometric conditions}, we imposed geometric conditions on Equation~\eqref{pde}. Now we introduce weaker analytic conditions, which are the ones used in the proofs of the main results; cf. Definition~10.1 from \cite{hans}.

\begin{definition} \label{silent}
    Consider \eqref{pde} and assume $(M,g)$ is a canonical separable cosmological model manifold. If there is a constant $C_{\text{coeff}}$ such that
    \[
    |\sigma(t, \iota)| + \| X(t,\iota) \| + \|\alpha(t)\| + \|\zeta(t)\| \leq C_{\text{coeff}}
    \]
    for $\iota \neq 0$ and $t \geq 0$, then Equation \eqref{pde} is said to be \emph{weakly balanced}. Define
    \[
    \ell(t,\iota) := \ln( \mathfrak{g}(t,\iota) ).
    \]
    If there is a constant $\mathfrak{b}_{\text{s}} > 0$ and a continuous non-negative function $\mathfrak{e} \in L^1([0,\infty))$ such that 
    \[
    \dot{\ell}(t,\iota) \leq -\mathfrak{b}_{\text{s}} + \mathfrak{e}(t)
    \]
    for all $\iota \neq 0$ and $t \geq 0$, then Equation \eqref{pde} is said to be \emph{weakly silent}. Set $c_{\mathfrak{e}} := \|\mathfrak{e}\|_1$. If \eqref{pde} is weakly silent, we define $T_{\text{ode}}$ as follows. If $\mathfrak{g}(0,\iota) \leq e^{-c_{\mathfrak{e}}}$, then $T_{\text{ode}} := 0$. If $\mathfrak{g}(0,\iota) > e^{-c_{\mathfrak{e}}}$, then $T_{\text{ode}}$ is the first $t \geq 0$ such that $\mathfrak{g}(t,\iota) = e^{-c_{\mathfrak{e}}}$.
\end{definition}

Now we need to relate the conditions introduced in Section~\ref{geometric conditions} with the notions of weak silence and weak balance. For that purpose, we will make use of the following result.

\begin{lemma} \label{silent lemma}
    Consider \eqref{pde} and assume that $(M,g)$ is a canonical separable cosmological model manifold. If \eqref{pde} is $C^1$-silent, then it is weakly silent with $\mathfrak{b}_{\text{\emph{s}}} = \mu$. Additionally, if \eqref{pde} is weakly convergent and $\mathcal{X}$ is $C^0$-future bounded, then \eqref{pde} is also weakly balanced.
\end{lemma}

\begin{proof}
    See Lemmas 25.12 and 25.17 in \cite{hans}. Note that \eqref{uniformly timelike} implies that $|\chi|_{\bar{g}}$ is bounded.
\end{proof}

This lemma justifies the use of the word weakly in Definition \ref{silent} for weakly silent. Similarly to Definition~\ref{c0 future bounded}, the condition of weak balance is required to prevent solutions from growing super exponentially, see \cite[Chapter 2]{hans} for a detailed explanation. The reason for using the word weak in weak balance, is that in \cite{hans} similar conditions are also introduced, which also present bounds on the first derivatives of the objects of interest.  

Let \eqref{pde} be weakly silent, balanced and convergent. The idea for the analysis is to analyze the modes for the solution through Equation \eqref{mode equation} and then put the results together. For that purpose, let $z(\iota)$, for $\iota \in \mathcal{J}_B$, be a solution to \eqref{mode equation}. Then \eqref{mode equation} can be written as
\begin{equation} \label{ode mode equation}
    \dot{v} = A_{\infty}v + A_{\text{rem}}v + F,
\end{equation}
where
\[
v := 
\begin{pmatrix}
    z\\
    \dot{z}
\end{pmatrix}, \quad A_{\infty} := 
\begin{pmatrix}
    0 & I_m\\
    -\zeta_{\infty} & -\alpha_{\infty}
\end{pmatrix}, \quad F := 
\begin{pmatrix}
    0\\
    \hat{f}
\end{pmatrix},
\]
and
\[
A_{\text{rem}} := 
\begin{pmatrix}
    0 & \, 0\\
    - \mathfrak{g}^2 I_m - in_lX^l + \zeta_{\infty} - \zeta & \, 2in_lg^{0l} I_m + \alpha_{\infty} - \alpha
\end{pmatrix}.
\]
There are two perspectives that we could take from here. From the definition of weak silence, we have  
\[
    \mathfrak{g}(t) \leq e^{c_{\mathfrak{e}}} e^{-\mathfrak{b}_{\text{s}}(t - t_0)} \mathfrak{g}(t_0),
\]
for $t \geq t_0 \geq 0$. Hence, if we let $\bar{t} := t - T_{\text{ode}}$, we have
\[
\mathfrak{g}(t) \leq e^{-\mathfrak{b}_{\text{s}} \bar{t}}
\]
for $t \geq T_{\text{ode}}$. Weak balance thus yields
\[
|n_l g^{0l}(t)| + \|n_l X^l(t)\| \leq C_{\text{coeff}} e^{-\mathfrak{b}_{\text{s}} \bar{t}}
\]
for $t \geq T_{\text{ode}}$, and together with weak convergence, we obtain
\begin{equation} \label{arem estimate 1}
    \|A_{\text{rem}}(t)\| \leq C e^{-\beta_{\text{rem}} \bar{t}}
\end{equation}
for $t \geq T_{\text{ode}}$; where $\beta_{\text{rem}} := \min\{ \mathfrak{b}_{\text{s}}, \eta_{\text{mn}} \}$ and $C$ only depends on $C_{\text{coeff}}$ and $C_{\text{mn}}$. Alternatively,
\[
\mathfrak{g}(t) \leq e^{c_{\mathfrak{e}}} e^{-\mathfrak{b}_{\text{s}}t} \mathfrak{g}(0) \leq C \langle \nu(\iota) \rangle e^{-\mathfrak{b}_{\text{s}} t} 
\]
where $C$ only depends on $c_{\mathfrak{e}}, g^{jl}(0)$ and $a_r(0)$. This together with weak balance and weak convergence yields
\begin{equation} \label{arem estimate 2}
    \|A_{\text{rem}}(t)\| \leq C \langle \nu(\iota) \rangle^2 e^{-\beta_{\text{rem}}t}
\end{equation}
for $t \geq 0$, where $C$ only depends on $c_{\mathfrak{e}}$, $C_{\text{coeff}}$, $C_{\text{mn}}$, $g^{jl}(0)$ and $a_r(0)$. The difference between \eqref{arem estimate 1} and \eqref{arem estimate 2} is the way in which the dependence on $\iota$ is manifested. In what follows, we shall make use of both of these perspectives. In any case, $A_{\text{rem}}$ satisfies an exponential decay estimate. This is the motivation for the analysis that we shall develop in the following section.

In terms of Equation \eqref{pde}, note that the above discussion shows that the coefficients of the spatial derivative terms decay exponentially as $t \to \infty$. Since the coefficients of the time derivative terms converge exponentially, this suggests that we should be able to approximate solutions to \eqref{pde} with solutions to the equation 
\[
u_{tt} + \alpha_{\infty} u_t + \zeta_{\infty} u = f,
\]
or equivalently
\[
    \partial_t
    \begin{pmatrix}
        u\\
        u_t
    \end{pmatrix} = A_{\infty}
    \begin{pmatrix}
        u\\
        u_t
    \end{pmatrix} +
    \begin{pmatrix}
        0\\
        f
    \end{pmatrix},
\]
where $A_{\infty}$ is as in \eqref{ode mode equation}.

\section{ODE analysis} \label{ode analysis}

Motivated by the remarks made in the previous section, we consider the following class of first order linear ODEs,
\begin{equation} \label{ode}
    \dot{v}(t) = A v(t) + A_{\text{rem}}(t) v(t) + F(t),
\end{equation}
where $A \in \mathbb{M}_k (\mathbb{C})$, $A_{\text{rem}}: \mathbb{R} \to \mathbb{M}_k(\mathbb{\mathbb{C}})$ and $F: \mathbb{R} \to \mathbb{C}^k$ are smooth, and there are constants $C_{\text{rem}}$, $T_{\text{ode}} \geq 0$ and $\beta_{\text{rem}} > 0$ such that
\begin{equation} \label{arem estimate}
    \| A_{\text{rem}}(t) \| \leq C_{\text{rem}} e^{-\beta_{\text{rem}} \bar{t}} 
\end{equation}
for $t \geq T_{\text{ode}}$, where $\bar{t} = t - T_{\text{ode}}$. The basic idea is that we should be able to approximate a solution to \eqref{ode} by a solution of the equation
\begin{equation} \label{model ode equation}
    \dot{v} = A v + F.
\end{equation}
In general, it might happen that the asymptotics for a solution of \eqref{ode} are dominated by $F$. In that case, it could be necessary to make detailed assumptions on $F$ in order to get detailed asymptotics. To rule out this possibility, we assume that
\begin{equation} \label{inhomogeneous assumption}
    \|F\|_A := \int_0^{\infty} e^{-\kappa_1 s} |F(s)|ds < \infty
\end{equation}
where $\kappa_1 = \kappa_{\text{max}}(A)$; this ensures that $F$ is asymptotically smaller than the largest solutions of $\dot{v} = Av$. Once we have obtained the asymptotic estimates, it will be of interest to know whether we can prescribe the asymptotic data in these estimates. 

In \cite[Chapter 9]{hans}, leading order asymptotic estimates are obtained under the assumptions just described. Moreover, it is shown that leading order asymptotic data can be specified. In what follows we develop these ideas further, deriving asymptotic estimates of all orders, and obtaining a complete characterization of solutions in terms of the asymptotic data. Furthermore, this is done in such a way that the results can be carried over to the PDE setting of interest. Before getting started, we introduce the setup that is required to describe the asymptotics. 

\begin{definition} \label{generalized eigenspaces}
    Let $\mathcal{N}$ be the smallest non-negative integer such that 
    \[
    \text{Rsp}A := \kappa_{\text{max}}(A) - \kappa_{\text{min}}(A) < \mathcal{N}\beta_{\text{rem}}. 
    \]
    Then there is a $T \in \text{GL}_k(\mathbb{C})$ such that $A_J := T^{-1}AT$ satisfies the following,
    \[
    A_J = \diag(A_1, \ldots, A_\mathcal{N}),
    \]
    where $A_n \in \mathbb{M}_{k_n}(\mathbb{C})$, the $k_n$ are non-negative integers such that $k_1 + \cdots + k_{\mathcal{N}} = k$, and the $A_n$ consist of Jordan blocks. Furthermore, $\text{Rsp}(\diag(A_1, \ldots, A_n)) < n\beta_{\text{rem}}$, $\kappa_{\text{max}}(A_1) = \kappa_{\text{max}}(A)$, $\kappa_{\text{max}}(A_{n+1}) \leq \kappa_{\text{max}}(A) - n\beta_{\text{rem}}$ (assuming $k_n \neq 0$). 

    The matrix $T$ is obtained by transforming $A$ into Jordan normal form and arranging the Jordan blocks appropriately. The subspace 
    \[
    E_{A,\beta_{\text{rem}}}^n := T(\{ 0 \}^{k_1 + \cdots + k_{n-1}} \times \mathbb{C}^{k_n} \times \{0\}^{k_{n+1} + \cdots + k_{\mathcal{N}}})
    \]
    is called the \emph{$n$-th generalized eigenspace in the $\beta_{\text{\emph{rem}}}$, $A$-decomposition of $\mathbb{C}^k$}. Note that ${\mathbb{C}^k = E_{A,\beta_{\text{rem}}}^1 \oplus \cdots \oplus E_{A,\beta_{\text{rem}}}^{\mathcal{N}}}$; we denote by $\pi_n: \mathbb{C}^k \to E_{A,\beta_{\text{rem}}}^n$ the corresponding projection map. We shall also use the notation $x = (x_1, \ldots, x_{\mathcal{N}}) \in \mathbb{C}^k$ where $x_n \in \mathbb{C}^{k_n}$. 
\end{definition}

Note that this definition is consistent with Definition~\ref{eigenspaces} with $\beta = \beta_{\text{rem}}$. The matrix $\re A_J - \kappa_1 I_k$ is real, consists of Jordan blocks and has non-positive diagonal entries. It is convenient to ensure that the Jordan blocks with negative eigenvalues are negative definite. If $J$ is a $m \times m$ Jordan block with diagonal entries equal to $\lambda \in \mathbb{R}$ and $D = \diag(1, \varepsilon, \ldots, \varepsilon^{m-1})$ for $\varepsilon > 0$, note that $D^{-1} J D$ is the same as $J$ but with upper diagonal entries replaced by $\varepsilon$. A matrix of the form $D^{-1} J D$ is called a \emph{generalized Jordan block}. Note that if ${\lambda < 0}$, by taking $\varepsilon$ small enough, we can make $D^{-1} J D$ negative definite. Going back to $A_J$, taking this observations into account, we see that there is a diagonal $D \in \text{GL}_k(\mathbb{R})$ such that if ${T_A := TD}$, then $T_A^{-1} A T_A$ consists of generalized Jordan blocks. Furthermore, all the generalized Jordan blocks of $J_A := \re(T_A^{-1} A T_A) - \kappa_1 I_k$ corresponding to negative eigenvalues are negative definite. We denote by $J_n$ the block of $J_A$ corresponding to $A_n$ in $A_J$.

We remark that this negative definiteness property of some of the Jordan blocks of $J_A$ is actually not used in what follows. Nonetheless, the matrix $J_A$ is the one used for the proofs in \cite{hans}, and we will refer to one of the proofs of \cite{hans} in the proof of Lemma~\ref{specify ode data} below. We thus choose to use $J_A$ to maintain consistency with \cite{hans}.

Let $A_{J,i} := i \img A_J$, then $A_{J,i}$ is diagonal and imaginary. Define $R_A(t) := \exp(-A_{J,i} t)$ and note that $\|R_A(t)\| = \|R_A(t)^{-1}\| = 1$. Moreover, $R_A(t)$ commutes with $J_A$, since for each generalized Jordan block of $J_A$, the corresponding diagonal entries of $A_{J,i}$ are all equal. The matrix $e^{At}$ can be expressed in terms of the matrices just defined as follows. We have
\[
T_A^{-1} A T_A = A_{J,i} + J_A + \kappa_1 I_k,
\]
thus
\[
e^{At} = e^{\kappa_1 t} T_A R_A(t)^{-1} e^{J_A t} T_A^{-1}.
\]
The following estimate for the operator norm of a matrix exponential will be useful in what follows.

\begin{lemma} \label{matrix exponential estimate}
    Let $A \in \mathbb{M}_k(\mathbb{C})$. Then there is a constant $C$ depending only on $A$ such that
    \[
    \|e^{At}\| \leq C \langle t \rangle^{d_1-1} e^{\kappa_1 t}
    \]
    for $t \geq 0$, where $\kappa_1$ is the largest real part of an eigenvalue of $A$ and $d_1$ is the largest size of a Jordan block corresponding to an eigenvalue of $A$ with real part $\kappa_1$.
\end{lemma}

\begin{proof}
    Let $J_{\lambda,d}$ denote a Jordan block of size $d$ and eigenvalue $\lambda$, then
    \[
    e^{J_{\lambda,d}t} = e^{\lambda t} e^{J_{0,d}t} = e^{\lambda t} \sum_{k = 0}^{d-1} \frac{1}{k!} J_{0,d}^k t^k,
    \]
    where the sum is finite because $J_{0,d}$ is nilpotent. Now let $T \in \text{GL}_k(\mathbb{C})$ such that 
    \[
    J = T^{-1} A T = \diag(J_{\lambda_1,d_1}, \ldots, J_{\lambda_m,d_m}), 
    \]
    where $\re(\lambda_i) \geq \re(\lambda_{i+1})$ and $d_i \geq d_{i+1}$ whenever $\re(\lambda_i) = \re(\lambda_{i+1})$. Then
    \[
    e^{At} = T e^{Jt} T^{-1} = T \diag(e^{J_{\lambda_1,d_1}t}, \ldots, e^{J_{\lambda_m,d_m}t}) T^{-1}.
    \]
    The lemma follows.
\end{proof}

\subsection{Asymptotic estimates}

Consider Equation \eqref{ode}. For what follows, we modify the condition \eqref{arem estimate} by requiring that
\begin{equation} \label{error}
    \| A_{\text{rem}}(t) \| \leq C_{\text{rem}} \gamma e^{-\beta_{\text{rem}} \bar{t}} 
\end{equation}
for $t \geq T_{\text{ode}}$ where $\gamma > 0$ is some constant. The reason for adding the constant $\gamma$ to the hypothesis on $A_{\text{rem}}$ is that, when we apply the results obtained here by using the estimate \eqref{arem estimate 2}, it will be important to keep track of $\langle \nu(\iota) \rangle^2$ since it depends on $\iota$. The following result is based on the proof of \cite[Lemma 9.16]{hans}. The idea is that, given an asymptotic estimate for a solution of \eqref{ode}, the assumption \eqref{error} makes it possible to improve on it; and the new approximation is given in terms of a solution of \eqref{model ode equation} plus a term depending on the original approximation. 

\begin{lemma} \label{asymptotic lemma}
    Consider Equation \eqref{ode} where \eqref{error} holds. Suppose that there is an $n \geq 0$, a constant $K$, a non-negative integer $M$, a smooth function $H: [0,\infty) \to \mathbb{C}^k$ and a function $\alpha: [0,\infty) \to (0,\infty)$ such that
    \begin{equation} \label{asymptotic estimate}
        |v(t) - H(t)| \leq K \langle \bar{t} \rangle^M e^{(\kappa_1 - n\beta_{\text{\emph{rem}}}) \bar{t}} \alpha(T_{\text{\emph{ode}}})
    \end{equation}
    for $t \geq T_{\text{\emph{ode}}}$. Then there is a unique $v_{\infty} \in E_{A,\beta_{\text{\emph{rem}}}}^1 \oplus \cdots \oplus E_{A,\beta_{\text{\emph{rem}}}}^{n+1}$ (if $n \geq \mathcal{N}$, then $v_{\infty} \in \mathbb{C}^k$), a constant $C$ and a non-negative integer $N$ such that
    \begin{equation} \label{dependence}
    \begin{split}
        \bigg|v(t) - e^{A t} v_{\infty} - \int_{T_{\emph{ode}}}^t &e^{A(t-s)} A_{\text{\emph{rem}}} H(s)ds - \int_{T_{\emph{ode}}}^t e^{A(t-s)} F(s)ds  \bigg|\\
        &\leq C \langle \bar{t} \rangle^N e^{(\kappa_1 - (n+1)\beta_{\text{\emph{rem}}}) \bar{t}} (|v(T_{\text{\emph{ode}}})| + \gamma \alpha(T_{\text{\emph{ode}}}))
    \end{split}
    \end{equation}
    for $t \geq T_{\text{\emph{ode}}}$, where $C$ depends only on $K$, $C_{\text{\emph{rem}}}$, $\beta_{\text{\emph{rem}}}$ and $A$, and $N$ depends only on $M$ and $A$. Furthermore, if $v_{\infty} = e^{-A T_{\emph{ode}}} u_{\infty}$, then
    \begin{equation}
        |u_{\infty}| \leq C(|v(T_{\emph{ode}})| + \gamma \alpha(T_{\emph{ode}}))
    \end{equation}
    where $C$ has the same dependence as in \eqref{dependence}.
\end{lemma}

\begin{proof}
Let us introduce the following notation. For $n < \mathcal{N} - 1$, let
\[
E_{A,\beta_{\text{rem}}}^{a,n} = E_{A,\beta_{\text{rem}}}^1 \oplus \cdots \oplus E_{A,\beta_{\text{rem}}}^{n+1};
\]
also let $x = (x_{a,n}, x_{b,n}) \in \mathbb{C}^k$ where $x_{a,n} \in \mathbb{C}^{k_1 + \cdots + k_{n+1}}$ and $x_{b,n} \in \mathbb{C}^{k_{n+2} + \cdots + k_{\mathcal{N}}}$. Note that if $A_{a,n} = \diag(A_1, \ldots, A_{n+1})$ and $A_{b,n} = \diag(A_{n+2}, \ldots, A_{\mathcal{N}})$ (we use similar notation for the matrix $J_A$), then 
\[
\text{Rsp}(A_{a,n}) < (n+1)\beta_{\text{rem}}, \quad \kappa_{\text{max}}(A_{b,n}) \leq \kappa_{\text{max}}(A) - (n+1)\beta_{\text{rem}}. 
\]
For $n \geq \mathcal{N} - 1$, we take $E_{A,\beta}^{a,n} = \mathbb{C}^k$, $x_{a,n} = x$, $A_{a,n} = A_J$ ($J_{a,n} = J_A$) and disregard everything referring to $x_{b,n}$ and $A_{b,n}$ ($J_{b,n}$) below. In that case, $\text{Rsp}(A_J) < \mathcal{N}\beta_{\text{rem}} \leq (n+1)\beta_{\text{rem}}$.

Let $R(t) = R_A(t)$, $T = T_A$, $J = J_A$ (remember the comments made after Definition~\ref{generalized eigenspaces}). Define $w$ and $G$ by
\begin{equation} \label{w def}
    w(t) := e^{-\kappa_1 t} R(\bar{t}) T^{-1} v(t), \qquad G(t) := e^{-\kappa_1t} R(\bar{t}) T^{-1} F(t).
\end{equation}
Then
\begin{equation} \label{theequation}
    \dot{w}(t) = Jw(t) + A_{\text{rest}}(t) w(t) + G(t),
\end{equation}
where $A_{\text{rest}}(t) = R(\bar{t}) T^{-1} A_{\text{rem}}(t) T R(\bar{t})^{-1}$, so that
\[
\| A_{\text{rest}}(t) \| \leq C \gamma e^{-\beta_{\text{rem}} \bar{t}}
\]
for $t \geq T_{\text{ode}}$, with $C$ depending only on $C_{\text{rem}}$ and $A$. The assumption \eqref{asymptotic estimate} implies that 
\begin{equation}
\begin{split}
    |w(t) - \mathcal{H}(t)| &\leq \| e^{-\kappa_1 t} R(\bar{t}) T^{-1} \| |v(t) - H(t)|\\
    &\leq e^{-\kappa_1 t} \|T^{-1}\| K \langle \bar{t} \rangle^M e^{(\kappa_1 - n \beta_{\text{rem}})\bar{t}} \alpha(T_{\text{ode}})\\
    &\leq C e^{-\kappa_1 T_{\text{ode}}} \langle \bar{t} \rangle^M e^{-n\beta_{\text{rem}} \bar{t}} \alpha(T_{\text{ode}})
\end{split}
\end{equation}
for $t \geq T_{\text{ode}}$, where $\mathcal{H}(t) := e^{-\kappa_1 t} R(\bar{t}) T^{-1} H(t)$ and $C$ depends only on $K$ and $A$. We go back to \eqref{theequation}, which yields
\begin{align*}
    \frac{d}{dt} (e^{-Jt}w(t)) &= e^{-Jt} A_{\text{rest}}w + e^{-Jt}G,\\
    w(t) &= e^{J \bar{t}} w(T_{\text{ode}}) + e^{Jt} \int_{T_{\text{ode}}}^t e^{-Js} (A_{\text{rest}}w)(s)ds + e^{Jt} \int_{T_{\text{ode}}}^t e^{-Js} G(s)ds.
\end{align*}
%
%
%
Focusing on $w_{a,n}$, we can write
\begin{equation} \label{hatequation}
\begin{split}
    e^{-J_{a,n} \bar{t}} w_{a,n}(t) - w_{a,n}(T_{\text{ode}}) - \int_{T_{\text{ode}}}^t e^{-J_{a,n}\bar{s}} ( A_{\text{rest}}& \mathcal{H} )_{a,n}(s) ds - \int_{T_{\text{ode}}}^t e^{-J_{a,n} \bar{s}} G_{a,n}(s)ds\\
    &= \int_{T_{\text{ode}}}^t e^{-J_{a,n} \bar{s}} ( A_{\text{rest}} ( w - \mathcal{H}))_{a,n}(s) ds.
\end{split}
\end{equation}
By construction $\kappa_1 - \kappa_{\text{min}}(A_{a,n}) < (n+1)\beta_{\text{rem}}$, so that
\[ 
    \left|
    e^{-J_{a,n} \bar{s}} ( A_{\text{rest}} ( w - \mathcal{H}) )_{a,n}(s)  \right|
\]
decays exponentially. We can thus define
\[
w_{\infty} := w_{a,n}(T_{\text{ode}}) + \int_{T_{\text{ode}}}^{\infty} e^{-J_{a,n} \bar{s}} ( A_{\text{rest}} ( w - \mathcal{H}) )_{a,n}(s) ds.
\]
Note that
\begin{equation} \label{w inf estimate}
    |w_{\infty}| \leq C( |w(T_{\text{ode}})| + e^{-\kappa_1 T_{\text{ode}}} \gamma \alpha(T_{\text{ode}}) )
\end{equation}
where $C$ depends only on $K$, $C_{\text{rem}}$, $\beta_{\text{rem}}$, and $A$. Equation \eqref{hatequation} then implies
\begin{align}
\begin{split}
    &\bigg| w_{a,n}(t) - e^{J_{a,n} \bar{t}} w_{\infty} - e^{J_{a,n}t} \int_{T_{\text{ode}}}^t e^{-J_{a,n}s} ( A_{\text{rest}} \mathcal{H})_{a,n}(s) ds - e^{J_{a,n}t}\int_{T_{\text{ode}}}^t e^{-J_{a,n}s} G_{a,n}(s)ds \bigg|\\
    &\hspace{6cm} \leq \left| \int_t^{\infty} e^{J_{a,n}(t-s)} 
    ( A_{\text{rest}} ( w - \mathcal{H} ) )_{a,n}(s) ds
    \right|.
\end{split} \nonumber\\
\intertext{By Lemma~\ref{matrix exponential estimate}, we can estimate the right hand side of this by}
\begin{split}
    &\hspace{1cm}C \gamma e^{-\kappa_1 T_{\text{ode}}} \alpha(T_{\text{ode}}) \int_t^{\infty} \langle \bar{s} \rangle^N e^{(\kappa_1 -\kappa_{\text{min}}(A_{a,n})) (s-t)} e^{-(n+1)\beta_{\text{rem}} \bar{s}}ds\\
    &\hspace{7cm} \leq C \gamma e^{-\kappa_1 T_{\text{ode}}} \langle \bar{t} \rangle^N e^{-(n+1)\beta_{\text{rem}} \bar{t}} \alpha(T_{\text{ode}}).
\end{split} \nonumber\\
\intertext{Here $C$ depends only on $K$, $C_{\text{rem}}$, $\beta_{\text{rem}}$ and $A$, and $N = M + d - 1$ where $d$ is the largest size of a Jordan block of $A$. Thus we have obtained the estimate}
\begin{split}
    &\bigg| w_{a,n}(t) - e^{J_{a,n} \bar{t}} w_{\infty} - e^{J_{a,n}t} \int_{T_{\text{ode}}}^t e^{-J_{a,n}s} ( A_{\text{rest}} \mathcal{H})_{a,n}(s) ds - e^{J_{a,n}t}\int_{T_{\text{ode}}}^t e^{-J_{a,n}s} G_{a,n}(s)ds \bigg|\\
    &\hspace{7cm} \leq C \gamma e^{-\kappa_1 T_{\text{ode}}} \langle \bar{t} \rangle^N e^{-(n+1)\beta_{\text{rem}} \bar{t}} \alpha(T_{\text{ode}})
\end{split}  \label{data for w}\\
\intertext{for $t \geq T_{\text{ode}}$. Now we need to estimate the remaining components of $w$. The equation for $w_{b,n}$ yields}
\begin{split}
    &w_{b,n}(t) - e^{J_{b,n}\bar{t}} w_{b,n}(T_{\text{ode}}) - e^{J_{b,n}t} \int_{T_{\text{ode}}}^t e^{-J_{b,n}s} ( A_{\text{rest}} \mathcal{H})_{b,n} ds - e^{J_{b,n} t}\int_{T_{\text{ode}}}^t e^{-J_{b,n} s} G_{b,n}(s)ds\\
    &\hspace{7cm} = e^{J_{b,n}t} \int_{T_{\text{ode}}}^t e^{-J_{b,n}s} \left( A_{\text{rest}} ( w - \mathcal{H}) \right)_{b,n}(s) ds,
\end{split} \nonumber\\
\intertext{implying}
\begin{split}
    &\hspace{0.5cm}\bigg| w_{b,n}(t) - \int_{T_{\text{ode}}}^t e^{J_{b,n}(t-s)} ( A_{\text{rest}} \mathcal{H} )_{b,n}(s) - \int_{T_{\text{ode}}}^t e^{J_{b,n} (t-s)} G_{b,n}(s)ds \bigg|\\
    &\hspace{6cm} \leq C \langle \bar{t} \rangle^N e^{-(n+1)\beta_{\text{rem}} \bar{t}} (|w(T_{\text{ode}})| + \gamma e^{-\kappa_1 T_{\text{ode}}} \alpha(T_{\text{ode}})),
\end{split} \nonumber\\
\intertext{where $C$ depends only on $K$, $C_{\text{rem}}$, $\beta_{\text{rem}}$, and $A$, and $N$ is the same as above. Combining both estimates we obtain}
\begin{split}
    &\hspace{0.5cm}\bigg| w(t) - e^{J \bar{t}} 
    \begin{pmatrix}
        w_{\infty}\\
        0
    \end{pmatrix} - \int_{T_{\text{ode}}}^t e^{J(t-s)} A_{\text{rest}} \mathcal{H} ds - \int_{T_{\text{ode}}}^t e^{J (t-s)} G(s)ds \bigg|\\
    &\hspace{6cm} \leq C \langle \bar{t} \rangle^N e^{-(n+1)\beta_{\text{rem}} \bar{t}} (|w(T_{\text{ode}})| + \gamma e^{-\kappa_1 T_{\text{ode}}} \alpha(T_{\text{ode}})).
\end{split} \nonumber\\
\intertext{Now we go back to expressing everything in terms of $v$,}
\begin{split}
    &\hspace{0.5cm}\bigg| v(t) - e^{A \bar{t}} u_{\infty} - \int_{T_{\text{ode}}}^t e^{A(t-s)} A_{\text{rem}} H(s) ds - \int_{T_{\text{ode}}}^t e^{A (t-s)} F(s)ds\bigg|\\ 
    &\hspace{6cm} \leq C \langle \bar{t} \rangle^N e^{(\kappa_1 - (n+1)\beta_{\text{rem}}) \bar{t}} (|v(T_{\text{ode}})| + \gamma \alpha(T_{\text{ode}})) 
\end{split} \nonumber
\end{align}
where 
\[
u_{\infty} := e^{\kappa_1 T_{\text{ode}}} T 
\begin{pmatrix}
    w_{\infty}\\
    0
\end{pmatrix}
\]
(in case $n \geq \mathcal{N} - 1$, $(w_{\infty},0)$ is replaced by just $w_{\infty}$). By defining $v_{\infty} := e^{-A T_{\text{ode}}} u_{\infty} \in E_{A,\beta_{\text{rem}}}^{a,n}$, we obtain \eqref{dependence}. The estimate for $u_{\infty}$ comes from its definition and \eqref{w inf estimate}.

Finally, for uniqueness, suppose there is a $\tilde{v}_{\infty} \in E^{a,n}_{A,\beta_{\text{rem}}}$ such that \eqref{dependence} holds with $\tilde{v}_{\infty}$ instead of $v_{\infty}$. Then, by using both estimates, we obtain
\[
|e^{At}( v_{\infty} - \tilde{v}_{\infty} )| \leq C\langle \bar{t} \rangle^N e^{(\kappa_1 - (n+1)\beta_{\text{rem}}) \bar{t}}.
\]
But $v_{\infty} - \tilde{v}_{\infty} \in E^{a,n}_{A,\beta_{\text{rem}}}$, so this estimate can only be true if the left hand side is zero, since otherwise the left hand side would grow more quickly that the right hand side. We conclude that $\tilde{v}_{\infty} = v_{\infty}$.
\end{proof}

\begin{remark}
    Note that the term
    \[
    e^{At}v_{\infty} + \int_{T_{\text{ode}}}^t e^{A(t-s)}F(s)ds
    \]
    present in \eqref{dependence} is a solution to the equation $\dot{v} = Av + F$.
\end{remark}

As a first application of this result, we can use the leading order estimates obtained in \cite[Lemma 9.16]{hans} as a starting point to deduce asymptotic estimates of all orders for solutions of Equation \eqref{ode}.

\begin{proposition} \label{ode asymptotic lemma}
Consider Equation \eqref{ode} and assume \eqref{arem estimate} and \eqref{inhomogeneous assumption} to hold. Then for every positive integer $n$, there is a constant $C$, depending only on $C_{\text{\emph{rem}}}$, $\beta_{{\text{\emph{rem}}}}$ and $A$; and a non-negative integer $N$ depending only on $k$, such that the following holds. If $v$ is a solution to \eqref{ode}, there is a unique $v_{\infty,n} \in E_{A,\beta_{\text{\emph{rem}}}}^1 \oplus \cdots \oplus E_{A,\beta_{\text{\emph{rem}}}}^n$ (if $n > \mathcal{N}$, then $v_{\infty,n} \in \mathbb{C}^k$) such that
\begin{equation} \label{estimate}
    |v(t) - F_{\infty,n}(t)| \leq C\langle \bar{t} \rangle^N e^{(\kappa_1 - n\beta_{\text{\emph{rem}}}) \bar{t}} (|v(T_{\text{\emph{ode}}})| + e^{\kappa_1 T_{\text{\emph{ode}}}}\|F\|_A),
\end{equation}
for $t \geq T_{\text{\emph{ode}}}$, where $F_{\infty,n}$ is given by the recursive formula
\[
    F_{\infty,n}(t) = e^{At} v_{\infty,n} + \int_{T_{\text{\emph{ode}}}}^t e^{A(t-s)} A_{\text{\emph{rem}}} F_{\infty, n-1}(s) ds + \int_{T_{\text{\emph{ode}}}}^t e^{A(t-s)} F(s)ds,
\]
and $F_{\infty, 0} = 0$. Furthermore, if $v_{\infty, n} = e^{-A T_{\emph{ode}}} u_{\infty, n}$, 
\begin{equation} \label{u estimate}
    |u_{\infty, n}| \leq C(|v(T_{\emph{ode}})| + e^{\kappa_1 T_{\text{\emph{ode}}}}\|F\|_A).
\end{equation}
\end{proposition}

\begin{remark}
    The non-negative integer $N$ depends linearly on $n$. This is readily seen by checking the proof of Lemma~\ref{asymptotic lemma}.
\end{remark}

\begin{proof}
We begin with an application of \cite[Lemma 9.16]{hans}, which yields a constant $C$ depending only on $C_{\text{rem}}$, $\beta_{\text{rem}}$ and $A$, and a non-negative integer $N$ depending only on $k$, such that for every solution $v$ to the equation, there is a unique $v_{\infty, 1} \in E_{A, \beta_{\text{rem}}}^1$ such that
\begin{equation*}
    \bigg|v(t) - e^{At} v_{\infty, 1} - \int_{T_{\text{ode}}}^t e^{A(t-s)} F(s)ds \bigg| \leq C \langle \bar{t} \rangle^N e^{(\kappa_1 - \beta_{\text{rem}}) \bar{t}} (|v(T_{\text{ode}})| + e^{\kappa_1 T_{\text{ode}}}\|F\|_A )
\end{equation*}
for $t \geq T_{\text{ode}}$. Note that this estimate is of the form \eqref{estimate} for $n = 1$ and that \eqref{u estimate} also holds for $v_{\infty,1}$. 

We now proceed with the inductive step, so let us assume that we have an estimate as in \eqref{estimate} for some $n \geq 1$. We apply Lemma \ref{asymptotic lemma} to \eqref{estimate}. Note that in this case we have $\gamma = 1$ and $\alpha(T_{\text{ode}}) = |v(T_{\text{ode}})| + e^{\kappa_1 T_{\text{ode}}}\|F\|_A$. The claim follows from the conclusions of the lemma.
\end{proof}

\subsection{Specifying asymptotic data}

From now on we focus on the homogeneous version of \eqref{ode}; that is, $F = 0$. Let $1 \leq n \leq \mathcal{N}$. If we have a solution $v$ for which $\pi_j(v_{\infty,j})$ vanishes for $j = 1, \ldots,n-1$ then, by uniqueness of the $v_{\infty,j}$, the estimate \eqref{estimate} becomes
\begin{equation}
    |v(t) - e^{At}v_{\infty,n}| \leq C \langle \bar{t} \rangle^N e^{(\kappa_1 - n\beta_{\text{rem}}) \bar{t}} |v(T_{\text{ode}})|
\end{equation}
for $t \geq T_{\text{ode}}$, where $v_{\infty,n} = \pi_n(v_{\infty,n}) \in E_{A,\beta_{\text{rem}}}^n$. This suggests that in general, the new asymptotic information provided by $v_{\infty,n}$ is contained in $\pi_n(v_{\infty,n})$ only. Moreover, note that if $n > \mathcal{N}$ and $\pi_j(v_{\infty,j}) = 0$ for $j = 1, \ldots, \mathcal{N}$, then $v_{\infty,n} = 0$ (in fact, this would imply that $v$ is the zero solution; see Proposition~\ref{ode iso} below). Therefore, the $v_{\infty,n}$ for $n > \mathcal{N}$ do not provide any new asymptotic information. Motivated by this observation, we make the following definition.

\begin{definition} \label{asymptotic data}
    Let $v$ be a solution of \eqref{ode} with $F = 0$ and assume \eqref{arem estimate} holds. Then, if $n \leq \mathcal{N}$, the vector $\pi_n(v_{\infty,n}) \in E_{A,\beta_{\text{rem}}}^n$ is called the \emph{$n$-th order asymptotic data} for $v$.
\end{definition}

\begin{remark} \label{w data}
    Let $v$ be a solution of \eqref{ode} with $F = 0$ and assume that \eqref{arem estimate} holds. Then $w$, as defined in \eqref{w def}, satisfies Equation \eqref{theequation} with $G = 0$ and we can apply Proposition \ref{ode asymptotic lemma} to it. The $n$-th order asymptotic data for $w$ is given by
    \[
    (w_{\infty,n})_n = w_n(T_{\text{ode}}) + \int_{T_{\text{ode}}}^{\infty} e^{-J_n \bar{s}} (A_{\text{rest}} (w - \mathcal{F}_{\infty,n-1}) )_n(s)ds,
    \]
    where $\mathcal{F}_{\infty, n-1}(t) := e^{-\kappa_1 t} R(\bar{t}) T^{-1} F_{\infty, n-1}(t)$ and $F_{\infty, n-1}$ is related to $v$ as in the statement of Proposition \ref{ode asymptotic lemma}. To see this, follow the proof of Lemma \ref{asymptotic lemma} with $H = F_{\infty, n-1}$. Alternatively, we also have
    \[
        (w_{\infty,n})_n = \lim_{t \to \infty} \bigg( e^{-J_n \bar{t}} w_n(t) + \int_{T_{\text{ode}}}^t e^{-J_n \bar{s}}(A_{\text{rest}} \mathcal{F}_{\infty, n-1})_n(s)ds \bigg),
    \]
    which follows from \eqref{data for w}.
\end{remark}

Our objective in this section is to prove that if $1 \leq n \leq \mathcal{N}$ and $\pi_j(v_{\infty,j}) 
= 0$ for $j = 1, \ldots, n-1$, then $v_{\infty,n}$ can be specified. 

\begin{lemma} \label{specify ode data}
Consider \eqref{ode} with $F = 0$ and assume that \eqref{arem estimate} holds. Then for every ${1 \leq n \leq \mathcal{N}}$, there is an injective linear map $\Psi_{\infty,n}: E_{A,\beta_{\text{\emph{rem}}}}^n \to \mathbb{C}^k$ such that if $v_{\infty,n} \in E_{A,\beta_{\text{\emph{rem}}}}^n$ and $v$ is the solution to \eqref{ode} satisfying $v(T_{\text{\emph{ode}}}) = \Psi_{\infty,n}(v_{\infty,n})$, then $v$ satisfies
\begin{equation} \label{special estimate}
    |v(t) - e^{At}v_{\infty,n}| \leq C \langle \bar{t} \rangle^N e^{(\kappa_1 - n\beta_{\text{\emph{rem}}}) \bar{t}} |v(T_{\text{\emph{ode}}})|
\end{equation}
Furthermore, if $u_{\infty,n} \in E_{A,\beta_{\text{\emph{rem}}}}^n$, 
\begin{equation} \label{psi bound}
    |\Psi_{\infty,n}(e^{-A T_{\text{\emph{ode}}}} u_{\infty,n})| \leq C|u_{\infty,n}|.
\end{equation}
Finally, the constant $C$ appearing in these estimates only depends on $C_{\text{\emph{rem}}}$, $\beta_{\text{\emph{rem}}}$ and $A$, and the non-negative integer $N$ depends only on $k$. 
\end{lemma}

\begin{proof}
First note that the statement for $n = 1$ is given by \cite[Lemma 9.21]{hans}. Indeed, note that the subspace $E_{A, \beta_{\text{rem}}}^1$ is equal to $E_{A,\beta}$ from \cite[Lemma 9.21]{hans} if we let $\beta = \beta_{\text{rem}}$. Thus we can define $\Psi_{\infty,1} = \Psi_{\infty}$, where $\Psi_{\infty}$ is the map given by \cite[Lemma 9.21]{hans}.  

Now we make the following inductive assumption. Assume that the statement of the Lemma is true for some $n \geq 1$. Furthermore, assume there is a $t_n \geq T_{\text{ode}}$, such that $t_n - T_{\text{ode}}$ is bounded by a constant depending only on $C_{\text{rem}}$, $\beta_{\text{rem}}$ and $A$; and for $j = 1, \ldots, n$, that there are subspaces $W_j(s) \subset \mathbb{C}^{k_1 + \cdots + k_j}$ of dimension $k_j$ for $s \geq t_n$ such that solutions $w$, as defined in \eqref{w def}, to \eqref{theequation} with initial data in $W_j(s) \times \{0\}$ at time $s$, have vanishing asymptotic data of order $1$ to $j-1$ and the map that takes initial data in $W_j(s) \times \{0\}$ to asymptotic data of order $j$ is a linear isomorphism. 

Note that for $n = 1$, the inductive assumption requires $W_1(s) = \mathbb{C}^{k_1}$, and the proof of this statement is contained in the proof of \cite[Lemma 9.21]{hans}. Specifically, $\bar{u}_a(t)$ in the notation of \cite{hans} is equal to $e^{-J_1 \bar{t}}w_1(t)$ in our notation (recall that our notation for $J$ was introduced in the comments after Definition~\ref{generalized eigenspaces}), and the map that fulfills the assumption just described for $n = 1$ is $L_a$ as defined in \cite[(9.52), p. 134]{hans}.

Since $\| A_{\text{rem}} \| \leq C_{\text{rem}} e^{-\beta_{\text{rem}}\bar{t_0}} e^{-\beta_{\text{rem}}(t-t_0)} \leq C_{\text{rem}} e^{-\beta_{\text{rem}}(t-t_0)}$ for $t_0 \geq T_{\text{ode}}$, we can work with initial data at any $t_0 \geq T_{\text{ode}}$ instead of $T_{\text{ode}}$ without worrying about the dependence of the constants in the estimates. By integrating \eqref{theequation}, but from some $t_0 \geq t_n$, we get
\begin{equation} \label{bequation}
    w(t) = e^{J(t - t_0)} w(t_0) + e^{Jt} \int_{t_0}^t e^{-Js} (A_{\text{rest}} w)(s) ds.
\end{equation}
Define the map $L_1(t_0): \mathbb{C}^{k_1 + \cdots + k_{n+1}} \to \mathbb{C}^{k_1}$ by
\[
L_1(t_0)(\eta) = \lim_{t \to \infty} e^{-J_1 (t - t_0)} w_1(t),
\]
where $w$ is the solution of \eqref{theequation} such that $w(t_0) = (\eta,0)$. This map gives the first order asymptotic data for $w$ (see Remark \ref{w data} with $T_{\text{ode}}$ replaced by $t_0$). Then, by the inductive assumption, $L_1(t_0)$ is a surjective linear map. We consider the space $\ker(L_1(t_0))$. It has dimension $k_2 + \cdots + k_{n+1}$ and solutions $w$ with initial data $w(t_0) \in \ker(L_1(t_0)) \times \{0\}$ have vanishing asymptotic data of order 1. Now consider the map $L_2(t_0): \ker (L_1(t_0)) \to \mathbb{C}^{k_2}$,
\[
L_2(t_0)(\eta) = \lim_{t \to \infty} e^{-J_2 (t - t_0)} w_2(t), \qquad w(t_0) = (\eta,0).
\]
This map gives the second order asymptotic data for $w$. Note that, since $W_2(t_0) \times \{0\}$ is a subspace of $\ker (L_1(t_0))$, $L_2(t_0)$ is surjective and $\ker (L_2(t_0))$ has dimension $k_3 + \cdots + k_{n+1}$. By continuing this process we eventually get to $W_{n+1}(t_0) = \ker (L_n(t_0))$, which has dimension $k_{n+1}$ and solutions $w$ with initial data $w(t_0) \in W_{n+1}(t_0) \times \{0\}$ have vanishing asymptotic data of order $1$ to $n$. Since the largest real part of an eigenvalue of $J$ in \eqref{theequation} is zero, the $n$-th order asymptotic estimate \eqref{estimate} then becomes
\begin{equation} \label{first estimate}
    |w(t)| \leq C \langle t - t_0 \rangle^N e^{-n\beta_{\text{rem}}(t - t_0)} |w(t_0)|
\end{equation}
when obtained from data at $t_0$; here $C$ only depends on $C_{\text{rem}}$, $\beta_{\text{rem}}$ and $A$.   

Now we want to have some control over the ``slope'' of the subspace $W_{n+1}(t_0)$. Let ${\eta = (\eta_a, \eta_{n+1}) \in W_{n+1}(t_0)}$, where $\eta_a \in \mathbb{C}^{k_1 + \cdots + k_n}$, and let $w$ be the solution of \eqref{theequation} with ${w(t_0) = (\eta,0)}$. Then \eqref{bequation} implies
\[
0 = \lim_{t \to \infty} e^{-J_a(t-t_0)} w_a(t) = w_a(t_0) + \int_{t_0}^{\infty} e^{-J_a(s-t_0)} (A_{\text{rest}} w)_a (s)ds. 
\]
Therefore
\[
\begin{split}
    |\eta_a| &\leq \int_{t_0}^{\infty} |e^{-J_a(s-t_0)} (A_{\text{rest}} w)_a(s)| ds\\
    &\leq C \int_{t_0}^{\infty} \langle s-t_0 \rangle^N e^{(\kappa_1 - \kappa_{\text{min}}(A_n)) (s-t_0)} e^{-\beta_{\text{rem}} (s - T_{\text{ode}})} e^{-n \beta_{\text{rem}} (s-t_0)} ds |\eta|\\
    &\leq C e^{-\beta_{\text{rem}} (t_0 - T_{\text{ode}})} \int_{t_0}^{\infty} \langle s-t_0 \rangle^N e^{(\kappa_1 - \kappa_{\text{min}}(A_n)) (s-t_0)} e^{-(n+1) \beta_{\text{rem}} (s-t_0)} ds (|\eta_a| + |\eta_{n+1}|)\\
    &\leq C e^{-\beta_{\text{rem}} (t_0 - T_{\text{ode}})} \int_{t_0}^{\infty} \langle s-t_0 \rangle^N e^{-\beta_{\text{rem}} (s-t_0)} ds (|\eta_a| + |\eta_{n+1}|)
\end{split}
\]
where $C$ depends only on $C_{\text{rem}}$, $\beta_{\text{rem}}$ and $A$; $N$ depends only on $k$; and we have used \eqref{first estimate}. Since the value of the integral is actually independent of $t_0$, given $\varepsilon > 0$, we can thus take $t_0$ large enough such that
\[
|\eta_a| \leq \varepsilon(|\eta_a| + |\eta_{n+1}|).
\]
By taking such a $t_0$ for $\varepsilon \leq 1/2$, we obtain the required slope estimate,
\begin{equation} \label{slope}
    |\eta_a| \leq |\eta_{n+1}|.
\end{equation}
Now let $L_{n+1}: W_{n+1}(t_0) \to \mathbb{C}^{k_{n+1}}$ be defined by
\[
L_{n+1}(\eta) = \lim_{t \to \infty} e^{-J_{n+1}(t-t_0)} w_{n+1}(t),
\]
where $w$ is the solution of \eqref{theequation} such that $w(t_0) = (\eta,0)$. As above, this is well defined because initial data in $W_{n+1}(t_0) \times \{0\}$ at $t_0$ gives rise to solutions with vanishing asymptotic data of orders $1$ to $n$. This map gives the $(n+1)$-th order asymptotic data for $w$. Also from \eqref{bequation} and \eqref{first estimate} we obtain the estimate
\[
\begin{split}
    |L_{n+1}(\eta) \, - \, &w_{n+1}(t_0)|\\
    &\leq \left| \int_{t_0}^{\infty} e^{-J_{n+1}(s-t_0)} (A_{\text{rest}} w)_{n+1} (s) ds \right|\\ 
    &\leq C \int_{t_0}^{\infty} \langle s - t_0 \rangle^N e^{(\kappa_1 - \kappa_{\text{min}}(A_{n+1})) (s - t_0)} e^{-\beta_{\text{rem}} (s - T_{\text{ode}})} e^{-n\beta_{\text{rem}} (s - t_0)} ds |w(t_0)|\\
    &\leq C e^{-\beta_{\text{rem}} (t_0 - T_{\text{ode}})}  \int_{t_0}^{\infty} \langle s-t_0 \rangle^N e^{( \kappa_1 - \kappa_{\text{min}}(A_{n+1}) - (n+1)\beta_{\text{rem}} ) (s-t_0)} ds |w(t_0)|,
\end{split}
\]
where $C$ depends only on $C_{\text{rem}}$, $\beta_{\text{rem}}$ and $A$, and $N$ depends only on $k$. Since the exponential inside the integral is decaying, and the value of the integral is actually independent of $t_0$, given $\varepsilon > 0$, we can take $t_0$ large enough to get
\[
|L_{n+1}(\eta) - \eta_{n+1}| \leq \varepsilon(|\eta_a| + |\eta_{n+1}|),
\]
and by using \eqref{slope},
\[
\begin{split}
    |L_{n+1}(\eta)| &\geq (1-\varepsilon)|\eta_{n+1}| - \varepsilon|\eta_a|\\
    &\geq (1-2\varepsilon)|\eta_{n+1}|\\
    &\geq \frac{1}{2}(1-2\varepsilon)(|\eta_a| + |\eta_{n+1}|),
\end{split}
\]
assuming $\varepsilon < 1/2$. We obtain that for an appropriate $t_0$, there is a $C > 0$ such that
\[
|L_{n+1}(\eta)| \geq C|\eta|,
\]
which means that $L_{n+1}$ is a linear injective map between spaces of the same dimension, so it is an isomorphism. Since all of the above is true for any $s \geq t_0$, the subspace $W_{n+1}(s)$ can be defined for any $s \geq t_0$. At this point, we can set $t_{n+1} = t_0$. The map $L_{n+1}^{-1}:\mathbb{C}^{k_{n+1}} \to W_{n+1}(t_0)$, then maps $(n+1)$-th order asymptotic data to initial data at $t_0$ such that the corresponding solution has vanishing asymptotic data of order 1 to $n$. For $L_{n+1}^{-1}$ we have the estimate
\[
|L_{n+1}^{-1}(\eta)| \leq \frac{1}{C} |\eta|.
\]
The remaining step is to translate this to data at $T_{\text{ode}}$ and write everything in terms of $v$. We have
\[
v_{\infty,n+1} = e^{-AT_{\text{ode}}} e^{\kappa_1 T_{\text{ode}}} T 
\begin{pmatrix}
    w_{\infty}\\
    0
\end{pmatrix},
\]
where $(w_{\infty},0) = (x, \xi, 0)$ for some $x \in \mathbb{C}^{k_1 + \cdots + k_n}$ and $\xi \in \mathbb{C}^{k_{n+1}}$, and $w_{\infty}$ is defined as in the proof of Lemma~\ref{asymptotic lemma}, so that $\xi$ is asymptotic data of order $n+1$ for $w$ obtained from initial data at $T_{\text{ode}}$. Note that in the present setting, $w$ arises from initial data in $W_{n+1}(t_0) \times \{0\}$ at $t_0$, which implies $x = 0$. We need to know how this relates to the map $L_{n+1}$. Explicitly, we have
\[
\xi = \lim_{t \to \infty} e^{-J_{n+1} \bar{t}} w_{n+1}(t).
\]
Hence
\[
\xi = e^{-J_{n+1} \bar{t_0}} L_{n+1}(\eta),
\]
where $w(t_0) = (\eta,0)$ and $\eta \in W_{n+1}(t_0)$. Now we can define the map $\Psi_{\infty,n+1}$. Given ${v_{\infty,n+1} \in E_{A,\beta_{\text{rem}}}^{n+1}}$, let $\xi \in \mathbb{C}^{k_{n+1}}$ be the unique vector such that
\[
v_{\infty,n+1} = e^{-AT_{\text{ode}}} e^{\kappa_1 T_{\text{ode}}} T 
\begin{pmatrix}
    0\\
    \xi\\
    0
\end{pmatrix}.
\]
Let $w$ be the solution of \eqref{theequation} with
\[
w(t_0) = 
\begin{pmatrix}
    L_{n+1}^{-1} (e^{J_{n+1} \bar{t_0}} \xi)\\
    0
\end{pmatrix},
\]
and set $\Psi_{\infty,n+1}(v_{\infty,n+1}) := e^{\kappa_1 T_{\text{ode}}} T w(T_{\text{ode}})$. Clearly $\Psi_{\infty,n+1}$ is linear and injective. In order to obtain \eqref{psi bound}, note that $t_0 - T_{\text{ode}}$ can be bounded by a constant depending only on $C_{\text{rem}}$, $\beta_{\text{rem}}$ and $A$. Indeed, our choice of $t_0$ is dictated by the condition $C e^{-\beta_{\text{rem}}(t_0 - T_{\text{ode}})} < 1/2$ for some constant $C$ depending only on $C_{\text{rem}}$, $\beta_{\text{rem}}$ and $A$. If this condition is satisfied by $t_0 = t_n$, we can just make this choice and the bound comes from the inductive assumption. Otherwise, we can choose $t_0$ such that
\[
\frac{1}{\beta_{\text{rem}}} \ln 2C < t_0 - T_{\text{ode}} \leq \frac{1}{\beta_{\text{rem}}} \ln 2C + 1. 
\]
Now we can use \eqref{theequation} to estimate,
\[
    \frac{d}{dt} |w|^2 = 2 \langle \dot{w}, w \rangle \geq -2(\|J\| + C e^{-\beta_{\text{rem}} \bar{t}}) |w|^2,
\]
where $C$ depends only on $C_{\text{rem}}$, $\beta_{\text{rem}}$ and $A$. By applying Grönwall's inequality (with time reversed) and using the bound on $t_0 - T_{\text{ode}}$, we obtain that there is a constant $C$ with the same dependence, such that for all solutions $w$ of \eqref{theequation},  
\[
    |w(T_{\text{ode}})| \leq C|w(t_0)|.
\]
Now let $v_{\infty,n+1} = e^{-A T_{\text{ode}}} u_{\infty,n+1}$, then
\begin{align*}
    |\Psi_{\infty,n+1}(e^{-A T_{\text{ode}}}u_{\infty,n+1})| &\leq C e^{\kappa_1 T_{\text{ode}}} |w(T_{\text{ode}})|\\
    &\leq C e^{\kappa_1 T_{\text{ode}}} |w(t_0)|\\
    &= C e^{\kappa_1 T_{\text{ode}}} |L_{n+1}^{-1}( e^{J_{n+1} \bar{t_0}} \xi )|\\
    &\leq C e^{\kappa_1 T_{\text{ode}}} |\xi|\\
    &\leq C |u_{\infty,n+1}|,
\end{align*}
where $C$ only depends on $C_{\text{rem}}$, $\beta_{\text{rem}}$ and $A$. By construction, it is clear that the map $\Psi_{\infty,n+1}$ does what it is supposed to.
\end{proof}

\begin{remark}
Note that the estimate \eqref{special estimate} implies that $\pi_j(v_{\infty,j})$ vanishes for $j = 1, \ldots,n-1$.
\end{remark}

\subsection{Isomorphism between initial data and asymptotic data}

We are now ready to show that solutions to \eqref{ode}, are completely determined by the asymptotic data of orders $1$ to $\mathcal{N}$ as defined in Definition \ref{asymptotic data}. For that purpose, we collect all of the asymptotic data in a single vector.

\begin{definition}
    Let $v$ be a solution to \eqref{ode} and assume that \eqref{arem estimate} and \eqref{inhomogeneous assumption} hold. Define $V_{\infty} \in \mathbb{C}^k$ by
    \[
    V_{\infty} := v_{\infty,1} + \pi_2(v_{\infty,2}) + \cdots + \pi_{\mathcal{N}}(v_{\infty, \,
    \mathcal{N}}).
    \]
    Then $V_{\infty}$ is called \emph{asymptotic data} for $v$.
\end{definition}

We begin by proving that asymptotic data can be added when $F = 0$. 

\begin{lemma} \label{add data}
    Let $v$ and $\tilde{v}$ be solutions to \eqref{ode} with $F = 0$ and assume that \eqref{arem estimate} holds. If the asymptotic data for $v$ and $\tilde{v}$ are given by $V_{\infty}$ and $\widetilde{V}_{\infty}$ respectively, then $v + \tilde{v}$ has asymptotic data given by $V_{\infty} + \widetilde{V}_{\infty}$.
\end{lemma}

\begin{proof}
We must do this inductively. Let
\[
\begin{split}
V_{\infty} &= \pi_1(V_{\infty}) + \cdots + \pi_{\mathcal{N}}(V_{\infty}),\\
\widetilde{V}_{\infty} &= \pi_1(\widetilde{V}_{\infty}) + \cdots + \pi_{\mathcal{N}}(\widetilde{V}_{\infty}).
\end{split}
\]
We begin with the leading order estimates. For $v$ and $\tilde{v}$ we have
\[
\begin{split}
    |v(t) - e^{At} \pi_1(V_{\infty})| &\leq C \langle \bar{t} \rangle^N e^{(\kappa_1 - \beta_{\text{rem}}) \bar{t}} |v(T_{\text{ode}})|,\\
    |\tilde{v}(t) - e^{At} \pi_1(\widetilde{V}_{\infty})| &\leq C \langle \bar{t} \rangle^N e^{(\kappa_1 - \beta_{\text{rem}}) \bar{t}} |\tilde{v}(T_{\text{ode}})|;
\end{split}
\]
which together imply
\[
|v(t) + \tilde{v}(t) - e^{At} (\pi_1(V_{\infty}) + \pi_1(\widetilde{V}_{\infty}))| \leq C \langle \bar{t} \rangle^N e^{(\kappa_1 - \beta_{\text{rem}}) \bar{t}} (|v(T_{\text{ode}})| + |\tilde{v}(T_{\text{ode}})|). 
\]
On the other hand, the estimate for $v + \tilde{v}$ is
\[
\begin{split}
    |v(t) + \tilde{v}(t) - e^{At} x| &\leq C \langle \bar{t} \rangle^N e^{(\kappa_1 - \beta_{\text{rem}}) \bar{t}} |v(T_{\text{ode}}) + \tilde{v}(T_{\text{ode}})|\\
    &\leq C \langle \bar{t} \rangle^N e^{(\kappa_1 - \beta_{\text{rem}}) \bar{t}} (|v(T_{\text{ode}})| + |\tilde{v}(T_{\text{ode}})|),
\end{split}
\]
where $x \in E_{A,\beta_{\text{rem}}}^1$. But then, uniqueness implies $x = \pi_1(V_{\infty}) + \pi_1(\widetilde{V}_{\infty})$. 

Now for the inductive step, consider the $(n+1)$-th estimates for $v$ and $\tilde{v}$;
\[
\begin{split}
    \left| v(t) - e^{At} v_{\infty, n+1} - \int_{T_{\text{ode}}}^t e^{A(t-s)} A_{\text{rem}} F_{\infty,n}(s)ds \right| &\leq C \langle \bar{t} \rangle^N e^{(\kappa_1 - (n+1)\beta_{\text{rem}}) \bar{t}} |v(T_{\text{ode}})|,\\
    \left| \tilde{v}(t) - e^{At} \tilde{v}_{\infty, n+1} - \int_{T_{\text{ode}}}^t e^{A(t-s)} A_{\text{rem}} \widetilde{F}_{\infty,n}(s)ds \right| &\leq C \langle \bar{t} \rangle^N e^{(\kappa_1 - (n+1)\beta_{\text{rem}}) \bar{t}} |\tilde{v}(T_{\text{ode}})|.
\end{split}
\]
Here $\pi_{n+1}(v_{\infty, n+1}) = \pi_{n+1}(V_{\infty})$ and $\pi_{n+1}(\tilde{v}_{\infty, n+1}) = \pi_{n+1}(\widetilde{V}_{\infty})$. Assume the $n$-th estimate for $v+\tilde{v}$ is
\[
|v(t) + \tilde{v}(t) - (F_{\infty,n} + \widetilde{F}_{\infty,n})(t)| \leq C \langle \bar{t} \rangle^N e^{(\kappa_1 - n\beta_{\text{rem}}) \bar{t}} |v(T_{\text{ode}}) + \tilde{v}(T_{\text{ode}})|.
\]
Thus the $(n+1)$-th estimate is 
\[
\begin{split}
    \bigg|v(t) + \tilde{v}(t) - e^{At} x - \int_{T_{\text{ode}}}^t e^{A(t-s)} &A_{\text{rem}} (F_{\infty,n} + \widetilde{F}_{\infty,n} )(s)ds\bigg|\\
    &\leq C \langle \bar{t} \rangle^N e^{(\kappa_1 - (n+1)\beta_{\text{rem}}) \bar{t}} |v(T_{\text{ode}}) + \tilde{v}(T_{\text{ode}})|\\
    &\leq C \langle \bar{t} \rangle^N e^{(\kappa_1 - (n+1)\beta_{\text{rem}}) \bar{t}} (|v(T_{\text{ode}})| + |\tilde{v}(T_{\text{ode}})|)
\end{split}
\]
for some $x \in E_{A,\beta_{\text{rem}}}^1 \oplus \cdots \oplus E_{A,\beta_{\text{rem}}}^{n+1}$. But the estimates for $v$ and $\tilde{v}$ imply
\[
\begin{split}
    \bigg|v(t) + \tilde{v}(t) - e^{At} (v_{\infty, n+1} + \tilde{v}_{\infty, n+1}) - \int_{T_{\text{ode}}}^t &e^{A(t-s)} A_{\text{rem}} (F_{\infty,n} + \widetilde{F}_{\infty,n} )(s)ds\bigg|\\
    &\leq C \langle \bar{t} \rangle^N e^{(\kappa_1 - (n+1)\beta_{\text{rem}}) \bar{t}} (|v(T_{\text{ode}})| + |\tilde{v}(T_{\text{ode}})|).
\end{split}
\]
Therefore, uniqueness yields $x = v_{\infty, n+1} + \tilde{v}_{\infty, n+1}$ and we can write the estimate for $v + \tilde{v}$ as 
\[
    |v(t) + \tilde{v}(t) - (F_{\infty,n+1} + \widetilde{F}_{\infty,n+1})(t)|
    \leq C \langle \bar{t} \rangle^N e^{(\kappa_1 - (n+1)\beta_{\text{rem}}) \bar{t}} |v(T_{\text{ode}}) + \tilde{v}(T_{\text{ode}})|.
\]
In particular,
\[
\pi_{n+1}(x) = \pi_{n+1}(v_{\infty, n+1}) + \pi_{n+1}(\tilde{v}_{\infty, n+1}) = \pi_{n+1}(V_{\infty}) + \pi_{n+1}(\widetilde{V}_{\infty}). \qedhere
\]
\end{proof}

We can now proceed with the main result of this section.

\begin{proposition} \label{ode iso}
    Consider Equation \eqref{ode} with $F = 0$ and assume that \eqref{arem estimate} holds. Then there is a linear isomorphism
    \[
    \Psi_{\infty}: \mathbb{C}^k \to \mathbb{C}^k
    \]
    such that if $U_{\infty} \in \mathbb{C}^k$,
    \[
    |\Psi_{\infty}(e^{-A T_{\text{\emph{ode}}}} U_{\infty})| \leq C |U_{\infty}|,
    \]
    where $C$ depends only on $C_{\text{\emph{rem}}}$, $\beta_{\text{\emph{rem}}}$ and $A$. Furthermore, if $V_{\infty} \in \mathbb{C}^k$ and $v$ is the solution of \eqref{ode} with $v(T_{\text{\emph{ode}}}) = \Psi_{\infty}(V_{\infty})$, then $v$ has asymptotic data given by $V_{\infty}$. 
\end{proposition}

\begin{proof}
Let $\mathcal{L}_n: \mathbb{C}^k \to E_{A,\beta_{\text{rem}}}^n$ be the map such that
\[
\mathcal{L}_n(\eta) = \pi_n(u_{\infty,n}),
\]
where $u_{\infty,n} \in E_{A,\beta_{\text{rem}}}^1 \oplus \cdots \oplus E_{A,\beta_{\text{rem}}}^n$ is the one appearing in the statement of Proposition \ref{ode asymptotic lemma} and $v$ is the solution with $v(T_{\text{ode}}) = \eta$ (note that these are not the same maps $L_j$ that were defined in the proof of Lemma \ref{specify ode data}). By Lemma \ref{specify ode data}, $\mathcal{L}_n$ is surjective for $n = 1, \ldots, \mathcal{N}$. From \eqref{u estimate}, we obtain
\[
|\mathcal{L}_n(\eta)| \leq C |v(T_{\text{ode}})| = C |\eta|
\]
for $n = 1, \ldots, \mathcal{N}$ and $C$ only depends on $C_{\text{rem}}$, $\beta_{\text{rem}}$ and $A$.

We are now ready to construct the isomorphism. For $\xi \in E_{A,\beta_{\text{rem}}}^n$, let 
\[
\widetilde{\Psi}_{\infty,n}(\xi) := \Psi_{\infty,n} (e^{-AT_{\text{ode}}} \xi),
\]
where $\Psi_{\infty,n}$ are the maps from Lemma \ref{specify ode data}. They satisfy the estimates
\[
|\widetilde{\Psi}_{\infty,n}(\xi)| \leq C |\xi|,
\]
where $C$ only depends on $C_{\text{rem}}$, $\beta_{\text{rem}}$ and $A$. The reason for working with these new maps instead of the $\Psi_{\infty,n}$, is that the new ones will allow us to obtain the desired estimate for the isomorphism. Say we want to specify $V_{\infty} \in \mathbb{C}^k$. Let $U_{\infty} = e^{A T_{\text{ode}}} V_{\infty}$. The idea is the following: the initial data $\widetilde{\Psi}_{\infty,1}(\pi_1(U_{\infty}))$ give rise to a solution $v_1$ with the correct first order asymptotic data $\pi_1(V_{\infty})$, but wrong second order asymptotic data, say $x = e^{-AT_{\text{ode}}}\mathcal{L}_2(\widetilde{\Psi}_{\infty,1}(\pi_1(U_{\infty})))$; then the initial data
\[
\widetilde{\Psi}_{\infty,2}(\pi_2(U_{\infty}) - \mathcal{L}_2(\widetilde{\Psi}_{\infty,1}(\pi_1(U_{\infty}))))
\]
give rise to a solution $v_2$ with zero first order asymptotic data, and second order asymptotic data $\pi_2(V_{\infty}) - x$; therefore, since asymptotic data can be added, the solution $v_1 + v_2$ has the correct first and second order asymptotic data; finally, we can iterate this process until we have the correct asymptotic data of all orders. Formally, define the vectors $D_n$ recursively by the formula
\[
D_n := \pi_n(U_{\infty}) - \mathcal{L}_n(\widetilde{\Psi}_{\infty,1}(D_1) + \cdots + \widetilde{\Psi}_{\infty,n-1}(D_{n-1}))
\]
for $n \geq 2$ and set $D_1 := \pi_1(U_{\infty})$. Using the estimates above, we see that
\[
|D_n| \leq |\pi_n(U_{\infty})| + C(|D_1| + \cdots + |D_{n-1}|)
\]
where $C$ depends only on $C_{\text{rem}}$, $\beta_{\text{rem}}$ and $A$. We can define the map
\[
    \widetilde{\Psi}_{\infty}(U_{\infty}) := \widetilde{\Psi}_{\infty,1}(D_1) + \widetilde{\Psi}_{\infty,2}(D_2) + \cdots + \widetilde{\Psi}_{\infty,\,\mathcal{N}}(D_{\mathcal{N}}).
\]
$\widetilde{\Psi}_{\infty}$ is linear since it is a sum of compositions of linear maps. Since we can add asymptotic data, the solution $v$ with initial data $v(T_{\text{ode}}) = \widetilde{\Psi}_{\infty}(U_{\infty})$ has asymptotic data $e^{-A T_{\text{ode}}} U_{\infty} = V_{\infty}$. Furthermore, $\widetilde{\Psi}_{\infty}$ is injective since the zero solution can only give rise to zero asymptotic data of all orders. Thus $\widetilde{\Psi}_{\infty}$ is a linear isomorphism. We estimate using the estimates for $\widetilde{\Psi}_{\infty,n}$ and the $D_n$,
\[
|\widetilde{\Psi}_{\infty}(U_{\infty})| \leq C(|\pi_1(U_{\infty})| + \cdots + |\pi_{\mathcal{N}}(U_{\infty})|) \leq C|U_{\infty}|
\]
where $C$ only depends on $C_{\text{rem}}$, $\beta_{\text{rem}}$ and $A$. Finally, we want our isomorphism to be defined on $V_{\infty}$. Let $\Psi_{\infty}(V_{\infty}) := \widetilde{\Psi}_{\infty}(e^{A T_{\text{ode}}} V_{\infty})$ for $V_{\infty} \in \mathbb{C}^k$. Then $\Psi_{\infty}:\mathbb{C}^k \to \mathbb{C}^k$ is a linear isomorphism and
\[
|\Psi_{\infty}(e^{-A T_{\text{ode}}} U_{\infty})| = |\widetilde{\Psi}_{\infty}(U_{\infty})| \leq C |U_{\infty}|. \qedhere
\]
\end{proof}

\begin{remark} \label{inhomogeneous}
    Proposition \ref{ode iso} can be used to obtain conclusions about the inhomogeneous equation \eqref{ode} when $\|F\|_A < \infty$. Let $v_0$ be the solution to \eqref{ode} with $v_0(T_{\text{ode}}) = 0$ and let $V_{0,\infty}$ be the corresponding asymptotic data. If we want to specify $V_{\infty} \in \mathbb{C}^k$, let $v_1$ be the solution to \eqref{ode} with $F = 0$ with initial data $v_1(T_{\text{ode}}) = \Psi_{\infty}(V_{\infty} - V_{0,\infty})$. Then $v := v_0 + v_1$ is a solution to \eqref{ode} and, by an argument similar to the proof of Lemma~\ref{add data}, it has asymptotic data $V_{\infty}$. To obtain estimates, let $U_{\infty} = e^{AT_{\text{ode}}} V_{\infty}$ and $U_{0,\infty} = e^{AT_{\text{ode}}} V_{0,\infty}$. Then
    \[
    \begin{split}
        |v(T_{\text{ode}})| &\leq |v_0(T_{\text{ode}})| + |v_1(T_{\text{ode}})|\\
        &= |\Psi_{\infty} (e^{-AT_{\text{ode}}}(U_{\infty} - U_{0,\infty}))|\\
        &\leq C|U_{\infty} - U_{0,\infty}|\\
        &\leq C(|U_{\infty}| + e^{\kappa_1 T_{\text{ode}}} \|F\|_A),
    \end{split}
    \]
    where $C$ only depends on $C_{\text{rem}}$, $\beta_{\text{rem}}$ and $A$.
\end{remark}

\section{Asymptotic analysis of weakly silent, balanced and convergent equations} \label{pde analysis}

We go back to \eqref{pde}, the equation of interest. Assume \eqref{pde} is weakly silent, balanced and convergent. Recall that these conditions on the equation ensure that the coefficients of the spatial derivative terms decay exponentially, and the coefficients of the time derivative terms converge exponentially as $t \to \infty$. Moreover, balance prevents solutions from growing super exponentially. Hence we expect to be able to approximate solutions of \eqref{pde} with solutions of 
\begin{equation} \label{model equation}
    \partial_t
    \begin{pmatrix}
        u\\
        u_t
    \end{pmatrix} = A_{\infty}
    \begin{pmatrix}
        u\\
        u_t
    \end{pmatrix} +
    \begin{pmatrix}
        0\\
        f
    \end{pmatrix},
\end{equation}
where $A_{\infty}$ is defined in connection to \eqref{ode mode equation}.

To perform the analysis, we apply the results obtained in Section \ref{ode analysis} to the ODE \eqref{mode equation}, satisfied by the Fourier modes of a solution $u$ to \eqref{pde}. Then we put together the mode by mode results to obtain information about $u$. The reason why it is possible to piece together the information on the modes into a statement for $u$ is that, in the results of Section \ref{ode analysis}, the constants $C$ appearing in the estimates are independent of $\iota$. 

In what follows, we obtain the same kind of results that were obtained in the ODE case. That is, we deduce asymptotic estimates of all orders for solutions of \eqref{pde}, and we prove that these solutions are completely characterized by the asymptotic data in the estimates. Then we finish this section by applying the results to a simple example.

\subsection{Asymptotic estimates}

When decomposing the equation in Fourier modes, the error matrix $A_{\text{rem}}$ arises,
\[
A_{\text{rem}} = 
\begin{pmatrix}
    0 & \, 0\\
    - \mathfrak{g}^2 I_m - in_lX^l + \zeta_{\infty} - \zeta & \, 2in_lg^{0l} I_m + \alpha_{\infty} - \alpha
\end{pmatrix}.
\]
Note that this matrix is related to the differential operator matrix
\[
D = 
\begin{pmatrix}
    0 & \, 0\\
    g^{jl}\partial_j \partial_l + \sum_r a_r^{-2} \Delta_{g_r} - X^j\partial_j + \zeta_{\infty} - \zeta & \, 2g^{0l} \partial_l + \alpha_{\infty} - \alpha
\end{pmatrix}
\]
through the Fourier decomposition. Before stating the result, we need to define the energies of interest. Let $\iota \in \mathcal{J}_B$ and $z(\iota)$ be a solution to \eqref{mode equation}. Then
\[
    \mathcal{E}_s(t,\iota) := \frac{1}{2} \langle \nu(\iota) \rangle^{2s} ( |\dot{z}(t,\iota)|^2 + \mathfrak{g}^2(t,\iota) |z(t,\iota)|^2 + |z(t,\iota)|^2 )
\]
for $s \in \mathbb{R}$. Additionally, if $u$ is a solution to \eqref{pde} and $z(\iota)$ are its Fourier modes, define
\[
\mathfrak{E}_s[u](t) :=  \sum_{\iota \in \mathcal{J}_B} \mathcal{E}_s (t,\iota) 
\]
for $s \in \mathbb{R}$. Recall that for every positive integer $n \leq \mathcal{N}$ we use the notation $E^n = E_{A_{\infty,\beta_{\text{rem}}}}^n$, where $\mathcal{N}$ is defined with respect to $\beta_{\text{rem}} = \min\{\mathfrak{b}_{\text{s}}, \eta_{\text{mn}}\}$ and $A_{\infty}$ as in Definition~\ref{generalized eigenspaces}, and $\mathfrak{b}_{\text{s}}$ and $\eta_{\text{mn}}$ come from the definitions of weak silence and weak convergence; cf. Definitions~\ref{silent} and \ref{weak convergence}.

\begin{theorem} \label{pde asymptotic lemma}
    Let \eqref{pde} be weakly silent, balanced and convergent. Furthermore, assume that $f$ is a smooth function such that for every $s \in \mathbb{R}$, 
    \begin{equation} \label{f condition}
        \|f\|_{A, s} := \int_0^{\infty} e^{-\kappa_1 \tau} \|f(\tau)\|_{(s)} d\tau < \infty
    \end{equation}
    where $\kappa_1 = \kappa_{\text{\emph{max}}}(A_{\infty})$. Then for every integer $n \geq 1$, there are non-negative constants $C$, $N$, $s_{\emph{hom}, n}$ and $s_{\emph{ih}, n}$, depending only on the coefficients of the equation, such that the following holds. If $u$ is a smooth solution to \eqref{pde}, there is a unique $V_{\infty,n} \in C^{\infty}(\bar{M},E^1 \oplus \cdots \oplus E^n)$ (if $n > \mathcal{N}$, then $V_{\infty,n} \in C^{\infty}(\bar{M},\mathbb{C}^{2m})$) such that
    \begin{equation} \label{pde estimate}
    \begin{split}
        \bigg\| 
        &\begin{pmatrix}
            u(t)\\
            u_t(t)
        \end{pmatrix} - F_{\infty, n}(t) \bigg\|_{(s)}\\
        &\leq C \langle t \rangle^N e^{(\kappa_1 - n\beta_{\emph{rem}})t} (\|u_t(0)\|_{(s + s_{\emph{hom}, n})} + \|u(0)\|_{(s + s_{\emph{hom}, n} + 1)} + \|f\|_{A, s + s_{\emph{ih}, n}})
    \end{split}
    \end{equation}
    for $t \geq 0$, where $F_{\infty,n}$ is given by the recursive formula
    \[
    F_{\infty,n}(t) = e^{A_{\infty}t} V_{\infty,n} + \int_0^t e^{A_{\infty} (t-\tau)} D F_{\infty,n-1}(\tau) d\tau + \int_0^t e^{A_{\infty}(t-\tau)}
        \begin{pmatrix}
            0\\
            f(\tau)
        \end{pmatrix}d\tau  
    \]
    and $F_{\infty,0} = 0$. Furthermore, 
    \begin{equation}
        \|V_{\infty,n}\|_{(s)} \leq C(\|u_t(0)\|_{(s + s_{\emph{hom}, n})} + \|u(0)\|_{(s + s_{\emph{hom}, n} + 1)} + \|f\|_{A, s + s_{\emph{ih}, n}}).
    \end{equation}
\end{theorem}

\begin{remark} \label{dependence of the constants}
    We have $s_{\text{hom},n} = s_{\text{hom},1} + 2(n-1)$ and $s_{\text{ih},n} = s_{\text{ih},1} + 2(n-1)$. Moreover, $s_{\text{hom},1} = s_{\text{hom}}$ and $s_{\text{ih},1} = s_{\text{ih}}$, where $s_{\text{hom}}$ and $s_{\text{ih}}$ are the ones appearing in \cite[Lemma~10.17]{hans}. Furthermore, similarly as in Proposition~\ref{ode asymptotic lemma}, the non-negative integer $N$ depends linearly on $n$.
\end{remark}

\begin{proof}
    By applying \cite[Lemma 10.17]{hans}, we obtain the following leading order estimate for a solution $u$ of \eqref{pde},
    \begin{equation}
    \begin{split}
        \bigg\| 
        \begin{pmatrix}
            u(t)\\
            u_t(t)
        \end{pmatrix} - e&^{A_{\infty}t} V_{\infty,1} - \int_0^t e^{A_{\infty}(t-\tau)} 
        \begin{pmatrix}
            0\\
            f(\tau)
        \end{pmatrix}d\tau
        \bigg\|_{(s)}\\
        &\leq C \langle t \rangle^N e^{(\kappa_1 - \beta_{\text{rem}})t} (\|u_t(0)\|_{(s + s_{\text{hom},1})} + \|u(0)\|_{(s + s_{\text{hom},1} + 1)} + \|f\|_{A, s + s_{\text{ih},1}}),
    \end{split}
    \end{equation}
    for $t \geq 0$ and $s \in \mathbb{R}$. Here $C, N$, $s_{\text{hom},1}$ and $s_{\text{ih},1}$ depend only on the coefficients of the equation. Furthermore, $V_{\infty,1} \in C^{\infty} (\bar{M},E^1)$ is unique and satisfies the estimate 
    \[
    \|V_{\infty,1}\|_{(s)} \leq C(\|u_t(0)\|_{(s + s_{\text{hom},1})} + \|u(0)\|_{(s + s_{\text{hom},1} + 1)} + \|f\|_{A, s + s_{\text{ih},1}}).
    \]
    This proves the statement for $n = 1$.
    
    For the inductive step, suppose we have an estimate like \eqref{pde estimate}. Furthermore, suppose the modes of $u$, $z(\iota) = \langle u, \varphi_{\iota} \rangle_B$, satisfy the estimates
    \begin{equation} \label{mode estimate}
    \begin{split}
        \bigg| v(t) - e^{A_{\infty}t} v_{\infty,n} - \int_0^t e^{A_{\infty}(t-\tau)} A_{\text{rem}} &\widehat{F}_{\infty,n-1}(\tau)d\tau - \int_0^t e^{A_{\infty}(t-\tau)} F(\tau)d\tau \bigg|\\
        &\leq C \langle t \rangle^N e^{(\kappa_1 - n \beta_{\text{rem}})t} (\mathcal{E}_{s_{\text{hom}, n}}^{1/2}(0) + \langle \nu(\iota) \rangle^{s_{\text{ih}, n}} \|\hat{f}\|_A) 
    \end{split}
    \end{equation}
    for $t \geq 0$; where $v(\iota) = (z(\iota),\dot{z}(\iota))$; $v_{\infty,n}(\iota) = \langle V_{\infty,n}, \varphi_{\iota} \rangle_B$; $\widehat{F}_{\infty, n-1}(\iota) = \langle F_{\infty, n-1}, \varphi_{\iota} \rangle_B$; $F(\iota) = (0,\hat{f}(\iota))$; $C$ depends only on $c_{\mathfrak{e}}$, $C_{\text{coeff}}$, $C_{\text{mn}}$, $\mathfrak{b}_{\text{s}}$, $\eta_{\text{mn}}$, $\eta_{\text{ode}}$, $A_{\infty}$, $g^{ij}(0)$ and $a_r(0)$; and $N$ depends only on $k$. Recall that the energy $\mathcal{E}_s$ was defined at the beginning of this subsection. Note that for $n = 1$, the estimates \eqref{mode estimate} correspond to  \cite[(10.42), p. 150]{hans}. We can apply Lemma \ref{asymptotic lemma} to $v$ by using this estimate and
    \[
    \|A_{\text{rem}}(t)\| \leq C \langle \nu(\iota) \rangle^2 e^{-\beta_{\text{rem}}t}
    \]
    for $t \geq 0$ (see \eqref{arem estimate 2}). In this case $T_{\text{ode}} = 0$, $\gamma = \langle \nu(\iota) \rangle^2$, $\alpha(0) = \mathcal{E}_{s_{\text{hom}, n}}^{1/2}(0) + \langle \nu(\iota) \rangle^{s_{\text{ih}, n}} \|\hat{f}\|_A$ and $H = \widehat{F}_{\infty,n}$. We obtain
    \begin{equation} \label{new mode estimate}
        \begin{split}
        \bigg| v(t) - e^{A_{\infty}t} &v_{\infty,n+1} - \int_0^t e^{A_{\infty}(t-\tau)} A_{\text{rem}} \widehat{F}_{\infty,n}(\tau)d\tau - \int_0^t e^{A_{\infty}(t-\tau)} F(\tau)d\tau \bigg|\\
        &\leq C \langle t \rangle^N e^{(\kappa_1 - (n+1)\beta_{\text{rem}})t} (|v(0)| + \langle \nu(\iota) \rangle^2 (\mathcal{E}_{s_{\text{hom}, n}}^{1/2}(0) + \langle \nu(\iota) \rangle^{s_{\text{ih}, n}} \|\hat{f}\|_A))\\
        &\leq C \langle t \rangle^N e^{(\kappa_1 - (n+1)\beta_{\text{rem}})t} (\mathcal{E}_{s_{\text{hom}, n} + 2}^{1/2}(0) + \langle \nu(\iota) \rangle^{s_{\text{ih}, n} + 2} \|\hat{f}\|_A)
    \end{split}
    \end{equation}
    for $t \geq 0$, where $C$ has the same dependence as in \eqref{mode estimate}. Moreover, $v_{\infty,n+1} \in E^1 \oplus \cdots \oplus E^{n+1}$ (if $n \geq \mathcal{N}$, then $v_{\infty, n+1} \in \mathbb{C}^{2m}$) satisfies
    \[
    |v_{\infty,n+1}| \leq C(\mathcal{E}_{s_{\text{hom}, n} + 2}^{1/2}(0) + \langle \nu(\iota) \rangle^{s_{\text{ih}, n} + 2} \|\hat{f}\|_A).
    \]
    Here we have used that $|v(0)| \leq \sqrt{2} \mathcal{E}_s^{1/2}(0)$ for any $s \geq 0$. If we let $s_{\text{hom}, n+1} = s_{\text{hom}, n} + 2$ and $s_{\text{ih}, n+1} = s_{\text{ih}, n} + 2$, \eqref{new mode estimate} corresponds to \eqref{mode estimate} for $n+1$. From \eqref{new mode estimate} we have
    \begin{equation} \label{new mode estmate 2}
    \begin{split}
        \langle \nu(\iota) \rangle^s \bigg| v(t) - e^{A_{\infty}t} &v_{\infty,n+1} - \int_0^t e^{A_{\infty}(t-\tau)} A_{\text{rem}} \widehat{F}_{\infty,n}(\tau)d\tau - \int_0^t e^{A_{\infty}(t-\tau)} F(\tau)d\tau \bigg|\\
        &\leq C  \langle t \rangle^N e^{(\kappa_1 - (n+1)\beta_{\text{rem}})t} (\mathcal{E}_{s + s_{\text{hom}, n + 1}}^{1/2}(0) + \langle \nu(\iota) \rangle^{s + s_{\text{ih}, n + 1}} \|\hat{f}\|_A)
    \end{split}
    \end{equation}
    for $s \in \mathbb{R}$. Now let
    \[
    V_{\infty,n+1} := \sum_{\iota} v_{\infty,n+1}(\iota) \varphi_{\iota}.
    \]
    Then
    \[
    \begin{split}
        \|V_{\infty,n+1}\|_{(s)} &= \bigg( \sum_{\iota} \langle \nu(\iota) \rangle^{2s} |v_{\infty,n+1}(\iota)|^2 \bigg)^{1/2}\\
        &\leq C\bigg(\sum_{\iota} \big(\mathcal{E}_{s + s_{\text{hom}, n+1}}(\iota,0) + \langle \nu(\iota) \rangle^{2(s + s_{\text{ih}, n + 1})} \|\hat{f}\|_A^2 \big) \bigg)^{1/2}\\
        &\leq C (\mathfrak{E}_{s + s_{\text{hom}, n+1}}^{1/2}[u](0) + \|f\|_{A, s_{\text{ih}, n + 1}})\\
        &\leq C( \|u_t(0)\|_{(s + s_{\text{hom}, n+1})} + \|u(0)\|_{(s + s_{\text{hom}, n+1} + 1)} + \|f\|_{A, s_{\text{ih}, n+1}})
    \end{split}
    \]
    and the right hand side is finite because of \eqref{f condition} and the smoothness of $u$. Consequently ${V_{\infty,n+1} \in C^{\infty}(\bar{M},E^1 \oplus \cdots \oplus E^{n+1})}$ (if $n \geq \mathcal{N}$, then $V_{\infty,n+1} \in C^{\infty}(\bar{M},\mathbb{C}^{2m})$). Finally, we can use \eqref{new mode estmate 2} to obtain
    \[
    \begin{split}
        \bigg\| 
        &\begin{pmatrix}
            u(t)\\
            u_t(t)
        \end{pmatrix} - e^{A_{\infty}t} V_{\infty,n+1} - \int_0^t e^{A_{\infty}(t-\tau)} D F_{\infty,n}(\tau)d\tau - \int_0^t e^{A_{\infty}(t-\tau)}
        \begin{pmatrix}
            0\\
            f(\tau)
        \end{pmatrix} d\tau
        \bigg\|_{(s)}\\
        &\leq C \langle t \rangle^N e^{(\kappa_1 - (n+1)\beta_{\text{rem}})t} (\|u_t(0)\|_{(s + s_{\text{hom}, n+1})} + \|u(0)\|_{(s + s_{\text{hom}, n+1} + 1)} + \|f\|_{A, s_{\text{ih}, n + 1}}).
    \end{split}
    \]
    Uniqueness of the $V_{\infty,n}$ is obtained in a similar way as in the proof of Lemma~\ref{ode asymptotic lemma}.
\end{proof}

\subsection{Specifying asymptotic data}

As we did in the ODE setting, from now on, we consider Equation \eqref{pde} with $f = 0$. In analogy with the discussion that led to Definition \ref{asymptotic data}, we are led to the definition of asymptotic data for $u$.

\begin{definition}
    Let \eqref{pde} be weakly silent, balanced and convergent. Let $u$ be a solution to \eqref{pde} with $f = 0$. Then, for $1 \leq n \leq \mathcal{N}$, the function $\pi_n (V_{\infty,n}) \in C^{\infty}(\bar{M}, E^n)$ is called the \emph{$n$-th order asymptotic data} for $u$.
\end{definition}
 
Assume that we have a solution $u$ for which the asymptotic data of orders 1 to $n-1$ vanishes. Then the $n$-th estimate becomes
\[
\bigg\| 
\begin{pmatrix}
    u(t)\\
    u_t(t)
\end{pmatrix} - e^{A_{\infty}t} V_{\infty,n} \bigg\|_{(s)} \leq C \langle t \rangle^N e^{(\kappa_1 - n\beta_{\text{rem}})t} ( \|u_t(0)\|_{(s + s_{\text{hom}, n})} + \|u(0)\|_{(s + s_{\text{hom}, n} + 1)} ),
\]
where $V_{\infty,n} \in C^{\infty}(\bar{M},E^n)$. We prove that in this situation, $V_{\infty,n}$ can be specified, given that we make one additional assumption on \eqref{pde}. 

\begin{lemma} \label{specify data pde}
    Let \eqref{pde} be weakly silent, balanced and convergent. Suppose that $f = 0$ and that there is a constant $\mathfrak{b}_{\text{\emph{low}}} > 0$ and a non-negative continuous function $\mathfrak{e}_{\text{\emph{low}}} \in L^1([0,\infty))$ such that
    \begin{equation} \label{low silent}
        \dot{\ell}(t) \geq -\mathfrak{b}_{\text{\emph{low}}} - \mathfrak{e}_{\text{\emph{low}}}(t)
    \end{equation}
    for $t \geq 0$ and $\iota \neq 0$. Then there is a linear injective map
    \[
    \Phi_{\infty,n}: C^{\infty}(\bar{M},E^n) \to C^{\infty}(\bar{M},\mathbb{C}^{2m})
    \]
    for $1 \leq n \leq \mathcal{N}$ such that the following holds. For $s \in \mathbb{R}$ and $\chi \in C^{\infty}(\bar{M},E^n)$, 
    \begin{equation}
        \|\Phi_{\infty,n}(\chi)\|_{(s)} \leq C\|\chi\|_{(s + s_{\infty})},
    \end{equation}
    where $C$ and $s_{\infty}$ only depend on the coefficients of the equation, $\mathfrak{b}_{\text{\emph{low}}}$ and $\|\mathfrak{e}_{\text{\emph{low}}}\|_1$. Furthermore, if $\chi \in C^{\infty}(\bar{M},E^n)$ and $u$ is the solution of \eqref{pde} such that
    \[
    \begin{pmatrix}
        u(0)\\
        u_t(0)
    \end{pmatrix} = \Phi_{\infty,n}(\chi),
    \]
    then
    \[
    \bigg\| 
\begin{pmatrix}
    u(t)\\
    u_t(t)
\end{pmatrix} - e^{A_{\infty}t} \chi \bigg\|_{(s)} \leq C \langle t \rangle^N e^{(\kappa_1 - n\beta_{\text{\emph{rem}}})t} ( \|u_t(0)\|_{(s + s_{\text{\emph{hom}}, n})} + \|u(0)\|_{(s + s_{\text{\emph{hom}}, n} + 1)} )
    \]
    for $t \geq 0$ and $s \in \mathbb{R}$, where $C$, $N$ and $s_{\text{\emph{hom}}, n}$ are as in Theorem \ref{pde asymptotic lemma}.
\end{lemma}

\begin{proof}
    Let $\chi \in C^{\infty}(\bar{M},E^n)$ and $\hat{\chi}(\iota) = \langle \chi, \varphi_{\iota} \rangle_B \in E^n$. Define $v(\iota) = (z(\iota), \dot{z}(\iota))$ as the solution to \eqref{mode equation} such that $v(T_{\text{ode}},\iota) = \Psi_{\infty,n}(\iota)(\hat{\chi}(\iota))$, where $\Psi_{\infty,n}(\iota)$ is the map from Lemma~\ref{specify ode data} associated to the equation of the $\iota$-th mode, and $T_{\text{ode}}$ is defined in Definition~\ref{silent}. Then
    \begin{equation} \label{chi mode estimate}
        |v(t) - e^{A_{\infty}t} \hat{\chi}| \leq C \langle \bar{t} \rangle^N e^{(\kappa_1 - n\beta_{\text{rem}}) \bar{t}} |v(T_{\text{ode}})|
    \end{equation}
    for $t \geq T_{\text{ode}}$, where $C$ only depends on $C_{\text{coeff}}$, $C_{\text{mn}}$, $\mathfrak{b}_{\text{s}}$, $\eta_{\text{mn}}$ and $A_{\infty}$. Now let $u_{\infty} = e^{A_{\infty} T_{\text{ode}}} \hat{\chi}$, then
    \[
    \begin{split}
        |v(T_{\text{ode}})| &= |\Psi_{\infty,n}(e^{-A_{\infty} T_{\text{ode}}} u_{\infty})|\\
        &\leq C|u_{\infty}|\\
        &= C|e^{A_{\infty} T_{\text{ode}}} \hat{\chi}(\iota)|\\
        &\leq C e^{(\kappa_1 + 1) T_{\text{ode}}} |\hat{\chi}(\iota)|,
    \end{split}
    \]
    where $C$ has the same dependence as above. Thus the assumption \eqref{low silent} implies
    \begin{equation} \label{energy}
        \mathcal{E}_s^{1/2}(0) \leq C \langle \nu(\iota) \rangle^{s + s_{\infty}} |\hat{\chi}(\iota)|,
    \end{equation}
    where $C$ and $s_{\infty}$ have dependence as in the statement of this lemma. For details on how to deduce this, see \cite[(10.50), p. 153]{hans} and the comments that follow. Now define
    \[
    \Phi_{\infty,n}(\chi) := \sum_{\iota} v(0,\iota) \varphi_{\iota}.
    \]
    Then by using \eqref{energy},
    \[
    \begin{split}
        \| \Phi_{\infty,n}(\chi) \|_{(s)} &= \bigg( \sum_{\iota} \langle \nu(\iota) \rangle^{2s} |v(0,\iota)|^2 \bigg)^{1/2}\\
        &\leq \sqrt{2} \bigg( \sum_{\iota} \mathcal{E}_s(0) \bigg)^{1/2}\\
        &\leq C \bigg( \sum_{\iota} \langle \nu(\iota) \rangle^{2(s + s_{\infty})} |\hat{\chi}(\iota)|^2 \bigg)^{1/2}\\
        &= C\|\chi\|_{(s + s_{\infty})}.
    \end{split}
    \]
    This estimate implies $\Phi_{\infty,n}(\chi) \in C^{\infty}(\bar{M},\mathbb{C}^{2m})$ and the injectivity comes from the injectivity of the maps $\Psi_{\infty,n}(\iota)$. Now we verify that $\Phi_{\infty,n}$ does what it is supposed to. Let $u$ be the solution of \eqref{pde} such that
    \[
    \begin{pmatrix}
        u(0)\\
        u_t(0)
    \end{pmatrix} = \Phi_{\infty,n}(\chi).
    \]
    By construction of $\Phi_{\infty,n}$, the $v(\iota)$ satisfy the estimates \eqref{chi mode estimate}, implying that they have vanishing asymptotic data of orders 1 to $n-1$. We verify what this says about $u$. The first estimate from Theorem~\ref{pde asymptotic lemma} for $u$ is
    \[
    \bigg\| 
    \begin{pmatrix}
        u(t)\\
        u_t(t)
    \end{pmatrix} - e^{A_{\infty}t} V_{\infty,1} \bigg\|_{(s)} \leq C \langle t \rangle^N e^{(\kappa_1 - \beta_{\text{rem}})t},
    \]
    where $C$ can depend on the solution. For the Fourier coefficients, this implies
    \[
    |v(t) - e^{A_{\infty}t} v_{\infty,1}| \leq C \langle t \rangle^N e^{(\kappa_1 - \beta_{\text{rem}})t}
    \]
    where $v_{\infty,1}(\iota) = \langle V_{\infty,1}, \varphi_{\iota} \rangle_B$. But we already know that $|v(t)| \leq C \langle t \rangle^N e^{(\kappa_1 - \beta_{\text{rem}})t}$ holds. Hence
    \[
    |e^{A_{\infty}t}v_{\infty,1}| \leq C \langle t \rangle^N e^{(\kappa_1 - \beta_{\text{rem}})t};
    \]
    but since $v_{\infty,1} \in E^1$, this can only be true if $v_{\infty,1} = 0$. Since this is true for all $\iota$, we conclude $V_{\infty,1} = 0$ and the second estimate for $u$ becomes
    \[
    \bigg\| 
    \begin{pmatrix}
        u(t)\\
        u_t(t)
    \end{pmatrix} - e^{A_{\infty}t} V_{\infty,2} \bigg\|_{(s)} \leq C \langle t \rangle^N e^{(\kappa_1 - 2\beta_{\text{rem}})t},
    \]
    for some $V_{\infty,2} \in C^{\infty}(\bar{M},E^2)$. Here we can apply the same argument to show that $V_{\infty,2} = 0$. By continuing this process, we eventually get $V_{\infty,j} = 0$ for $j = 1, \ldots, n-1$, so that the $n$-th estimate for $u$ becomes
    \[
    \bigg\| 
    \begin{pmatrix}
        u(t)\\
        u_t(t)
    \end{pmatrix} - e^{A_{\infty}t} V_{\infty,n} \bigg\|_{(s)} \leq C \langle t \rangle^N e^{(\kappa_1 - n\beta_{\text{rem}})t}
    \]
    for some $V_{\infty,n} \in C^{\infty}(\bar{M},E^n)$. This estimate implies
    \[
    |v(t) - e^{A_{\infty}t} v_{\infty,n}| \leq C \langle t \rangle^N e^{(\kappa_1 - n\beta_{\text{rem}})t},
    \]
    where $v_{\infty,n}(\iota) = \langle V_{\infty,n}, \varphi_{\iota} \rangle_B \in E^n$. But we already know that \eqref{chi mode estimate} holds, thus
    \[
    |e^{A_{\infty}t}(v_{\infty,n} - \hat{\chi})| \leq C \langle t \rangle^N e^{(\kappa_1 - n \beta_{\text{rem}})t};
    \]
    and since $v_{\infty,n} - \hat{\chi} \in E^n$, for this estimate to hold, we must have $v_{\infty,n} = \hat{\chi}$. Therefore $V_{\infty,n} = \chi$.
\end{proof}

\begin{remark} \label{silent remark}
    Regarding the condition \eqref{low silent}, there is a criterion similar to Lemma \ref{silent lemma} which ensures that \eqref{low silent} holds, and is more easily verifiable in practice. Namely, if there is a ${0 < \mu_+ \in \mathbb{R}}$ and a continuous non-negative function $\mathfrak{e}_+ \in L^1([0,\infty))$ such that the associated metric $g$ satisfies
    \[
    \bar{k} \leq (\mu_+ + \mathfrak{e}_+) \bar{g},
    \]
    then \eqref{low silent} holds. For details, see \cite[Remark 25.14]{hans}.
\end{remark}

\subsection{Homeomorphism between initial data and asymptotic data}

As we did in Section \ref{ode analysis}, we collect asymptotic data of all orders in a single function.

\begin{definition} \label{pde asymptotic data}
    Let \eqref{pde} be weakly silent, balanced and convergent, and assume that \eqref{f condition} holds. Let $u$ be a solution to \eqref{pde}. Define $\mathcal{V}_{\infty} \in C^{\infty}(\bar{M}, \mathbb{C}^{2m})$ by
    \[
    \mathcal{V}_{\infty} := V_{\infty,1} + \pi_2(V_{\infty,2}) + \cdots + \pi_{\mathcal{N}}(V_{\infty,\,\mathcal{N}}).
    \]
    Then $\mathcal{V}_{\infty}$ is called \emph{asymptotic data} for $u$.
\end{definition}

We show that under the conditions of Lemma \ref{specify data pde}, $\mathcal{V}_{\infty}$ uniquely determines $u$.

\begin{theorem} \label{homeomorphism}
    Let \eqref{pde} be weakly silent, balanced and convergent. Suppose that $f = 0$ and that there is a constant $\mathfrak{b}_{\text{\emph{low}}} > 0$ and a non-negative continuous function $\mathfrak{e}_{\text{\emph{low}}} \in L^1([0,\infty))$ such that \eqref{low silent} holds for $t \geq 0$ and $\iota \neq 0$. Then there is a linear isomorphism
    \[
    \Phi_{\infty}:C^{\infty}(\bar{M},\mathbb{C}^{2m}) \to C^{\infty}(\bar{M},\mathbb{C}^{2m}) 
    \]
    such that, for $\mathcal{V}_{\infty} \in C^{\infty}(\bar{M},\mathbb{C}^{2m})$, 
    \[
    \|\Phi_{\infty}(\mathcal{V}_{\infty})\|_{(s)} \leq C\|\mathcal{V}_{\infty}\|_{(s + \xi_{\infty})},
    \]
    where $C$ and $\xi_{\infty}$ only depend on the coefficients of the equation, $\mathfrak{b}_{\text{\emph{low}}}$ and $\|\mathfrak{e}_{\text{\emph{low}}}\|_1$. Furthermore, if $u$ is the solution of \eqref{pde} such that
    \[
    \begin{pmatrix}
        u(0)\\
        u_t(0)
    \end{pmatrix} = \Phi_{\infty}(\mathcal{V}_{\infty}),
    \]
    then $u$ has asymptotic data given by $\mathcal{V}_{\infty}$.
\end{theorem}

\begin{proof}
    First note that, by an argument similar to the one used for Lemma \ref{add data}, asymptotic data can be added. Let $L_n: C^{\infty}(\bar{M},\mathbb{C}^{2m}) \to C^{\infty}(\bar{M},E^n)$ be the map that takes initial data for \eqref{pde} at $t = 0$ to asymptotic data of order $n$. By Theorem \ref{pde asymptotic lemma}, it satisfies the estimate
    \[
    \begin{split}
        \|L_n(\eta_0,\eta_1)\|_{(s)} &\leq C(\|\eta_1\|_{(s + s_{\text{hom}, n})} + \|\eta_0\|_{(s + s_{\text{hom}, n} + 1)})\\
        &\leq C\|(\eta_0,\eta_1)\|_{(s + s_{\text{hom}, n} + 1)}
    \end{split}
    \]
    where $u(0) = \eta_0$ and $u_t(0) = \eta_1$. Suppose we want to specify $\mathcal{V}_{\infty} \in C^{\infty}(\bar{M},\mathbb{C}^{2m})$. Similarly to the proof of Proposition \ref{ode iso}, define $D_n$ by
    \[
    D_n := \pi_n(\mathcal{V}_{\infty}) - L_n(\Phi_{\infty,1}(D_1) + \cdots + \Phi_{\infty,n-1}(D_{n-1}))
    \]
    for $n \geq 2$ and let $D_1 := \pi_1(\mathcal{V}_{\infty})$, where the maps $\Phi_{\infty,n}$ are the ones from Lemma \ref{specify data pde}. Now define the map $\Phi_{\infty}$ by
    \[
    \Phi_{\infty}(\mathcal{V}_{\infty}) := \Phi_{\infty,1}(D_1) + \cdots + \Phi_{\infty,\,\mathcal{N}}(D_{\mathcal{N}}).
    \]
    Since we can add asymptotic data, the solution $u$ such that
    \[
    \begin{pmatrix}
        u(0)\\
        u_t(0)
    \end{pmatrix} = \Phi_{\infty}(\mathcal{V}_{\infty})
    \]
    has asymptotic data 
 given by $\mathcal{V}_{\infty}$. The map $\Phi_{\infty}$ is linear and injective, similarly as in the proof of Proposition \ref{ode iso}. Furthermore, by using the estimates for $L_n$ and $\Phi_{\infty,n}$, we have
    \[
    \|\Phi_{\infty}(\mathcal{V}_{\infty})\|_{(s)} \leq C \|\mathcal{V}_{\infty}\|_{(s + \xi_{\infty})}
    \]
    where $\xi_{\infty} = a_0 s_{\infty} + a_1 s_{\text{hom},1} + \cdots + a_{\mathcal{N}} s_{\text{hom},\,\mathcal{N}} + b$, for $a_i$ and $b$ some positive integers depending only on $\mathcal{N}$. It remains to show that $\Phi_{\infty}$ is surjective. Let $\psi \in C^{\infty}(\bar{M},\mathbb{C}^{2m})$ and let $u$ be the solution of \eqref{pde} such that
    \[
    \begin{pmatrix}
        u(0)\\
        u_t(0)
    \end{pmatrix} = \psi.
    \]
    Denote by $\mathcal{V}_{\infty} \in C^{\infty}(\bar{M},\mathbb{C}^{2m})$ the asymptotic data for $u$. Let $\bar{u}$ be the solution of \eqref{pde} such that 
    \[
    \begin{pmatrix}
        \bar{u}(0)\\
        \bar{u}_t(0)
    \end{pmatrix} = \Phi_{\infty}(\mathcal{V}_{\infty}),
    \]
    then $u - \bar{u}$ has vanishing asymptotic data, which for the Fourier coefficients, implies that $v - \bar{v}$ has vanishing asymptotic data for all $\iota$. By Proposition \ref{ode iso}, we conclude that $v = \bar{v}$ for all $\iota$. In particular $v(0,\iota) = \bar{v}(0,\iota)$ for all $\iota$, hence $u(0) = \bar{u}(0)$ and $u_t(0) = \bar{u}_t(0)$, implying $\Phi_{\infty}(\mathcal{V}_{\infty}) = \psi$.
\end{proof}

\begin{remark}
    Similarly to Remark~\ref{inhomogeneous}, Theorem~\ref{homeomorphism} can be used to obtain conclusions for \eqref{pde} when it is inhomogeneous.
\end{remark}

To finish this section we illustrate the results obtained by applying them to a simple example.

\begin{example} \label{the example}
Consider the equation
\begin{equation} \label{example}
    u_{tt} - e^{-2t}u_{\theta \theta} + u_t + e^{-t} u_{\theta} = 0
\end{equation}
on $\mathbb{R} \times S^1$. The metric associated with \eqref{example} is
\[
g = -dt \otimes dt + e^{2t} d\theta \otimes d\theta,
\]
thus $(\mathbb{R} \times S^1, g)$ is a canonical separable cosmological model manifold; cf. Definition~\ref{canonical}. The second fundamental form of the constant $t$ hypersurfaces is given by
\[
\bar{k}(\partial_{\theta},\partial_{\theta}) = g(\nabla_{\partial_{\theta}} \partial_t, \partial_{\theta}) = e^{2t}, 
\]
that is 
\[
\bar{k} = e^{2t} d\theta \otimes d\theta = \bar{g}.
\]
We also have $\mathcal{X} = e^{-t} \partial_{\theta}$. Thus $|\mathcal{X}|_{\bar{g}} = \sqrt{2}$ (note that in this case $\bar{h} = \bar{g}$). Note that $\alpha$ and $\zeta$ are constant, hence \eqref{example} is weakly convergent with $\eta_{\text{mn}}$ any positive real number. Since $\chi = 0$ in this case, we can apply Lemma~\ref{silent lemma} to conclude that \eqref{example} is weakly silent with $\mathfrak{b}_{\text{s}} = 1$, and weakly balanced. Moreover, $\bar{k} = \bar{g}$ means that we can also apply Remark \ref{silent remark}, so that \eqref{low silent} is satisfied. We see that $\beta_{\text{rem}} = 1$ and the hypotheses of Theorems \ref{pde asymptotic lemma} and \ref{homeomorphism} are satisfied. In this case,
\[
A_{\infty} = 
\begin{pmatrix}
    0 & \phantom{-}1\\
    0 & -1
\end{pmatrix},
\]
which has eigenvalues $0$ and $-1$, and the generalized eigenspaces are
\[
E^1 = \generated \bigg\{
\begin{pmatrix}
    1\\
    0
\end{pmatrix}
\bigg\}, \qquad E^2 = \generated \bigg\{ 
\begin{pmatrix}
    1\\
    -1
\end{pmatrix} \bigg\}.
\]
For $n = 1$, Theorem \ref{pde asymptotic lemma} yields 
\[
\bigg\| 
\begin{pmatrix}
    u\\
    u_t
\end{pmatrix} - 
\begin{pmatrix}
    u_{\infty}\\
    0
\end{pmatrix}
\bigg\|_{(s)} \leq C \langle t \rangle^N e^{-t},
\]
where $u_{\infty}$ is a function in $S^1$ and we allow $C$ to depend on the solution. Note that this is the estimate that would be obtained by applying \cite[Lemma 10.17]{hans}. Now we compute the $n = 2$ estimate.
\[
F_{\infty, 2}(t) = V_{\infty, 2} + \int_0^t e^{-A_{\infty}\tau} D 
\begin{pmatrix}
    u_{\infty}\\
    0
\end{pmatrix} d\tau,
\]
where the operator $D$ is given by
\[
D = \begin{pmatrix}
    0 & 0\\
    e^{-2t} \partial_{\theta}^2 - e^{-t} \partial_{\theta} & 0
\end{pmatrix}.
\]
Hence
\[
\begin{split}
    \int_0^t e^{-A_{\infty}\tau} D 
    \begin{pmatrix}
        u_{\infty}\\
        0
    \end{pmatrix} d\tau &= \int_0^t e^{-A_{\infty}\tau} 
    \begin{pmatrix}
        0\\
        e^{-2\tau} \partial_{\theta}^2 u_{\infty} - e^{-\tau} \partial_{\theta} u_{\infty}
    \end{pmatrix} d\tau\\
    &= \int_0^t (e^{-2\tau} \partial_{\theta}^2 u_{\infty} - e^{-\tau} \partial_{\theta} u_{\infty}) \bigg(
    \begin{pmatrix}
        1\\
        0
    \end{pmatrix} + e^{\tau}
    \begin{pmatrix}
        -1\\
        1
    \end{pmatrix} \bigg) d\tau\\
    &= -\frac{1}{2} (e^{-2t} - 1) \partial_{\theta}^2 u_{\infty}
    \begin{pmatrix}
        1\\
        0
    \end{pmatrix} - (e^{-t} - 1) \partial_{\theta}^2 u_{\infty}
    \begin{pmatrix}
        -1\\
        1
    \end{pmatrix}\\
    &\phantom{=} \; + (e^{-t} - 1) \partial_{\theta} u_{\infty}
    \begin{pmatrix}
        1\\
        0
    \end{pmatrix} + \partial_{\theta} u_{\infty} t 
    \begin{pmatrix}
        1\\
        -1
    \end{pmatrix}.
    \end{split}
\]
We obtain
\[\begin{split}
    \bigg\| 
\begin{pmatrix}
    u\\
    u_t
\end{pmatrix} - e^{A_{\infty}t} V_{\infty,2} &- \bigg(\frac{1}{2} \partial_{\theta}^2 u_{\infty} - \partial_{\theta} u_{\infty}\bigg) 
\begin{pmatrix}
    1\\
    0
\end{pmatrix} + e^{-2t} \partial_{\theta}^2 u_{\infty} 
\begin{pmatrix}
    -1/2\\
    1
\end{pmatrix}\\
&- e^{-t} \partial_{\theta}^2 u_{\infty} 
\begin{pmatrix}
    -1\\
    1
\end{pmatrix} - e^{-t} \partial_{\theta} u_{\infty}
\begin{pmatrix}
    1\\
    0
\end{pmatrix} - te^{-t} \partial_{\theta} u_{\infty} 
\begin{pmatrix}
    1\\
    -1
\end{pmatrix}
\bigg\|_{(s)} \leq C \langle t \rangle^N e^{-2t},
\end{split}
\]
where the term involving $e^{-2t}$ is an error term, so we throw it away. By splitting $V_{\infty,2}$ into its components relative to the eigenspaces of $A_{\infty}$, we can write
\[
V_{\infty,2} = \tilde{u}_{\infty} 
\begin{pmatrix}
    1\\
    0
\end{pmatrix} + \tilde{v}_{\infty}
\begin{pmatrix}
    1\\
    -1
\end{pmatrix},
\]
where $\tilde{u}_{\infty}$ and $\tilde{v}_{\infty}$ are some functions on $S^1$. The estimate becomes
\[
\begin{split}
    \bigg\| 
\begin{pmatrix}
    u\\
    u_t
\end{pmatrix} - \bigg(\tilde{u}_{\infty} + \frac{1}{2} \partial_{\theta}^2 &u_{\infty} - \partial_{\theta} u_{\infty}\bigg) 
\begin{pmatrix}
    1\\
    0
\end{pmatrix} - e^{-t} \tilde{v}_{\infty}  
\begin{pmatrix}
    1\\
    -1
\end{pmatrix}\\
&- e^{-t} \partial_{\theta}^2 u_{\infty} 
\begin{pmatrix}
    -1\\
    1
\end{pmatrix} - e^{-t} \partial_{\theta} u_{\infty}
\begin{pmatrix}
    1\\
    0
\end{pmatrix} - te^{-t} \partial_{\theta} u_{\infty} 
\begin{pmatrix}
    1\\
    -1
\end{pmatrix}
\bigg\|_{(s)} \leq C \langle t \rangle^N e^{-2t}.
\end{split}
\]
But we have uniqueness of $u_{\infty}$ from the $n = 1$ estimate. Therefore
\begin{equation} \label{data clarification}
    u_{\infty} = \tilde{u}_{\infty} + \frac{1}{2} \partial_{\theta}^2 u_{\infty} - \partial_{\theta} u_{\infty}.
\end{equation}
Hence we can write the estimate as,
\[
\bigg\|
\begin{pmatrix}
    u\\
    u_t
\end{pmatrix} - u_{\infty} 
\begin{pmatrix}
    1\\
    0
\end{pmatrix} - \tilde{v}_{\infty} e^{-t}
\begin{pmatrix}
    1\\
    -1
\end{pmatrix} - \partial_{\theta} u_{\infty} e^{-t}
\begin{pmatrix}
    t + 1\\
    -t
\end{pmatrix} + \partial_{\theta}^2 u_{\infty} e^{-t}
\begin{pmatrix}
    1\\
    -1
\end{pmatrix}
\bigg\|_{(s)} \leq C \langle t \rangle^N e^{-2t}.
\]
At this point it would seem like the term involving $\partial_{\theta}^2 u_{\infty}$ is redundant, since it has the same form as the term involving $\tilde{v}_{\infty}$. To make this precise, consider the map $L: C^{\infty}(S^1,\mathbb{R}^2) \to C^{\infty}(S^1,\mathbb{R}^2)$ defined by
\[
L \begin{pmatrix}
    u\\
    v
\end{pmatrix} :=
\begin{pmatrix}
    u\\
    v + \partial_{\theta}^2 u
\end{pmatrix}.
\]
Clearly it is linear and bijective. Also,
\[
\bigg\| L
\begin{pmatrix}
    u\\
    v
\end{pmatrix}
\bigg\|_{(s)} \leq \|u\|_{(s)} + \|v\|_{(s)} + \|\partial_{\theta}^2 u\|_{(s)} \leq \bigg\| 
\begin{pmatrix}
    u\\
    v
\end{pmatrix}
\bigg\|_{(s + 2)},
\]
so $L$ is continuous and similarly $L^{-1}$ is continuous. Note that the asymptotic data for $u$ is given by 
\[
\mathcal{V}_{\infty} = V_{\infty, 1} + \pi_2(V_{\infty, 2}) = u_{\infty} 
\begin{pmatrix}
    1\\
    0
\end{pmatrix} + \tilde{v}_{\infty} 
\begin{pmatrix}
    1\\
    -1
\end{pmatrix},
\]
and Theorem \ref{homeomorphism} implies that the solution $u$ is uniquely determined by the functions $u_{\infty}$ and $\tilde{v}_{\infty}$. Define 
\[
v_{\infty} := \tilde{v}_{\infty} - \partial_{\theta}^2 u_{\infty}.
\]
Then $L$ ensures that $u$ is also uniquely determined by $u_{\infty}$ and $v_{\infty}$. Since $L$ is a homeomorphism, $v_{\infty}$ can serve as second order asymptotic data instead of $\tilde{v}_{\infty}$. We finish by writing the estimate as
\[
\bigg\|
\begin{pmatrix}
    u\\
    u_t
\end{pmatrix} - u_{\infty} 
\begin{pmatrix}
    1\\
    0
\end{pmatrix} - v_{\infty} e^{-t}
\begin{pmatrix}
    1\\
    -1
\end{pmatrix} - \partial_{\theta} u_{\infty} e^{-t}
\begin{pmatrix}
    t + 1\\
    -t
\end{pmatrix} \bigg\|_{(s)} \leq C \langle t \rangle^N e^{-2t}.
\]
\end{example}

\begin{remark}
    This example serves to illustrate why our definition of asymptotic data makes sense, since \eqref{data clarification} shows explicitly how $\pi_1(V_{\infty,2})$ is completely determined by $\pi_1(V_{\infty,1}) = u_{\infty}$. If we consider the estimates for $n \geq 3$, similar expressions will arise for $\pi_1(V_{\infty,n})$ and $\pi_2(V_{\infty,n})$ in terms of $u_{\infty}$ and $v_{\infty}$; hence, as functions on $S^1$, the functions $F_{\infty,n}$ only depend on $u_{\infty}$, $v_{\infty}$ and their derivatives. 
\end{remark}

\section{Maxwell's equations in Kasner spacetimes} \label{maxwells equations}

In what follows, we apply the results obtained in Section \ref{pde analysis} to analyze solutions to the source free Maxwell's equations near the initial singularity of Kasner spacetimes. In particular, we study how inertial observers measure the energy density of such solutions near the big bang. After obtaining the dominant behavior by using Theorem \ref{pde asymptotic theorem}, we will be able to use Theorem~\ref{homeomorphism theorem} to ensure that this behavior is actually generic.  

Let $(M,g)$ be a spacetime. In relativistic electrodynamics, the electromagnetic field is given by a 2-form $F$, the Faraday tensor, which satisfies the following equations,
\begin{align*}
    \diver F & = -4 \pi J, \\
    dF & = 0. 
\end{align*}
Where $J$ is the charge density one form (see \cite{wald}). Since $dF = 0$, locally there is a one form $\omega$, the potential, such that $d\omega = F$. The Weizenböck formula for one forms, \cite[Corollary~7.21]{petersen}, reads
\begin{equation} \label{divergence of the differential}
    \diver d\omega = \Box_g \omega - \ric(\,\cdot\,,\omega^{\sharp}) - d\diver \omega
\end{equation}
Thus we obtain the following expression for the first equation in terms of $\omega$,
\[
\square_g \omega - \ric(\, \cdot \,, \omega^{\sharp}) - d\diver \omega = -4\pi J.
\]
Here
\[
\square_g \omega = \tr_g(\nabla^2) \omega = g^{ij}(\nabla_{\partial_i} \nabla_{\partial_j} \omega - \nabla_{\nabla_{\partial_i} \partial_j} \omega).
\]
Since with regards to $\omega$, we only care about the value of $d\omega$, we have the freedom to replace it by $\tilde{\omega} = \omega + du$ for any function $u \in C^{\infty}(M)$. We introduce the Lorenz gauge condition, which consists of choosing $u$ such that it satisfies
\[
\square_g u = -\diver \omega,
\]
so that $\diver \tilde{\omega} = 0$. This means that we can assume $\diver \omega = 0$ from the beginning, and the equation for $\omega$ becomes
\begin{equation} \label{lorenz gauge equation}
    \square_g \omega - \ric(\, \cdot \,, \omega^{\sharp}) = -4\pi J,
\end{equation}
which is a wave equation. In order to study the initial value problem for Maxwell's equations by using the wave equation above, we need to translate the gauge condition $\diver \omega = 0$ into constraint equations for the initial data.

\begin{lemma} \label{propagation of constraints}
    Let $(M,g)$ be a spacetime and $\omega \in \Omega^1(M)$ be a solution to \eqref{lorenz gauge equation}. If $\diver J = 0$, then $\Box_g \diver \omega = 0$.
\end{lemma}

\begin{proof}
    We compute using the Weizenböck formula for one forms, \cite[Corollary 7.21]{petersen},
    \[
    \diver( \Box_g \omega - \ric(\,\cdot\,,\omega^{\sharp}) ) = \delta(d \delta + \delta d)\omega = \delta d \delta \omega = \Box_g \diver \omega,
    \]
    where $\delta$ is the codifferenial. The lemma follows.
\end{proof}

Since charge is a conserved quantity, the charge density $J$ must satisfy $\diver J = 0$. Assume $\Sigma$ is a Cauchy surface in $(M,g)$ with unit normal $U$. By Lemma~\ref{propagation of constraints}, in order to ensure that a solution $\omega$ to \eqref{lorenz gauge equation} satisfies $\diver \omega = 0$, it is enough to ensure that $\diver \omega |_{\Sigma} = U \diver \omega |_{\Sigma} = 0$. But to translate these into constraints for initial data, we need an expression for $U \diver \omega$, which does not involve second $U$ derivatives of $\omega$.

\begin{lemma} \label{derivative of the divergence}
    Let $(M,g)$ be a spacetime, $\Sigma \subset M$ a spacelike hypersurface with unit normal $U$, and $\omega \in \Omega^1(M)$ a solution to \eqref{lorenz gauge equation}. Then
    \[
    U \diver \omega = -\diver(d\omega)(U)  - 4\pi J(U).
    \]
    In particular, if $\{e_1,\ldots,e_n\}$ is an orthonormal frame for $\Sigma$ with the induced metric,
    \[
    \diver(d\omega)(U) = \sum_i \Big( \nabla^2_{e_i,U} \omega(e_i) - \nabla^2_{e_i,e_i} \omega(U) \Big), 
    \]
    which does not contain second $U$ derivatives of $\omega$.
\end{lemma}

\begin{proof}
    This follows immediately from \eqref{divergence of the differential} and \eqref{lorenz gauge equation}.
\end{proof}

\begin{definition}
    Let $(M,g)$ be a spacetime and $\Sigma \subset M$ a Cauchy surface with unit normal $U$. Let $\omega \in \Omega^1(M)$ be a solution to \eqref{lorenz gauge equation}. Then the induced initial data on $\Sigma$, given by $\omega|_{\Sigma}$ and $\nabla_U \omega|_{\Sigma}$, satisfy the \emph{constraint equations} for Maxwell's equations in the Lorenz gauge if
    \begin{subequations} \label{constraint equations}
        \begin{align}
            \diver \omega|_{\Sigma} &= 0,\\
            \diver(d\omega)|_{\Sigma}(U) &= -4\pi J|_{\Sigma}(U).
        \end{align}
    \end{subequations}
\end{definition}

If the constraint equations \eqref{constraint equations} are satisfied, then Lemmas~\ref{propagation of constraints} and \ref{derivative of the divergence} imply that $\diver \omega = 0$ everywhere. Hence $d\omega$ is a solution to Maxwell's equations.

Below, we study solutions $F$ to Maxwell's equations in Kasner spacetimes, by studying solutions to the wave equation \eqref{lorenz gauge equation} satisfied by a potential $\omega$ in the Lorenz gauge. Of course, there is the question of how relevant such an analysis is for general $F$, given that in principle the potential $\omega$ only exists locally. However, for the Kasner spacetimes, this issue can be dealt with by using cohomology; see Remark~\ref{localisation argument} below for details. For this reason, henceforth we assume $F = d\omega$ globally.

Recall that the Kasner spacetimes are the spacetimes $(M,g)$, where $M = (0, \infty) \times \mathbb{T}^3$,
\[
g = -dt \otimes dt + \sum_i t^{2p_i} dx^i \otimes dx^i
\]
and $p_i$ are constants. We assume that they are vacuum, which in this case is equivalent to the Kasner relations,
\[
\sum_i p_i = \sum_i p_i^2 = 1,
\]
being satisfied. Note that if $p_i = 1$ for some $i$, the above conditions force $p_j = 0$ for $j \neq i$ and the corresponding spacetime is actually flat. Here we focus on the other possibility, that is, we assume $p_i < 1$ for all $i$. 

Since $\ric = 0$, the source free Maxwell's equations for a potential $\omega$ in the Lorenz gauge reduce to $\square_g \omega = 0$. In the natural coordinates of $(M,g)$, the equation becomes the system
\[
\begin{cases}
&\displaystyle -\partial_t^2 \omega_t + \sum_i t^{-2p_i} \partial_i^2 \omega_t - \frac{1}{t} \partial_t \omega_t - \sum_i 2p_i t^{-2p_i - 1} \partial_i \omega_i + \frac{1}{t^2} \omega_t = 0, \\[15pt]
&\displaystyle -\partial_t^2 \omega_j + \sum_i t^{-2p_i} \partial_i^2 \omega_j + \frac{2p_j - 1}{t} \partial_t \omega_j - \frac{2p_j}{t} \partial_j \omega_t = 0,
\end{cases}
\]
for $j = 1, 2, 3$. For our main results to be applicable, we need to make some modifications to the system. We introduce the time coordinate $\tau = -\ln t$; cf. the comments made in connection to Definition~\ref{canonical}. Note that with this new coordinate, the initial singularity happens at $\tau = \infty$. The system becomes
\begin{subnumcases}{\label{bad equation}}
&\label{bad equation 1}$\displaystyle -\partial_{\tau}^2 \omega_{\tau} + \sum_i e^{-2(1 - p_i) \tau} \partial_i^2 \omega_{\tau} - 2\partial_{\tau} \omega_{\tau} + \sum_i 2p_i e^{-2(1 - p_i) \tau} \partial_i \omega_i = 0,$\\ 
&$\displaystyle -\partial_{\tau}^2 \omega_j + \sum_i e^{-2(1 - p_i) \tau} \partial_i^2 \omega_j - 2p_j \partial_{\tau} \omega_j + 2p_j \partial_j \omega_{\tau} = 0,$
\end{subnumcases}
for $j = 1, 2, 3$. In the time coordinate $\tau$, the metric takes the form 
\[
g = -e^{-2\tau} d\tau \otimes d\tau + \sum_i e^{-2p_i \tau} dx^i \otimes dx^i;
\]
and the metric associated with \eqref{bad equation} is 
\begin{equation} \label{associated metric}
    \hat{g} = e^{2\tau}g = -d\tau \otimes d\tau + \sum_i e^{2(1 - p_i)\tau} dx^i \otimes dx^i,
\end{equation}
so that $(M,\hat{g})$ is a canonical separable cosmological model manifold. The system \eqref{bad equation} cannot satisfy the hypotheses of our results, since the coefficients of the spatial derivative terms do not all decay exponentially. In order to fix this, we consider the functions
\[
\widehat{\omega}_j := e^{-(1-p_j)\tau} \omega_j.
\]
For these, we obtain the system
\begin{subnumcases}{\label{good equation}}
&$ \displaystyle \partial^2_{\tau} \omega_{\tau} - \sum_i e^{-2(1-p_i)\tau} \partial^2_i \omega_{\tau} + 2\partial_{\tau} \omega_{\tau} - \sum_i 2p_i e^{-(1-p_i)\tau} \partial_i \widehat{\omega}_i = 0,$ \\
&$\displaystyle \partial^2_{\tau} \widehat{\omega}_j - \sum_i e^{-2(1-p_i)\tau} \partial^2_i \widehat{\omega}_j + 2\partial_{\tau} \widehat{\omega}_j - 2p_j e^{-(1-p_j)\tau} \partial_j \omega_{\tau} + (1-p_j^2) \widehat{\omega}_j = 0,$
\end{subnumcases}
for $j = 1, 2, 3$. It turns out to be the case that this system satisfies the conditions required to apply our results. Note that introducing the functions $\widehat{\omega}_j$, corresponds to considering the components of the one form $\omega$ in terms of an orthonormal dual frame for the metric $\hat{g}$.

\subsection{Asymptotic expansion for the potential of the electromagnetic field}

Our goal is to obtain information about the behavior of $\omega$ near $\tau = \infty$. Regarding the $p_i$, note that there is exactly one of them which is negative. From now on, we arrange them in ascending order; in that case $p_1 \in [-1/3,0)$, $p_2 \in (0,2/3]$ and $p_3 \in [2/3,1)$. We need to verify that \eqref{good equation} satisfies the hypotheses of Theorems \ref{pde asymptotic theorem} and \ref{homeomorphism theorem}. Recall that the metric associated with \eqref{good equation} is \eqref{associated metric}. We have that
\[
\hat{g}(\widehat{\nabla}_{\partial_i}\partial_{\tau}, \partial_i) = (1-p_i)e^{2(1-p_i)\tau},
\]
where $\widehat{\nabla}$ is the Levi-Civita connection of $\hat{g}$. So if $\bar{k}_{\hat{g}}$ is the second fundamental form of the $\mathbb{T}^3_{\tau} = \{\tau\} \times \mathbb{T}^3$ hypersurfaces with respect to $\hat{g}$, then
\[
\begin{split}
    \bar{k}_{\hat{g}} &= \sum_i (1-p_i)e^{2(1-p_i)\tau} dx^i \otimes dx^i,\\
    \bar{k}_{\hat{g}} &\geq (1-p_3) \bar{g}
\end{split}
\]
where $\bar{g}$ is the metric induced on $\mathbb{T}^3_{\tau}$ by $\hat{g}$. We also have 
\[
X^j(\tau) = 
\begin{pmatrix}
0 & \cdots & -2p_je^{-(1-p_j)\tau} & \cdots & 0 \\
\vdots & \phantom{0} & \phantom{0} & \phantom{0} & \vdots\\
-2p_je^{-(1-p_j)\tau} & \phantom{0} & \cdots & \phantom{0} & 0\\
\vdots & \phantom{0} & \phantom{0} & \phantom{0} & \vdots\\
0 & \phantom{0} & \cdots & \phantom{0} & 0
\end{pmatrix},
\]
hence
\[
|\mathcal{X}|_{\bar{g}} = \left( 2^3 \sum_i \hat{g}_{ii} ||X^i||^2 \right)^{1/2} \leq \left( 2^3 C \sum_i p_i^2 \right)^{1/2} = 2^{3/2} C^{1/2}
\]
for some constant $C$ (note that in this case $\bar{h} = \bar{g}$), thus $\mathcal{X}$ is $C^0$-future bounded. Since $\alpha$ and $\zeta$ are constant, \eqref{good equation} is weakly convergent with $\eta_{\text{mn}}$ any positive real number. Since $\chi = 0$, we conclude that \eqref{good equation} is $C^1$-silent with $\mu = 1-p_3$. Furthermore, we also have the bound $\bar{k}_{\hat{g}} \leq (1-p_1) \bar{g}$. We conclude that $\beta_{\text{rem}} = \mu = 1-p_3$ and \eqref{good equation} satisfies all the hypotheses of Theorems \ref{pde asymptotic theorem} and \ref{homeomorphism theorem}. 

We proceed by computing the estimates given by Theorem \ref{pde asymptotic theorem} for solutions of \eqref{good equation}. In the end, this will allow us to obtain an asymptotic expansion for $\omega$ as $\tau \to \infty$. We have
\[
A_{\infty} =
\begin{pmatrix}
\vphantom{\bigg(} 0 & \rvline & I_{4} \\
\hline
    \begin{matrix}
    0 & \phantom{0} & \phantom{0} & \phantom{0}\\
    \phantom{0} & p_1^2 - 1 & \phantom{0} & \phantom{0}\\
    \phantom{0} & \phantom{0} & p_2^2 - 1 & \phantom{0}\\
    \phantom{0} & \phantom{0} & \phantom{0} & p_3^2 - 1
    \end{matrix} &
\rvline & -2I_{4}
\end{pmatrix}.
\]
Note that, after an appropriate permutation of the rows and columns, $A_{\infty}$ can be written as a block diagonal matrix. Then $\det(A_{\infty} - \lambda I) = \lambda (2 + \lambda) \prod_i (\lambda^2 + 2\lambda + 1 - p_i^2)$ and the eigenvalues of $A_{\infty}$ are $0, -2, -1 \pm p_i$; hence $\kappa_1 = 0$. In order to compute the estimates, we need to figure out the terms of the form $e^{A_{\infty}\tau} V_{\infty,n}$. Note that these are solutions to the equation
\[
\frac{d}{d\tau} 
\begin{pmatrix}
    u\\
    \dot{u}
\end{pmatrix} = A_{\infty}
\begin{pmatrix}
    u\\
    \dot{u}
\end{pmatrix},
\]
where $u(\tau) \in \mathbb{R}^4$. But because of the form of $A_{\infty}$, the equation decouples into the following four individual equations,
\[
\begin{split}
    \frac{d}{d\tau}
    \begin{pmatrix}
        u_0\\
        \dot{u}_0
    \end{pmatrix} &= 
    \underbrace{\begin{pmatrix}
    0 & \phantom{-}1\\
    0 & -2
    \end{pmatrix}}_{A_0}
    \begin{pmatrix}
        u_0\\
        \dot{u}_0
    \end{pmatrix},\\
    \frac{d}{d\tau}
    \begin{pmatrix}
        u_i\\
        \dot{u}_i
    \end{pmatrix} &=
    \underbrace{\begin{pmatrix}
    0 & \phantom{-}1\\
    p_i^2 - 1 & -2
    \end{pmatrix}}_{A_i}
    \begin{pmatrix}
        u_i\\
        \dot{u}_i
    \end{pmatrix},
\end{split}
\]
where $u = (u_0,u_1,u_2,u_3)$. Also note that the differential operator $D$ is given by
\begin{equation} \label{D}
D
\begin{pmatrix}
    u\\
    \partial_{\tau} u
\end{pmatrix} = 
\begin{pmatrix}
    0\\
    \sum_i e^{-2(1-p_i)\tau} \partial_i^2 u - X^i(\tau) \partial_i u
\end{pmatrix},
\end{equation}
and the second component corresponds to the spatial derivative terms of the system. That is, if 
\[
\begin{split}
    D_0 u &= \sum_i (e^{-2(1-p_i)\tau} \partial_i^2 u_0 + 2p_i e^{-(1-p_i)\tau} \partial_i u_i),\\
    D_j u &= \sum_i e^{-2(1-p_i)\tau} \partial_i^2 u_j + 2p_j e^{-(1-p_j)\tau} \partial_j u_0,
\end{split}
\]
then $D (u,\partial_{\tau} u) = (0,0,0,0,D_0 u, D_1 u, D_2 u, D_3 u)$. Finally, note that the decomposition of $\mathbb{R}^8$ into eigenspaces of $A_{\infty}$ is compatible with the decomposition of $\mathbb{R}^2$ into eigenspaces of $A_0$ and the $A_i$. The eigenspaces of $A_{\infty}$ are given by
\[
\begin{split}
    E_{0, +} &= \generated \{ v_{0, +} = (1,0,0,0,0,0,0,0) \},\\
    E_{0, -} &= \generated \{ v_{0, -} = (1,0,0,0,-2,0,0,0) \},\\
    E_{1, \pm} &= \generated \{ v_{1, \pm} =(0,1,0,0,0,-1 \pm p_1,0,0)\},\\
    E_{2, \pm} &= \generated \{ v_{2, \pm} =(0,0,1,0,0,0,-1 \pm p_2,0)\},\\
    E_{3, \pm} &= \generated \{ v_{3, \pm} = (0,0,0,1,0,0,0,-1 \pm p_3)\}.
\end{split}
\]
For $v \in \mathbb{R}^8$ and $\mu = 0, \ldots, 3$, let $\mathcal{P}_{\mu}v \in \mathbb{R}^2$ have as first component the $\mu$-th component of $v$ and as second component the $(\mu + 4)$-th component of $v$. Then the eigenvector decomposition of $v$ is given by
\[
v = \sum_{\mu} (a_{\mu, +} v_{\mu,+} + a_{\mu,-} v_{\mu,-})  
\]
if and only if the eigenvector decompositions of $\mathcal{P}_0 v$ and $\mathcal{P}_i v$, with respect to $A_0$ and $A_i$ respectively, are given by
\[
\begin{split}
    \mathcal{P}_0 v &= a_{0,+} \mathcal{P}_0 v_{0,+} + a_{0,-} \mathcal{P}_0 v_{0,-} = 
    a_{0,+}
    \begin{pmatrix}
        1\\
        0
    \end{pmatrix} + a_{0,-}
    \begin{pmatrix}
        1\\
        -2
    \end{pmatrix},\\
    \mathcal{P}_i v &= a_{i,+} \mathcal{P}_i v_{i,+} + a_{i,-} \mathcal{P}_i v_{i,-} = 
    a_{i,+}
    \begin{pmatrix}
        1\\
        -1+p_i
    \end{pmatrix} + a_{i,-}
    \begin{pmatrix}
        1\\
        -1-p_i
    \end{pmatrix}.
\end{split}
\]
This implies that the way in which $e^{A_{\infty}\tau}$ acts on $v$ can be broken down in terms of $A_{\mu}$ as follows,
\[
\begin{split}
    \mathcal{P}_{\mu}(e^{A_{\infty}\tau}v) = e^{A_{\mu}\tau} \mathcal{P}_{\mu} v.
\end{split}
\]
From the previous discussion, we can conclude that the estimates \eqref{pde estimate 2} in this case can be broken down into the following estimates for $\omega_{\tau}$ and $\widehat{\omega}_i$,
\begin{equation} \label{broken estimate}
    \bigg \| 
\begin{pmatrix}
    \widehat{\omega}_{\mu}\\
    \partial_{\tau} \widehat{\omega}_{\mu}
\end{pmatrix} - e^{A_{\mu} \tau} \mathcal{P}_{\mu} V_{\infty, n} - \int_0^{\tau} e^{A_{\mu}(\tau - r)} 
\begin{pmatrix}
    0\\
    D_{\mu} \pi(F_{\infty, n-1}(r))
\end{pmatrix} dr
\bigg \|_{(s)} \leq C \langle \tau \rangle^N e^{-n(1-p_3)\tau},
\end{equation}
where $\widehat{\omega} = (\omega_{\tau},\widehat{\omega}_1, \widehat{\omega}_2, \widehat{\omega}_3)$, $\pi$ denotes projection onto the first four components and $C$ depends on the solution. 

Now we turn our attention to asymptotic data. If $\mathcal{V}_{\infty} \in C^{\infty}(\mathbb{T}^3, \mathbb{R}^8)$ is asymptotic data for the solution then, in terms of the eigenspaces of $A_{\infty}$, we have
\[
\mathcal{V}_{\infty} = \sum_{\mu} (u_{\mu,+} v_{\mu,+} + u_{\mu,-} v_{\mu,-}) 
\]
where $u_{\mu,\pm} \in C^{\infty}(\mathbb{T}^3)$. In particular
\begin{align*}
    \mathcal{P}_0 \mathcal{V}_{\infty} &= u_{0,+}
    \begin{pmatrix}
        1\\
        0
    \end{pmatrix} + u_{0,-}
    \begin{pmatrix}
        1\\
        -2
    \end{pmatrix},\\
    \mathcal{P}_i \mathcal{V}_{\infty} &= u_{i,+}
    \begin{pmatrix}
        1\\
        -1+p_i
    \end{pmatrix} + u_{i,-}
    \begin{pmatrix}
        1\\
        -1-p_i
    \end{pmatrix}.
\end{align*}
Hence all asymptotic data is contained in the functions $u_{\mu,\pm}$. Now we are ready to deduce an asymptotic expansion for the $\widehat{\omega}_{\mu}$. 

\begin{lemma}
    Let $\widehat{\omega} = (\omega_{\tau},\widehat{\omega}_1, \widehat{\omega}_2, \widehat{\omega}_3)$ be a smooth solution to \eqref{good equation}. Then for every positive integer $n$,
    \begin{equation} \label{hat expansion}
    \begin{split}
        &\bigg\|
        \begin{pmatrix}
            \omega_{\tau}\\
            \partial_{\tau} \omega_{\tau}
        \end{pmatrix} - 
        u_{0,+} \begin{pmatrix}
            1\\
            0
        \end{pmatrix} - \sum_{k = 1}^{q_{0n}} P_{0k}(\tau) e^{-a_{0k}\tau}
        \bigg\|_{(s)}\\
        &+ \sum_i \bigg\|
        \begin{pmatrix}
            \widehat{\omega}_i\\
            \partial_{\tau} \widehat{\omega}_i
        \end{pmatrix} - P_{i0}(\tau) e^{-(1-|p_i|)\tau} - \sum_{k = 1}^{q_{in}} P_{ik}(\tau) e^{-a_{ik}\tau}
        \bigg\|_{(s)} \leq C \langle \tau \rangle^N e^{-n(1-p_3)\tau},
    \end{split}
    \end{equation}
    where $q_{\mu n}$ are non-negative integers, non-decreasing in $n$ with $q_{\mu 1} = 0$; the $a_{\mu k}$ are positive real numbers satisfying $a_{jk} > 1 - |p_j|$, that only depend on the $p_i$ and are increasing in $k$; and the $P_{\alpha \beta}$ are polynomials in $\tau$ with coefficients in $C^{\infty}(\mathbb{T}^3, \mathbb{R}^2)$, whose components are functions in $C^{\infty}(\mathbb{T}^3)$ which only depend linearly on the $u_{\mu,\pm}$ and their derivatives. Furthermore, ${a_{01} = 2(1-p_3)}$, $a_{11} = \min\{ 1+p_1 + 2(1-p_3), 1-p_1 \}$ and $a_{i1} = 1-p_i + 2(1-p_3)$ for $i = 2, 3$.
\end{lemma}

\begin{proof}
    We prove the statement by induction. First note that the only eigenvalue of $A_{\infty}$ which falls in the interval $(-(1-p_3),0]$ is $0$. Thus $E^1 = E_{0,+}$, which implies
    \[
    V_{\infty, 1} = u_{0,+}(1,0,0,0,0,0,0,0).
    \]
    For $n = 1$, we use \eqref{broken estimate} to obtain
    \begin{equation} \label{n=1 estimate}
        \bigg\|
    \begin{pmatrix}
        \omega_{\tau}\\
        \partial_{\tau} \omega_{\tau}
    \end{pmatrix} - u_{0,+}
    \begin{pmatrix}
        1\\
        0
    \end{pmatrix}
    \bigg\|_{(s)} + \sum_i ( \| \widehat{\omega}_i \|_{(s)} + \|\partial_{\tau} \widehat{\omega}_i\|_{(s)} ) \leq C \langle \tau \rangle^N e^{-(1-p_3)\tau},
    \end{equation}
    which is of the required form for $n = 1$. Now assume that for some $n$, the estimates \eqref{broken estimate} take the form \eqref{hat expansion}. We begin with $\mu = 0$. If $Q_{\alpha \beta}$ denotes the first component of $P_{\alpha \beta}$,
    \[
    \begin{split}
        D_0 \pi(F_{\infty,n}(\tau)) &= \sum_i \left[e^{-2(1-p_i)\tau} \partial_i^2 u_{0,+} + \sum_{k = 1}^{q_{0n}} e^{-(2(1-p_i) + a_{0k})\tau} \partial_i^2 Q_{0k}(\tau)\right]\\
        &\phantom{=} +  2p_1 e^{-2\tau} \partial_1 Q_{10}(\tau) + \sum_{k = 1}^{q_{1n}} 2p_1 e^{-(1-p_1 + a_{1k})\tau} \partial_1 Q_{1k}(\tau) \\
        &\phantom{=} + \sum_{i = 2}^3 \left[ 2p_i e^{-2(1-p_i)\tau} \partial_i Q_{i0}(\tau) + \sum_{k = 1}^{q_{in}} 2p_i e^{-(1-p_i + a_{ik})\tau} \partial_i Q_{ik}(\tau) \right]\\
        &= \sum_{k = 1}^{q} P_k(\tau) e^{-b_k\tau},
    \end{split}
    \]
    where $q$ is an integer, $P_k(\tau)$ are polynomials with coefficients in $C^{\infty}(\mathbb{T}^3)$ and $b_k$ is an increasing sequence of reals with $b_1 = 2(1-p_3)$. To see this note that, since $a_{0 k} > 0$ and $a_{jk} > 1 - |p_j|$ for $j = 1, 2, 3$, the largest exponential on the right hand side of the first equality is $e^{-2(1-p_3) \tau}$. Now we compute the integral inside \eqref{broken estimate},
    \[
    \begin{split}
    \int_0^{\tau} e^{A_0(\tau - r)} 
    &\begin{pmatrix}
        0\\
        D_0 \pi(F_{\infty, n}(r))
    \end{pmatrix} dr\\
    &= \int_0^{\tau} D_0 \pi(F_{\infty, n}(r)) \bigg(
    \begin{pmatrix}
    1/2\\
    0
    \end{pmatrix} +
    e^{-2(\tau-r)} 
    \begin{pmatrix}
    -1/2\\
    1
    \end{pmatrix} \bigg) dr\\
    &= \sum_{k = 1}^q \bigg[ \int_0^{\tau} P_k(r)e^{-b_k r} dr 
    \begin{pmatrix}
        1/2\\
        0
    \end{pmatrix} + e^{-2\tau} \int_0^{\tau} P_k(r) e^{(2-b_k)r} dr
    \begin{pmatrix}
        -1/2\\
        1
    \end{pmatrix}
    \bigg].
    \end{split}
    \]
    Here we can use the fact that $\int_0^{\tau} e^{ar} P^d(r)dr = e^{a\tau} Q^d(\tau) + \alpha$, where $\alpha$ is a function of the spatial variables, and $P^d$ and $Q^d$ are polynomials of degree $d$, to write the result of the integration as
    \[
    \alpha 
    \begin{pmatrix}
        1/2\\
        0
    \end{pmatrix} + \beta e^{-2\tau}
    \begin{pmatrix}
        -1/2\\
        1
    \end{pmatrix}
    + \sum_{k = 1}^q \widetilde{P}_k(\tau) e^{-b_k \tau},
    \]
    for some $\alpha, \beta \in C^{\infty}(\mathbb{T}^3)$ and some polynomials $\widetilde{P}_k(\tau)$ with coefficients in $C^{\infty}(\mathbb{T}^3)$. As functions on $\mathbb{T}^3$, these only depend linearly on the $u_{\mu,\pm}$ and their derivatives. Also note that 
    \[
    e^{A_0\tau} \mathcal{P}_0 V_{\infty, n+1} = 
    \gamma 
    \begin{pmatrix}
        1\\
        0
    \end{pmatrix} + \delta e^{-2\tau}
    \begin{pmatrix}
        \phantom{-}1\\
        -2
    \end{pmatrix}
    \]
    for some $\gamma, \delta \in C^{\infty}(\mathbb{T}^3)$. By putting this information in \eqref{broken estimate} for $\mu = 0$ and using uniqueness of $u_{0,+}$, we obtain that the $(n+1)$-th estimate for $\omega_{\tau}$ has the desired form. The claim about $a_{01}$ comes from $b_1 = 2(1-p_3)$. The results for $\widehat{\omega}_i$ are obtained in a similar way.
\end{proof}

Now that we know the general form of the asymptotic expansion, we continue our analysis by computing the leading order terms for each $\widehat{\omega}_{\mu}$. The leading order behavior for $\omega_{\tau}$ is already given by \eqref{n=1 estimate}. Since $p_1 < 0$, the situation for $\widehat{\omega}_i$ is different for $i = 1$ and $i = 2, 3$. We proceed with $i = 2, 3$. Let $n$ be the first integer such that $n(1-p_3) > 1 - p_i$, then the relevant terms (for the computation) in the $(n-1)$-th estimate are
\[
\begin{split}
    \mathcal{P}_0 F_{\infty, n-1}(\tau) &= \left(u_{0,+} + \sum_{k = 1}^{q_{0(n-1)}} Q_{0k}(\tau) e^{-a_{0k}\tau}, \,*\, \right),\\
    \mathcal{P}_i F_{\infty, n-1}(\tau) &= (0,0),
\end{split}
\]
where the $*$ stands for a quantity which is not relevant for the computation. Now we compute the $n$-th estimate for $\widehat{\omega}_i$,
\[
D_i \pi(F_{\infty, n-1}(\tau)) = 2p_i e^{-(1-p_i)\tau} \partial_i u_{0,+} + \sum_{k = 1}^{q_{0(n-1)}} 2p_i \partial_i Q_{0k}(\tau) e^{-(1 - p_i +a_{0k})\tau}.
\]
Hence
\begin{align*}
&\int_0^{\tau} e^{A_i(\tau - r)} 
\begin{pmatrix}
    0\\
    D_i \pi(F_{\infty, n-1}(r))
\end{pmatrix} dr\\
&= \int_0^{\tau} D_i \pi(F_{\infty, n-1}(r)) \bigg( \frac{1}{2p_i} e^{-(1 + p_i)(\tau - r)}
\begin{pmatrix}
    -1\\
    1+p_i
\end{pmatrix} + \frac{1}{2p_i} e^{-(1-p_i)(\tau-r)} 
\begin{pmatrix}
    1\\
    -1+p_i
\end{pmatrix}
\bigg) dr\\
&= \partial_i u_{0,+} e^{-(1+p_i)\tau} \int_0^{\tau} e^{2p_ir} dr
\begin{pmatrix}
    -1\\
    1+p_i
\end{pmatrix} + \partial_i u_{0,+} e^{-(1-p_i)\tau} \int_0^{\tau} dr 
\begin{pmatrix}
    1\\
    -1+p_i
\end{pmatrix}\\
&\phantom{=} + \sum_{k = 1}^{q_{0(n-1)}} \bigg[ e^{-(1 + p_i)\tau} \int_0^{\tau} \partial_i Q_{0k}(r) e^{(2p_i - a_{0k})r} dr
\begin{pmatrix}
    -1\\
    1 + p_i
\end{pmatrix}\\
&\phantom{=} + e^{-(1-p_i)\tau} \int_0^{\tau} \partial_i Q_{0k}(r) e^{-a_{0k}r} dr
\begin{pmatrix}
    1\\
    -1+p_i
\end{pmatrix}
\bigg]\\
&= \bigg((\partial_i u_{0,+} \tau + \alpha)
\begin{pmatrix}
    1\\
    -1+p_i
\end{pmatrix} + \frac{1}{2p_i} \partial_i u_{0,+}
\begin{pmatrix}
    -1\\
    1+p_i
\end{pmatrix}
\bigg) 
 e^{-(1 - p_i)\tau} + \cdots,
\end{align*}
for some function $\alpha \in C^{\infty}(\mathbb{T}^3)$ which depends linearly on the $u_{\mu,\pm}$ and their derivatives. Here the dots stand for lower order terms. The leading order contribution from $V_{\infty, n}$ is 
\[
e^{A_i \tau} \mathcal{P}_i V_{\infty, n} = u_{i,+} e^{-(1-p_i)\tau}
\begin{pmatrix}
    1\\
    -1+p_i
\end{pmatrix} + \cdots,
\]
(here we know that the function appearing on the right hand side is $u_{i,+}$ since, by choice of $n$, $E_{i,+}$ is a subspace of $E^n$) so we conclude that
\[
P_{i0}(\tau) = \bigg( u_{i,+} + \partial_i u_{0,+} \tau + \alpha \bigg)
\begin{pmatrix}
    1\\
    -1+p_i
\end{pmatrix} +\frac{1}{2p_i} \partial_i u_{0,+}
\begin{pmatrix}
    -1\\
    1+p_i
\end{pmatrix}.
\]
If we let $u_i := u_{i,+} - (1/2p_i) \partial_i u_{0,+} + \alpha$, then
\[
P_{i0}(\tau) = u_i
\begin{pmatrix}
    1\\
    -1+p_i
\end{pmatrix} + \partial_i u_{0,+} 
\begin{pmatrix}
    \tau\\
    (-1+p_i)\tau + 1
\end{pmatrix}.
\]
Note that, by a translation in the space of asymptotic data (as in Example \ref{the example}), $u_i$ can replace $u_{i,+}$ as part of the asymptotic data for the solution. Similarly, we conclude that
\[
P_{10}(\tau) = \bigg( u_{1,-} + \frac{1}{2p_1} \partial_1 u_{0,+} + \alpha \bigg)
\begin{pmatrix}
    1\\
    -1-p_1
\end{pmatrix},
\]
and by letting $u_1 :=  u_{1,-} + (1/2p_1) \partial_1 u_{0,+} + \alpha$,
\[
P_{10}(\tau) = u_1
\begin{pmatrix}
    1\\
    -1-p_1
\end{pmatrix}.
\]
Here $u_1$ also can replace $u_{1,-}$ as part of the asymptotic data. To summarize, we can write the asymptotic expansion \eqref{hat expansion} as
\begin{equation} \label{hat expansion 2}
    \begin{split}
        &\bigg\|
        \begin{pmatrix}
            \omega_{\tau}\\
            \partial_{\tau} \omega_{\tau}
        \end{pmatrix} - 
        u_{0,+} \begin{pmatrix}
            1\\
            0
        \end{pmatrix} - \sum_{k = 1}^{q_{0n}} P_{0k}(\tau) e^{-a_{0k}\tau}
        \bigg\|_{(s)}\\
        &+ \bigg\|
        \begin{pmatrix}
            \widehat{\omega}_1\\
            \partial_{\tau} \widehat{\omega}_1
        \end{pmatrix} - u_1 e^{-(1+p_1)\tau}
        \begin{pmatrix}
            1\\
            -1-p_1
        \end{pmatrix} - \sum_{k = 1}^{q_{in}} P_{ik}(\tau) e^{-a_{ik}\tau}
        \bigg\|_{(s)}\\
        &+ \sum_{i=1}^2 \bigg\|
        \begin{pmatrix}
            \widehat{\omega}_i\\
            \partial_{\tau} \widehat{\omega}_i
        \end{pmatrix} - \bigg[u_i
        \begin{pmatrix}
            1\\
            -1+p_i
        \end{pmatrix} + \partial_i u_{0,+} 
        \begin{pmatrix}
            \tau\\
            (-1+p_i)\tau + 1
        \end{pmatrix}\bigg] e^{-(1-p_i)\tau} - \sum_{k = 1}^{q_{in}} P_{ik}(\tau) e^{-a_{ik}\tau}
        \bigg\|_{(s)}\\
        &\leq C \langle \tau \rangle^N e^{-n(1-p_3)\tau}.
    \end{split}
\end{equation}
Finally, we can go back to the original functions $\omega_i = e^{(1-p_i)\tau} \widehat{\omega}_i$. To do this, first note that the expansion for $\partial_{\tau} \widehat{\omega}_i$ is the $\tau$ derivative of the expansion for $\widehat{\omega}_i$ (see Remark~\ref{2nd component derivative of first}). If we let $(G,\partial_{\tau} G)$ denote the $n$-th approximation for $(\widehat{\omega}_i, \partial_{\tau} \widehat{\omega}_i)$, then
\[
\bigg\|
\begin{pmatrix}
    \widehat{\omega}_i\\
    \partial_{\tau} \widehat{\omega}_i
\end{pmatrix} - 
\begin{pmatrix}
    G\\
    \partial_{\tau} G
\end{pmatrix}
\bigg\|_{(s)} \leq C \langle \tau \rangle^N e^{-n(1-p_3)\tau}.
\]
In the estimate for the second component, we can expand $\partial_{\tau} ( e^{-(1-p_i)\tau} \omega_i )$ to obtain
\[
    \| -(1-p_i) \widehat{\omega}_i + e^{-(1-p_i)\tau} \partial_{\tau} \omega_i + (1-p_i) G - e^{-(1-p_i)\tau} \partial_{\tau}(e^{(1-p_i)\tau} G)\|_{(s)} \leq C \langle \tau \rangle^N e^{-n(1-p_3)\tau}.
\]
Therefore
\[    
\begin{split}
    e^{-(1-p_i)\tau} \| \partial_{\tau} \omega_i - \partial_{\tau} (e^{(1-p_i)\tau} G) \|_{(s)} &\leq C \langle \tau \rangle^N e^{-n(1-p_3)\tau} + (1-p_i) \| \widehat{\omega}_i - G \|_{(s)},\\
    \| \partial_{\tau} \omega_i - \partial_{\tau} (e^{(1-p_i)\tau} G) \|_{(s)} &\leq C \langle \tau \rangle^N e^{[1-p_i-n(1-p_3)]\tau}.
\end{split}
\]
Using this information and \eqref{hat expansion 2}, we can obtain the following result. In the statement below, $u = O(f(\tau))$ means that for every $s \in \mathbb{R}$, there is a constant $C_s$ such that
\[
\|u(\tau)\|_{(s)} \leq C_s |f(\tau)|.
\]

\begin{theorem} \label{omega expansion theorem}
Let $\omega = \omega_{\tau}d\tau + \omega_j dx^j$ be a solution of $\Box_g \omega = 0$, then
\begin{subequations} \label{omega expansion}
    \begin{align}
        \begin{pmatrix}
            \omega_{\tau}\\
            \partial_{\tau} \omega_{\tau}
        \end{pmatrix} &= u_{0,+}
        \begin{pmatrix}
            1\\
            0
        \end{pmatrix} + \sum_{k = 1}^{q_{0n}} P_{0k}(\tau) e^{-a_{0k}\tau} + O(\langle \tau \rangle^N e^{-n(1-p_3)\tau}),\\
        \begin{pmatrix}
            \omega_1\\
            \partial_{\tau} \omega_1
        \end{pmatrix} &= u_1 e^{-2p_1\tau}
        \begin{pmatrix}
            1\\
            -2p_1
        \end{pmatrix} + \sum_{k = 1}^{q_{1n}} P_{1k}(\tau) e^{-a_{1k}\tau} + O(\langle \tau \rangle^N e^{[1-p_1-n(1-p_3)]\tau}), \label{omega 1 expansion}\\
        \begin{pmatrix}
            \omega_i\\
            \partial_{\tau} \omega_i
        \end{pmatrix} &= 
        \begin{pmatrix}
            u_i + \partial_i u_{0,+} \tau\\
            \partial_i u_{0,+}
        \end{pmatrix} + \sum_{k = 1}^{q_{in}} P_{ik}(\tau) e^{-a_{ik}\tau} + O(\langle \tau \rangle^N e^{[1-p_i-n(1-p_3)]\tau})
    \end{align}
\end{subequations}
where $q_{\mu n}$ are non-negative integers, non-decreasing in $n$ with $q_{\mu 1} = 0$; and the $a_{\mu k}$ are positive real numbers that only depend on the $p_i$ and are increasing in $k$. Furthermore, the $P_{\alpha \beta}$ are polynomials in $\tau$ with coefficients in $C^{\infty}(\mathbb{T}^3, \mathbb{R}^2)$, whose components are functions in $C^{\infty}(\mathbb{T}^3)$ which only depend linearly on the $u_i$ for $i = 1,2,3$; on $u_{0,\pm}, u_{1,+}, u_{2,-}$ and $u_{3,-}$; and on their derivatives. Additionally, $a_{11} = \min\{ 2p_1 + 2(1-p_3), 0 \}$ and $ a_{01} = a_{i1} = 2(1-p_3)$ for $i = 2, 3$.
\end{theorem}

Note that the polynomials $P_{\alpha \beta}$ in the expansion above are not exactly the same as in the expansion \eqref{hat expansion 2}, they only coincide in the first component. Also, the numbers $a_{j k}$ for $j = 1,2,3$ are not the same as in \eqref{hat expansion 2}, they are the result of adding $-1+p_j$ to the previous ones.

\subsection{Generic blow up of the energy density} \label{proof of blow up}

Our objective now is to use the expansion \eqref{omega expansion} to obtain a blow up statement concerning the energy density of the electromagnetic field. The constraint equations \eqref{constraint equations} for the initial data at the hypersurface $\mathbb{T}_{\tau_0}^3$ are 
\begin{subequations} \label{constraint eqs initial data}
    \begin{align}
        -e^{2\tau_0} \partial_{\tau} \omega_{\tau}(\tau_0) + \sum_i e^{2p_i \tau_0} \partial_i \omega_i(\tau_0) &= 0,\\
        \sum_i e^{2p_i \tau_0} \partial_i ( \partial_{\tau} \omega_i(\tau_0) - \partial_i \omega_{\tau}(\tau_0) ) &= 0.
    \end{align}
\end{subequations}
Thus a solution $\omega$ of $\Box_g \omega = 0$ has $\diver \omega = 0$, if and only if its initial data satisfies \eqref{constraint eqs initial data}. To obtain constraint equations for the asymptotic data, recall that by Lemma~\ref{propagation of constraints}, $\diver \omega$ is a solution to the scalar wave equation, thus we can apply Theorem~\ref{pde asymptotic theorem} to obtain that there are functions $\alpha_{\infty}, \beta_{\infty} \in C^{\infty}(\mathbb{T}^3)$ such that
\[
\diver \omega = \alpha_{\infty} \tau + \beta_{\infty} + O(\langle \tau \rangle^N e^{-2(1-p_3)\tau}).
\]
Moreover, $\diver \omega$ is uniquely determined by $\alpha_{\infty}$ and $\beta_{\infty}$. In particular, $\diver \omega = 0$ if and only if 
\begin{equation} \label{constraint eqs asymptotic data}
    \alpha_{\infty} = \beta_{\infty} = 0,
\end{equation}
which are the asymptotic constraint equations. In principle, these two equations can be written as equations for the asymptotic data for $\omega$ by expanding $\diver \omega$ using Theorem~\ref{omega expansion theorem}, finding the terms corresponding to $\alpha_{\infty}$ and $\beta_{\infty}$, and setting those to zero. It should be possible to do this explicitly for any particular Kasner spacetime; but it is not possible to do it in general in this way, since it requires detailed knowledge of the lower order terms in Theorem~\ref{omega expansion theorem}, something which is not possible for arbitrary $p_i$. Nonetheless, this will not be an issue for our purposes. Note that the asymptotic constraint equations \eqref{constraint eqs asymptotic data} are linear differential equations in the asymptotic data for $\omega$. We make the following observation regarding the structure of the equations and the function $u_1$. Note that the differential operator $D$ only contains first and second derivatives (see \eqref{D}), so the only place where $u_1$ itself appears in the expansion \eqref{omega expansion} is in the leading order term of \eqref{omega 1 expansion}. In terms of $u_1$, the rest of the expansion can only depend on its derivatives. From this we conclude that the constraint equations \eqref{constraint eqs asymptotic data} involve derivatives of $u_1$, but not $u_1$ itself. This observation will be important below.

We now turn our attention to the energy. The electromagnetic stress-energy tensor is given by
\[
T_{\alpha \beta} = \frac{1}{4\pi} \bigg( F_{\alpha \gamma} F_{\beta}^{\phantom{a} \gamma} - \frac{1}{4} g_{\alpha \beta} F_{\mu \nu} F^{\mu \nu} \bigg),
\]
where $F = d\omega$ is the Faraday tensor. We want to study the behavior of the energy density along past directed timelike geodesics approaching the initial singularity. We prove the following result. Note that we go back to the $t$ time coordinate.

\begin{theorem} \label{generic blow up}
    Let $\mathcal{S} \subset C^{\infty}(\mathbb{T}^3, \mathbb{R}^8)$ be the space of solutions to the constraint equations \eqref{constraint eqs initial data} at $t_0 = 1$, equipped with the subspace topology. Then there is a subset $\mathcal{A} \subset \mathcal{S}$, open and dense in $\mathcal{S}$, such that, if $\omega = \omega_t dt + \omega_j dx^j$ solves $\Box_g \omega = 0$ with initial data 
    \[(\omega_t(1), \ldots,  \omega_3(1), \partial_t \omega_t(1), \ldots, \partial_t \omega_3(1)) \in \mathcal{A},
    \]
    the following holds. Let $T$ be the electromagnetic stress-energy tensor associated to $d\omega$, then for almost every timelike geodesic $\gamma$ (see Remark~\ref{almost every}), the energy density $T(\dot{\gamma}, \dot{\gamma})$ along $\gamma$ blows up like $t^{-(2p_2 + 4p_3)}$ as $t \to 0$.
\end{theorem}

\begin{proof}
We start by obtaining some information about the geodesics. Let $\gamma$ be a geodesic. Then
\[
\frac{d}{ds} g(\dot{\gamma}, \partial_i) = g(\dot{\gamma}, \nabla_{\dot{\gamma}} \partial_i) = \frac{1}{2} \mathcal{L}_{\partial_i} g(\dot{\gamma}, \dot{\gamma}) = 0,
\]
since $\partial_i$ is a Killing vector field. Hence
\[
g(\dot{\gamma}, \partial_i) = \dot{x}^i t^{2p_i} = c_i,
\]
where $c_i$ is some constant. If  $\gamma$ is past directed unit timelike, we can use this fact to write $\dot{\gamma}$ as
\begin{equation} \label{tangent vector}
    \dot{\gamma} = -\bigg(1 + \sum_i c_i^2 t^{-2p_i}\bigg)^{1/2} \partial_t + \sum_i c_i t^{-2p_i} \partial_i.
\end{equation}
Hence the energy density measured along $\gamma$ is given by
\[
T(\dot{\gamma}, \dot{\gamma}) = \bigg( 1 + \sum_i c_i^2 t^{-2p_i} \bigg) T_{tt} - \sum_i 2 c_i \bigg( 1 + \sum_j c_j^2 t^{-2p_j} \bigg)^{1/2}t^{-2p_i} T_{ti} + \sum_{ij} c_i c_j t^{-2(p_i + p_j)} T_{ij}.
\]
To obtain the leading order behavior of the energy density, we begin by obtaining the leading order behavior of the components of the Faraday tensor $F = d\omega$. From \eqref{omega expansion}, we get (note the use of $\tau$ time)
\begin{align*}
    F_{\tau 1} &= \partial_{\tau} \omega_1 - \partial_1 \omega_{\tau} = -2p_1 u_1 e^{-2p_1\tau} + O( \langle \tau \rangle^N e^{-\min\{ 2p_1 + 2(1-p_3), 0 \} \tau}),\\
    F_{\tau i} &= \partial_{\tau} \omega_i - \partial_i \omega_{\tau} = O(\langle \tau \rangle^N e^{-2(1-p_3)\tau}),\\
    F_{1 i} &= \partial_1 \omega_i - \partial_i \omega_1 = - \partial_i u_1 e^{-2p_1\tau} + O( \langle \tau \rangle^N e^{-\min\{ 2p_1 + 2(1-p_3), 0 \} \tau}),\\
    F_{2 3} &= \partial_2 \omega_3 - \partial_3 \omega_2 = \partial_2 u_3 - \partial_3 u_2 + O( \langle \tau \rangle^N e^{-2(1-p_3)\tau}).
\end{align*}
Now we look at the stress-energy tensor,
\[
\begin{split}
    T_{\tau \tau} &= \frac{1}{8\pi} \Bigg( \sum_i  e^{2p_i\tau} F_{\tau i}^2 + \sum_{i<j} F_{ij}^2 e^{2(p_i + p_j - 1)} \Bigg),\\
    T_{ii} &= \frac{1}{8\pi}( -e^{2\tau} F_{i \tau}^2 + e^{2p_j\tau} F_{ij}^2 + e^{2p_k\tau}F_{ik}^2\\
    &\phantom{=} + e^{2(1+p_j - p_i)\tau} F_{\tau j}^2 + e^{2(1+p_k - p_i)\tau} F_{\tau k}^2 - e^{2(p_j + p_k - p_i)\tau} F_{jk}^2 ),\\
    T_{\tau i} &= \frac{1}{4\pi} ( e^{2p_j \tau} F_{\tau j} F_{ij} + e^{2p_k\tau} F_{\tau k} F_{ik} ),\\
    T_{ij} &= \frac{1}{4\pi} ( -e^{2\tau} F_{i \tau} F_{j \tau} + e^{2p_k\tau} F_{ik} F_{jk} );
\end{split}
\]
where $i$, $j$ and $k$ are distinct. By noting that $T_{ti} = -e^{\tau}T_{\tau i}$ and $T_{tt} = e^{2\tau}T_{\tau \tau}$, we can now expand $T(\dot{\gamma},\dot{\gamma})$. We find that the leading order terms of $T(\dot{\gamma}, \dot{\gamma})$ have to be of order $e^{(2p_2 + 4p_3)\tau}$ or $e^{(8p_3-2)\tau}$. But the terms of order $e^{(8p_3-2)\tau}$ come from the terms involving $F_{\tau 3}$ in $T_{33}$ and $T_{\tau \tau}$, which happen to have opposite signs. So these terms cancel out identically, and we are left with
\[
T(\dot{\gamma}, \dot{\gamma}) = \frac{\delta c_2^2 + c_3^2}{4\pi} (4p_1^2 u_1^2 + (\partial_2 u_3 - \partial_3 u_2)^2) e^{(2p_2 + 4p_3)\tau} + \cdots
\]
as leading order behavior, where $\delta = 1$ in the case when $p_1 = -1/3$ and ${p_2 = p_3 = 2/3}$, and $\delta = 0$ otherwise.

In order to prove the genericity statement, we will exploit the fact that $u_1^2$ appears in the expression above. Let $\Phi_{\infty}$ be the homeomorphism between asymptotic data and initial data as given in Theorem \ref{homeomorphism theorem}. Remember that asymptotic data is codified in the functions $u_i$ for $i = 1, 2, 3$ and $u_{0,\pm}, u_{1,+}, u_{2,-}, u_{3,-}$. Since we are interested exclusively in $u_1$, we think of the space of asymptotic data as 
\[
C^{\infty}(\mathbb{T}^3, \mathbb{R}^8) = C^{\infty}(\mathbb{T}^3) \times C^{\infty}(\mathbb{T}^3, \mathbb{R}^7),
\]
where the first component corresponds to the function $u_1$. Define the set
\[
\mathcal{C} := \{ \varphi \in C^1(\mathbb{T}^3) : 0 \; \text{is a regular value of} \; \varphi \}.
\]
Then $\mathcal{C}$ is an open subset of $C^1(\mathbb{T}^3)$. Note that the inclusion map
\[
i: C^{\infty}(\mathbb{T}^3) \to C^1(\mathbb{T}^3)
\]
is continuous, hence $i^{-1}(\mathcal{C})$ is open in $C^{\infty}(\mathbb{T}^3)$. Let $\mathcal{S}_{\infty} \subset C^{\infty}(\mathbb{T}^3, \mathbb{R}^8)$ be the space of solutions to the asymptotic constraint equations \eqref{constraint eqs asymptotic data}. Clearly $\mathcal{S}_{\infty} \cap (i^{-1}(\mathcal{C}) \times C^{\infty}(\mathbb{T}^3, \mathbb{R}^7))$ is open in $\mathcal{S}_{\infty}$. We will show that it is also dense in $\mathcal{S}_{\infty}$. Let $\chi = (\chi_1, \chi_2) \in \mathcal{S}_{\infty}$ with $\chi_1 \in C^{\infty}(\mathbb{T}^3)$ and $\chi_2 \in C^{\infty}(\mathbb{T}^3, \mathbb{R}^7)$. By Sard's theorem, the set of $a \in \mathbb{R}$ such that $-a$ is a critical value of $\chi_1$ has measure zero. Hence $\chi_1 + a \in i^{-1}(\mathcal{C})$ for almost every $a$. Moreover, since the constraint equations \eqref{constraint eqs asymptotic data} do not involve $u_1$ itself, but only its derivatives, then $\chi + (a,0) \in \mathcal{S}_{\infty}$ for all $a \in \mathbb{R}$. This proves that $\chi$ is in the closure of  $\mathcal{S}_{\infty} \cap (i^{-1}(\mathcal{C}) \times C^{\infty}(\mathbb{T}^3, \mathbb{R}^7))$ in $\mathcal{S}_{\infty}$. Now we can define
\[
\mathcal{A} := \Phi_{\infty}(\mathcal{S}_{\infty} \cap (i^{-1}(\mathcal{C}) \times C^{\infty}(\mathbb{T}^3, \mathbb{R}^7))).
\]
Note that, since the conditions \eqref{constraint eqs asymptotic data} and \eqref{constraint eqs initial data} are equivalent, we must have $\Phi_{\infty}(\mathcal{S}_{\infty}) = \mathcal{S}$, so that $\mathcal{A} \subset \mathcal{S}$. Moreover, since $\Phi_{\infty}$ is a homeomorphism, $\mathcal{A}$ is open and dense in $\mathcal{S}$. Here we have argued as in \cite[Section 17.3]{hans2019} for obtaining $\mathcal{A}$. If $\omega$ arises from initial data in $\mathcal{A}$, we have that $u_1^{-1}(0)$ is a surface in $\mathbb{T}^3$. 

To complete the blow up statement, we now need to ensure that the set of timelike geodesics approaching this surface as $t \to 0$ has measure zero in the sense of Remark~\ref{almost every}.  

Take a curve $\gamma(\tau) = (e^{-\tau}, x(\tau))$. If $\gamma$ is causal, then
\[
g(\dot{\gamma}, \dot{\gamma}) = -e^{-2\tau} + \sum_i e^{-2p_i \tau} (\dot{x}^i)^2 \leq 0,
\]
implying
\[
|\dot{x}^i| \leq e^{-(1-p_i)\tau}.
\]
Let $\tau_j \leq \tau_k$, then
\begin{equation} \label{causal estimate}
    |x^i(\tau_k) - x^i(\tau_j)| \leq \int_{\tau_j}^{\tau_k} e^{-(1-p_i)s}ds \leq \frac{1}{1-p_i} e^{-(1-p_i)\tau_{j}}.
\end{equation}
This means that if we take a sequence $\tau_k \to \infty$, then $x^i(\tau_k)$ is Cauchy, so it converges to some $q_i$. Since a causal curve can always be parametrized by $\tau$, we have thus shown that $x(\tau)$ converges to some $q \in \mathbb{T}^3$ as $\tau \to \infty$ for all causal curves $\gamma$.

Fix a $t_0$ and consider the hypersurface $\mathbb{T}^3_{t_0}$. We want to consider the map which takes past directed timelike geodesics starting at $\mathbb{T}^3_{t_0}$ to the point in $\mathbb{T}^3$ to which their spatial components converge as $t \to 0$. For that purpose, let $P$ be the set of past directed unit timelike vectors with base point in $\mathbb{T}^3_{t_0}$, then $P$ is a 6 dimensional submanifold of $TM$. Define the map $F: P \to \mathbb{T}^3$ by
\[
F(v) = \lim_{t \to 0} x(\gamma_v(s(t))),
\]
where $\gamma_v$ is the geodesic with initial velocity $v$ and $s$ is proper time. Our objective is to show that $F^{-1}(u_1^{-1}(0))$ has measure zero in $P$.

If $v \in P$, then \eqref{tangent vector} for $\gamma_v$ at $t = t_0$ gives us a parametrization of $v$ by the $c_i$ and the coordinates $x^i$, so that we can use $(c_i, x^i)$ as a coordinate system for $P$. Furthermore, if $\alpha$ is the reparametrization of $\gamma_v$ by the $t$ coordinate, then
\[
\dot{\alpha} = \partial_t - \sum_i \frac{c_i t^{-2p_i}}{\big( 1 + \sum_j c_j^2 t^{-2p_j} \big)^{1/2}} \partial_i.
\]
We can now integrate the spatial components of $\dot{\alpha}$ to obtain an expression for the map $F$, namely
\[
x^i(F(v)) = x_0^i + \int_0^{t_0} \frac{c_i t^{-2p_i}}{\big( 1 + \sum_j c_j^2 t^{-2p_j} \big)^{1/2}} dt,
\]
where $v \in P$ has coordinates $(c_i,x_0^i)$. We shall use this expression to show that $F$ is almost everywhere differentiable. First note that the inequality
\begin{equation} \label{nice inequality}
    \frac{|c_i| t^{-2p_i}}{\big( 1 + \sum_j c_j^2 t^{-2p_j} \big)^{1/2}} \leq t^{-p_i}
\end{equation}
holds for $t > 0$. Since $p_i < 1$, then $t^{-p_i}$ is integrable on $[0,t_0]$ and by Lebesgue's dominated convergence theorem, $F$ is continuous as a function of the $c_i$ variables, hence $F$ is continuous. We now proceed to prove existence of the partial derivatives. We have
\[
\frac{\partial}{\partial c_i} \left( 
\frac{c_i t^{-2p_i}}{\big( 1 + \sum_j c_j^2 t^{-2p_j} \big)^{1/2}} \right) = \frac{t^{-2p_i} (1 + \sum_{j \neq i} c_j^2 t^{-2p_j})}{(1 + \sum_j c_j^2 t^{-2p_j})^{3/2}}.
\]
If $c_i$ happens to be zero, then this expression will not be integrable on $[0,t_0]$ in general. Take an interval $[a,b]$ not containing zero. Then if $a$ and $b$ are positive, we can use \eqref{nice inequality} to obtain
\[
\left| \frac{\partial}{\partial c_i} \left( 
\frac{c_i t^{-2p_i}}{\big( 1 + \sum_j c_j^2 t^{-2p_j} \big)^{1/2}} \right) \right| \leq \frac{1}{a} t^{-p_i},
\]
for all $c_i \in (a,b)$, which is integrable on $[0,t_0]$. There is a similar bound for $a$ and $b$ negative. Thus by differentiating under the integral sign, we obtain that $\partial_{c_i} F_i$ exists for $c_i \in (a,b)$, that is for $c_i \neq 0$. Moreover, the bound implies that $\partial_{c_i} F_i$ is also continuous by the dominated convergence theorem. Now let $k \neq i$. Then
\[
\frac{\partial}{\partial c_k} \left( 
\frac{c_i t^{-2p_i}}{\big( 1 + \sum_j c_j^2 t^{-2p_j} \big)^{1/2}} \right) = -\frac{c_i c_k t^{-2(p_i + p_k)}}{ \big( 1 + \sum_j c_j^2 t^{-2p_j} \big)^{3/2} }.
\]
Again, if $[a,b]$ does not contain zero and $a$ and $b$ are positive, then
\[
\left| \frac{\partial}{\partial c_k} \left( 
\frac{c_i t^{-2p_i}}{\big( 1 + \sum_j c_j^2 t^{-2p_j} \big)^{1/2}} \right) \right| \leq \frac{1}{a} t^{-p_k}
\]
for all $c_i \in (a,b)$. There is again a similar bound for $a$ and $b$ negative. We conclude that $\partial_{c_k} F_i$ exists and is continuous for $c_i \neq 0$. Since 
\[
\frac{\partial F_i}{\partial x^k} = \delta_k^i,
\]
we have shown that $F$ has continuous partial derivatives at every point where none of the $c_i$ are zero, and $dF$ is surjective at each of these points. Denote by $Q$ the subset of $P$ where none of the $c_i$ coordinates are zero. Then $Q$ is open, and $(u_1 \circ F)|_Q$ is a $C^1$ map for which $0$ is a regular value. By the implicit function theorem, we conclude that $(u_1 \circ F)|_Q^{-1}(0)$ is a $C^1$ hypersurface in $P$, in particular it has measure zero. Finally, note that if $v \in P$ and $F(v) \in u_1^{-1}(0)$, then
\[
v \in (u_1 \circ F)|_Q^{-1}(0) \cup (P \setminus Q),
\]
which has measure zero.
\end{proof}

\begin{remark} \label{localisation argument}
    In the discussion above, we studied solutions to Maxwell's equations through the corresponding equations for a potential in the Lorenz gauge. But in general, since the topology of $\mathbb{T}^3$ is non-trivial, we can only find $\omega$ such that $F = d\omega$ locally. We need to justify why our results are still meaningful when $F$ is not globally exact. Since $M = (0,\infty) \times \mathbb{T}^3$ is homotopic to $\mathbb{T}^3$, then the de Rham cohomology groups of $M$ are isomorphic to those of $\mathbb{T}^3$. In particular,
    \[
    H^2(M) \cong \mathbb{R}^3.
    \]
    Moreover, $H^2(M)$ is generated by the classes $[dx^1 \wedge \, dx^2]$, $[dx^1 \wedge \, dx^3]$ and $[dx^2 \wedge \, dx^3]$. Therefore, if $dF = 0$, there are real constants $a$, $b$ and $c$, and a one form $\omega$ such that
    \[
    F = a\, dx^1 \wedge dx^2 + b\, dx^1 \wedge dx^3 + c\, dx^2 \wedge dx^3 + d \omega.
    \]
    Note that for $i \neq j$,
    \[
    \diver (dx^i \wedge dx^j) = 0.
    \]
    Hence $F$ is a solution to Maxwell's equations if and only if $d\omega$ is also a solution. Furthermore, similarly as for the Lorenz gauge, we may assume $\diver \omega = 0$, so that both Theorem~\ref{omega expansion theorem} and Theorem~\ref{generic blow up} apply to $\omega$. We can then use the asymptotics for $\omega$ to find the asymptotics for $F$. In particular, if $T$ is the stress-energy tensor associated with $F$ and $\gamma$ is a past directed unit timelike geodesic, the leading order behavior of the energy density is
    \[
    T(\dot{\gamma}, \dot{\gamma}) = \frac{\delta c_2^2 + c_3^2}{4\pi} (4p_1^2 u_1^2 + (\partial_2 u_3 - \partial_3 u_2 + c)^2) t^{-(2p_2 + 4p_3)} + \cdots.
    \]
    Thus the only difference with the situation in Theorem~\ref{generic blow up} is the appearance of the constant $c$, which does not affect the result. 
\end{remark}

It is also of interest to look into the components $T_{\mu}^{\mu}$ (no summation). We have
\[
\begin{split}
    -T_t^t &\sim T_2^2 = \frac{1}{8\pi}(4p_1^2 u_1^2 + (\partial_2 u_3 - \partial_3 u_2)^2) t^{-2(p_2 + p_3)} + \alpha^2 t^{2-6p_3} + \cdots,\\
    -T_1^1 &\sim T_3^3 = \frac{1}{8\pi}(4p_1^2 u_1^2 + (\partial_2 u_3 - \partial_3 u_2)^2) t^{-2(p_2 + p_3)} - \alpha^2 t^{2-6p_3} + \cdots;
\end{split}
\]
where $\alpha$ is some function which is polynomial in the asymptotic data. Depending on the particular values of the $p_i$, one of the two powers of $t$ above will be dominant. Since we have no information about $\alpha$, using our results about $u_1$, we can only conclude that for $p_2 + p_3 > 3p_3 - 1$, the $T_{\mu}^{\mu}$ blow up like $t^{-2(p_2 + p_3)}$ as $t \to 0$ at almost every $x \in \mathbb{T}^3$. In \cite{goorjian}, the author obtains blow up of $T_{\mu}^{\mu}$ at the same rate for some particular plane wave solutions of Maxwell's equations. Even though our results only confirm that the blow up observed in \cite{goorjian} is generic for $p_2 + p_3 > 3p_3 - 1$, they definitely suggest that blow up at a rate of at least $t^{-2(p_2 + p_3)}$ as $t \to 0$ is generic for every non-flat Kasner spacetime.

\section*{Acknowledgements}

The author would like to thank Hans Ringström for the helpful discussions and the comments on the manuscript, and Oliver Petersen for the helpful discussions. This research was funded by the Swedish Research Council, dnr. 2017-03863 and 2022-03053.

\printbibliography

@article{hans,
  author = {H. Ringström},
  
  title = {Linear systems of wave equations on cosmological backgrounds with convergent asymptotics},
  
  publisher = {Société Mathématique de France},
  
  year = {2020},
  
  volume = {420},

  journal = {Astérisque}
}

@article{goorjian,
  title = {Electromagnetic plane-wave perturbations in Kasner cosmologies},
  
  author = {P. Goorjian},
  
  journal = {Phys. Rev. D},
  
  volume = {12},
  
  issue = {10},
  
  pages = {2978--2983},
  
  numpages = {0},
  
  year = {1975},
  
  publisher = {American Physical Society},
}

@book{cauchyproblem,
    author = {H. Ringström},
    title = {The Cauchy Problem in General Relativity},
    year = {2009},
    publisher = {European Mathematical Socierty}
}

@article{hans2019,
  
	year = 2019,
  
	publisher = {Springer Science and Business Media {LLC}},
  
	volume = {372},
  
	number = {2},
  
	pages = {599--656},
  
	author = {H. Ringström},
  
	title = {A Unified Approach to the Klein{\textendash}Gordon Equation on Bianchi Backgrounds},
  
	journal = {Communications in Mathematical Physics}
}

@article{alho,

author = {A. Alho and G. Fournodavlos and A. T. Franzen},

title = {The wave equation near flat Friedmann–Lemaître–Robertson–Walker and Kasner Big Bang singularities},

journal = {Journal of Hyperbolic Differential Equations},

volume = {16},
number = {02},

pages = {379-400},

year = {2019},

}

@article{oliver,

	year = 2016,
  
	publisher = {Springer Science and Business Media {LLC}
},
  
	volume = {19},
  
	number = {4},
  
	author = {O. {Lindblad Petersen}},
  
	title = {The Mode Solution of the Wave Equation in Kasner Spacetimes and Redshift},
  
	journal = {Mathematical Physics, Analysis and Geometry}
}

@book{wald,
    author = {R. M. Wald},
    title = {General Relativity},
    year = {1984},
    publisher = {The University of Chicago Press}
}

@misc{hans2021wave,
      title={Wave equations on silent big bang backgrounds}, 
      author={H. Ringström},
      year={2021},
      eprint={2101.04939},
      archivePrefix={arXiv},
      note = {Preprint}
}

@article{bkl1,
  author  = {V. A. Belinskiĭ and I. M. Khalatnikov and E. M. Lifshitz},
  title   = {Oscillatory Approach to a Singular Point in Relativistic Cosmology},
  journal = {Soviet Physics Uspekhi},
  year    = {1971},
  volume  = {13},
  number  = {6},
  pages   = {745}
}

@article{bkl2,
  author    = {V. A. Belinskiĭ and I. M. Khalatnikov and E. M. Lifshitz},
  title     = {A general solution of the Einstein equations with a time singularity},
  journal   = {Advances in Physics},
  year      = {1982},
  volume    = {31},
  number    = {6},
  pages     = {639-667},
  publisher = {Taylor \& Francis}
}

@article{allen-rendall,
author = {P. T. Allen and A. D. Rendall},
title = {Asymptotics of Linearized Cosmological Perturbations},
journal = {Journal of Hyperbolic Differential Equations},
volume = {07},
number = {02},
pages = {255-277},
year = {2010}
}

@article{andersson-rendall,
	year = 2001,
  
	publisher = {Springer Science and Business Media {LLC}
},
  
	volume = {218},
  
	number = {3},
  
	pages = {479--511},
  
	author = {L. Andersson and A. D. Rendall},
  
	title = {Quiescent Cosmological Singularities},
  
	journal = {Communications in Mathematical Physics}
}

@article{fournodavlos-rodnianski-speck,
      title={Stable Big Bang formation for Einstein's equations: The complete sub-critical regime}, 
      author={G. Fournodavlos and I. Rodnianski and J. Speck},
      year={2023},
      journal={J. Amer. Math. Soc.},
      volume={36},
      pages = {827--916},
      publisher = {American Mathematical Society}
}

@article{heinzle-uggla-rohr,
author = {J. Heinzle and C. Uggla and N. Röhr},
year = {2009},
pages = {293--407},
title = {The cosmological billiard attractor},
volume = {13},
number = {2},
journal = {Advances in Theoretical and Mathematical Physics},
}

@proceedings{proceedings,
title = {Proc. Centre Math. Appl.},
volume = {27},
year = {1991},
publisher = {Australian National University},
venue = {Canberra, Australia}
}

@inproceedings{chrusciel,
  title={On uniqueness in the large of solutions of Einstein's equations (``strong cosmic censorship")},
  author={P. T. Chruściel},
  year={1991},
  crossref = {proceedings}
}

@incollection{penrose,
    title = {Singularities and time--asymmetry},
    author = {R. Penrose},
    booktitle = {General Relativity: an Einstein Centenary Survey},
    editor = {S. Hawking and W. Israel},
    publisher = {Cambridge University Press},
    year = 1979,
    address = {Cambridge}
}

@article{radermacher,
    title = {Strong Cosmic Censorship in Orthogonal Bianchi Class B Perfect Fluids and Vacuum Models},
    author = {K. Radermacher},
    journal = {Ann. Henri Poincaré},
    volume = {20},
    year = {2019},
    pages = {689--796}
}

@article{hans-scc,
  title={Strong cosmic censorship in $T^3$-Gowdy spacetimes},
  author={H. Ringström},
  journal={Annals of mathematics},
  pages={1181--1240},
  year={2009},
}

@article{chrusciel-isenberg-moncrief,
  title={Strong cosmic censorship in polarised Gowdy spacetimes},
  author={P. T. Chruściel and J. Isenberg and V. Moncrief},
  journal={Classical and Quantum Gravity},
  volume={7},
  number={10},
  pages={1671},
  year={1990},
  publisher={IOP Publishing}
}

@article{KL,
    author = {E. M. Lifshitz and I. M. Khalatnikov},
    title = {Investigations in relativistic cosmology},
    journal = {Advances in Physics},
    year = {1963},
    volume = {12},
    number = {46},
    pages = {185--249}
}

@article{BK-electromagnetic,
    author = {V. A. Belinskiĭ and I. M. Khalatnikov},
    title = {On the influence of the spinor and electromagnetic fields on the cosmological singularity character},
    journal = {Rend. Sem. Mat. Univ. Politech. Torino},
    volume = {35},
    pages = {159--180},
    year = {1977}
}

@article{BKeffect,
  title={Effect of scalar and vector fields on the nature of the cosmological singularity},
  author={V. A. Belinskiǐ and I. M. Khalatnikov},
  journal={Soviet Journal of Experimental and Theoretical Physics},
  volume={36},
  pages={591},
  year={1973}
}

@book{oneill,
    author = {B. O'Neill},
    title = {Semi-Riemannian geometry: With applications to relativity},
    publisher = {Academic Press},
    year = {1983}
}

@misc{quiescent,
      title={Formation of quiescent big bang singularities}, 
      author={H. Oude Groeniger and O. Petersen and H. Ringström},
      year={2023},
      eprint={2309.11370},
      archivePrefix={arXiv},
      note = {Preprint}
}

@article{fournodavlos-luk,
    author = {G. Fournodavlos and J. Luk},
    title = {Asymptotically Kasner--like singularities},
    journal = {American Journal of Mathematics},
    year = {2023},
    volume = {145},
    number = {4},
    pages = {1183--1272}
}

@book{petersen,
    author = {P. Petersen},
    title = {Riemannian Geometry. Second Edition},
    publisher = {Springer},
    year = {2006},
}

\end{document}